\documentclass[12pt]{amsart}
\allowdisplaybreaks
\usepackage{amsmath}
\usepackage{amsfonts}
\usepackage{amssymb}
\usepackage{latexsym}
\usepackage{graphicx}
\usepackage{enumerate}
\usepackage{multirow}
\usepackage{subfigure}

\numberwithin{equation}{section}

\newtheorem{theorem}{Theorem}
\newtheorem{proposition}[theorem]{Proposition}
\newtheorem{lemma}[theorem]{Lemma}
\newtheorem{cor}[theorem]{Corollary}
\newtheorem{remark}{Remark}
\newtheorem{definition}{Definition}

\usepackage{mathtools}

\makeatletter
\DeclareRobustCommand\widecheck[1]{{\mathpalette\@widecheck{#1}}}
\def\@widecheck#1#2{%
    \setbox\z@\hbox{\m@th$#1#2$}%
    \setbox\tw@\hbox{\m@th$#1%
       \widehat{%
          \vrule\@width\z@\@height\ht\z@
          \vrule\@height\z@\@width\wd\z@}$}%
    \dp\tw@-\ht\z@
    \@tempdima\ht\z@ \advance\@tempdima2\ht\tw@ \divide\@tempdima\thr@@
    \setbox\tw@\hbox{%
       \raise\@tempdima\hbox{\scalebox{1}[-1]{\lower\@tempdima\box
\tw@}}}%
    {\ooalign{\box\tw@ \cr \box\z@}}}
\makeatother

\newcommand{\ep}{\epsilon}

\DeclareMathOperator{\sech}{sech}

\DeclareMathOperator{\diag}{diag}
\DeclareMathOperator{\sgn}{sgn}

\DeclareMathOperator{\spn}{span}
\DeclareMathOperator{\Lip}{Lip}

\newcommand{\per}{{\text{per}}}
\newcommand{\ds}{\displaystyle}

\newcommand{\be}{\begin{equation}}
\newcommand{\ee}{\end{equation}}
\newcommand{\bes}{\begin{equation*}}
\newcommand{\ees}{\end{equation*}}
\newcommand{\mand}{\quad \text{and}\quad}

\newcommand{\R}{{\bf{R}}}
\newcommand{\C}{{\bf{C}}}

\newcommand{\Z}{{\bf{Z}}}
\newcommand{\N}{{\bf{N}}}

\newcommand{\h}{{\bf{h}}}
\renewcommand{\v}{{\bf{v}}}
\newcommand{\ub}{{\bf{u}}}
\newcommand{\f}{{\bf{f}}}
\newcommand{\jb}{{\bf{j}}}
\newcommand{\ib}{{\bf{i}}}
\newcommand{\pb}{{\bf{p}}}
\newcommand{\xb}{{\bf{x}}}
\newcommand{\yb}{{\bf{y}}}
\newcommand{\nb}{{\bf{n}}}
\newcommand{\gnb}{{\grave{\bf{n}}}}
\newcommand{\ga}{{\grave{a}}}
\newcommand{\tlambda}{\tilde{\lambda}}
\newcommand{\tmu}{\tilde{\mu}}
\newcommand{\tB}{\tilde{\B}}
\newcommand{\tbeta}{\tilde{\beta}}

\newcommand{\tvarpi}{\tilde{\varpi}}
\newcommand{\tzeta}{\tilde{\zeta}}
\newcommand{\txi}{\tilde{\xi}}
\newcommand{\tgamma}{\tilde{\gamma}}

\newcommand{\Fo}{{\mathfrak{F}}}

\newcommand{\X}{{\mathcal{X}}}

\renewcommand{\O}{{\mathcal{O}}}
\newcommand{\A}{{\mathcal{A}}}
\newcommand{\B}{{\mathcal{T}}}
\renewcommand{\P}{{\mathcal{P}}}
\newcommand{\thetab}{{\boldsymbol \theta}}
\newcommand{\gthetab}{\grave{\thetab}}
\newcommand{\varphib}{{\boldsymbol \varphi}}
\newcommand{\sigmab}{{\boldsymbol \sigma}}
\newcommand{\phib}{{\boldsymbol \phi}}
\newcommand{\psib}{{\boldsymbol \psi}}
\newcommand{\Psib}{{\boldsymbol \Psi}}
\newcommand{\etab}{{\boldsymbol \eta}}
\newcommand{\nub}{{\boldsymbol \nu}}
\newcommand{\upsilonb}{{\boldsymbol \upsilonb}}

\renewcommand{\H}{{\mathcal{H}}}

\renewcommand{\tilde}{\widetilde}
\renewcommand{\hat}{\widehat}
\renewcommand{\check}{\widecheck}

\newcommand{\ball}{\mathcal{B}}
\newcommand{\norm}[1]{\|#1\|}
\newcommand{\normlarge}[1]{\left\|#1\right\|}
\newcommand{\ortho}{\mathcal{Z}}
\newcommand{\Y}{\mathcal{Y}}
\newcommand{\bF}{{\bf{F}}}

\newcommand{\bdop}{{\bf{B}}}
\newcommand{\Phib}{\boldsymbol \Phi}
\newcommand{\ip}[2]{\left\langle#1,#2\right\rangle}
\usepackage{color}

\newcommand{\cep}{c_{\ep}}
\renewcommand{\O}{{\mathcal{O}}}
\newcommand{\bG}{{\bf{G}}}

\title[Diatomic FPUT solitary waves with optical ripples]{Exact diatomic Fermi-Pasta-Ulam-Tsingou solitary waves with optical band ripples at infinity}
\author{Timothy E. Faver}\address{Department of Mathematics, Drexel University, 3141 Chestnut St, Philadelphia, PA 19104}
\author{J. Douglas Wright}\address{Department of Mathematics, Drexel University, 3141 Chestnut St, Philadelphia, PA 19104}
\begin{document}
\thanks{The authors would like acknowledge the National Science Foundation which has generously supported
the work through grants DMS-1105635 and DMS-1511488. A debt of gratitude also goes to Nsoki Mavinga and the  Department of Mathematics and Statistics at Swarthmore College who hosted JDW during much of the research which went into this document. 
Additionally, they would like to thank Aaron Hoffman for a huge number of helpful comments and insights.
Finally, they dedicate this article to the memory of Malinda Gilchrist, graduate coordinator of Drexel's math department, colleague, navigator and friend.
}

\begin{abstract}
We study the existence of solitary  waves in a diatomic Fermi-Pasta-Ulam-Tsingou (FPUT) lattice. 
 For monatomic FPUT the 
 traveling wave equations are a regular perturbation of the Korteweg-de Vries (KdV) equation's
 but, surprisingly, we find that for the diatomic lattice the traveling wave equations are a {\it singular} perturbation of KdV's.
Using a method first developed by Beale to study traveling solutions for capillary-gravity waves we 
demonstrate that for wave speeds in slight excess of the lattice's speed of sound
there exists nontrivial traveling wave solutions which are the superposition an exponentially localized solitary wave  and a periodic wave whose amplitude is extremely small.  That is to say, we construct nanopteron solutions. The presence of the periodic wave is an essential part of the analysis and is connected to the fact that
linear diatomic lattices have optical band waves with any possible phase speed.
\end{abstract}

\maketitle

We consider the problem of traveling waves in a diatomic Fermi-Pasta-Ulam-Tsingou (FPUT) lattice. The physical situation is this:
suppose that infinitely many particles are arranged on a line. The mass of the $j$th (where $j \in \Z$) particle
is 
$$
m_j=\begin{cases} m_1 &\text{ when $j$ is odd} \\ m_2 &\text{when $j$ is even.} \end{cases}
$$ 
Without loss of generality we assume that $m_1 > m_2>0$. The position of the $j$th particle at time $\bar{t}$ is $\bar{y}_j(\bar{t})$.
Suppose that each mass  is connected to its two nearest neighbors by a spring and furthermore assume that each spring is identical to every other spring in the sense that the force exerted by said spring when stretched by an amount $r$ from its equilibrium length $l_s$ is given by
$$
F_s( r ) := - k_s r - b_s r^2
$$
where $k_s>0$ and $b_s \ne 0$ are specified constants. 
Such a system is called a ``diatomic lattice" or ``dimer."
Newton's law gives the equations of motion
for the system:
\be\label{nl}
m_j  {d^{2} \bar{y}_j \over d \bar{t}^2} = - k_s \bar{s}_{j-1} - b_s\bar{s}_{j-1} + k_s\bar{s}_j + b_s \bar{s}^2_j
\ee
where
$
\bar{s}_j:=\bar{y}_{j+1}-\bar{y}_j - l_s.
$

In the setting where $m_1=m_2$ it is well-known that there exist\footnote{This result is true for much more general (but still spatially homegeneous) forms of the spring force than the one we have here.}  localized traveling wave solutions of \eqref{nl}, see the seminal articles of Friesecke \& Wattis \cite{friesecke-wattis} and 
Friesecke \& Pego \cite{friesecke-pego1}.
Here we are interested in extending this result to the diatomic case where $m_1>m_2$. There has been quite a bit of interest in the propagation of
waves through polyatomic FPUT lattices. Such systems represent a paradigm for the evolution of waves through heterogeneous and nonlinear granular media (see \cite{brillouin} and \cite{kevrekidis} for an overview).
There are several existence proofs for traveling spatially periodic waves for polyatomic problems \cite{betti-pelinovsky} \cite{qin}, a  number of semi-rigourous  asymptotics for solitary wave solutions in various contexts \cite{jayaprakash}, as well as both formal and rigorous results which state that polyatomic FPUT with periodic material coefficients is well-approximated by the soliton bearing Korteweg-de Vries (KdV) equation over very long time scales \cite{pnevmatikos} \cite{chirilus-bruckner-etal} \cite{gaison-etal}. While all of this previous work strongly suggests that localized traveling waves for polyatomic FPUT will exist, the question of whether a truly localized traveling wave akin to those developed in \cite{friesecke-pego1} remains open.
In this article we demonstrate that the answer to this question---at least for waves  which travel at a speed just a bit larger than the speed of sound---is 
``sort of."

As it happens, the existence problem is inescapably singular. This is particularly surprising because
the existence proof of small amplitude solitary waves for monatomic FPUT in \cite{friesecke-pego1} 
goes through using regular perturbation methods.
In that article,
the equation for the solitary wave's profile $\phi$ is shown to be equivalent to
 $\phi = p_c (\phi^2)$ where $p_c$ is a Fourier multiplier operator and $c$ is the speed of propagation. Making the ``long wave scaling" $\phi(x) = \ep^2 \Phi(\ep x)$ and $c=c_{sound} + \ep^2$ yields $\Phi = P_\ep (\Phi^2)$. The hinge on which their result turns is the fact that the operator $P_\ep$ converges in the operator norm to $(1-\partial_X^2)^{-1}$ as $\ep \to 0^+$.
Which means that the traveling wave equation at $\ep =0$ can be rewritten as $\Phi'' - \Phi + \Phi^2 =0$. This is the traveling wave equation for KdV and has $\sech^2$ type solutions.
Moreover, using classical results from quantum mechanics, they show that the linearization of the equation $\Phi = P_0 (\Phi^2)$ at the KdV solitary
wave results in an operator which is  invertible on even functions. This, with the uniform convergence of $P_\ep$, allows them to extend the wave's existence to $\ep > 0$  using a quantitative inverse function theorem.

This process goes awry in the diatomic setting. In this case the dispersion relation for the linearization of \eqref{nl} has two parts. The ``acoustic" band, which is more or less just like the dispersion relation for the monatomic problem, and the ``optical" band\footnote{The language ``acoustic" and ``optical" bands is taken from Brillouin's 
foundational text on wave propagation in periodic media \cite{brillouin}. This book contains a detailed discussion
of  waves in linear diatomic lattices and in particular contains an excellent treatment of the dispersion relation.} which does not exist in the monatomic problem at all.
Roughly speaking, we are able to decompose our problem into a pair of equations, 
one for the acoustic part and another for the optical part. The analysis for the acoustic part goes forward
along lines much like in the  monatomic case of \cite{friesecke-pego1}; it limits to a KdV  traveling wave equation.
On the other hand,  the equation for the optical part is classically singularly perturbed in the sense that the highest derivative 
of the unknown is multiplied by the small parameter $\ep$.

The possible outcomes for singularly perturbed problems like this is pretty vast, of course. Many of the approaches for sussing out the consequences are either geometric or dynamical in nature. Our problem is nonlocal and as such it is not obvious to us how to use, say, geometric singular perturbation theory (as in  \cite{alexander-etal} for instance) or fast-slow averaging ({\it e.g.}  \cite{benbachir-etal}) to our setting. And so we turn to the functional analytic approach developed by Beale to prove the existence of solitary capillary-gravity waves in his staggering article \cite{beale2}.

Here is what we discover:
\begin{theorem}\label{main result nontech} Suppose that $m_1>m_2>0$. For wavespeeds $c$ sufficiently close to, but larger than, the speed of sound of the lattice, $c_{sound} := \sqrt{2k_s/(m_1+m_2)}$, there exist traveling wave solutions of the diatomic FPUT problem which are  the superposition of two pieces. 
One piece is a nonzero exponentially localized function that is a small perturbation of a $\sech^2$ profile which, in turn, solves a KdV traveling wave equation. This whole localized piece has amplitude roughly proportional to $\left(c-c_{\textrm{sound}}\right)$
and has wavelength roughly proportional to $\left(c-c_{\textrm{sound}}\right)^{-1/2}$.
The other piece is a periodic function, called a ``ripple." The frequency of the ripple is $\O(1)$ when compared to $\left(c-c_{\textrm{sound}}\right)$. Its amplitude is small beyond all orders of $\left(c-c_{\textrm{sound}}\right)$.  
\end{theorem}
As we shall see,  the periodic part 
is fundamentally tied to the optical branch of the dispersion relation.
Moreover we expect that the periodic part is exponentially small in $\left(c-c_{\textrm{sound}}\right)$. 
Solutions of this type---a localized piece plus an extremely small oscillatory part\footnote{Or, equivalently, a heteroclinic connection between small amplitude periodic orbits.}---are sometimes called {\it nanopterons} \cite{boyd}.

It is because the solutions we discover do not converge to zero at spatial infinity that we 
were cagey about our answer to the existence question earlier; our result raises as many questions as it answers. Chief of these is whether or not the ripple at infinity is genuine or merely a technical byproduct of our proof. After all, we do not provide lower bounds on its size, only upper bounds; perhaps the amplitude is zero! While we do not have the right sort of estimates at this time to answer this question either way we point out that Sun, in \cite{sun}, showed that the ripple for the capillary-gravity waves studied in \cite{beale2} was in fact non-zero. 
And so we conjecture that the same happens here, at least for almost all wave speeds.

This article is structured in the following manner. 
\begin{itemize}
\item
In the next section we nondimensionalize \eqref{nl} and rewrite the resulting system in terms of the relative displacements.
\item In Section \ref{TWE} we make the traveling wave ansatz and get the traveling wave equations. We then diagonalize the resulting system using Fourier methods. 
It is during the diagonalization that the structure of the branches of the dispersion relation becomes apparent. 
Then we make a useful ``long wave" rescaling.
It is during this part that  the singular nature of the problem comes into sight.
\item In Section \ref{LWL} we analyze the rescaled system in the limit where $c = c_{sound}$; we find that the problem in this case reduces to a single KdV traveling wave equation.
\item In Section \ref{PS} we construct exact traveling wave solutions which are spatially periodic. These will ultimately be the ripples. A major difference between our results and the extant existence results for periodic traveling waves in polyatomic FPUT (\cite{qin},\cite{betti-pelinovsky}) is that we prove 
estimates on their size and frequency which are uniform in the speed $c$. This is done using a Crandall-Rabinowitz-Zeidler bifurcation analysis \cite{crandall-rabinowitz} \cite{zeidler1}.
\item In Section \ref{N} we make what we call ``Beale's ansatz." That is,  we assume the solution is  the superposition of (a) the KdV solitary wave profile from Section \ref{LWL}, (b)  a periodic solution from Section \ref{PS} with unknown amplitude and (c) a small, localized remainder. We then derive  equations for the remainder and the amplitude of the periodic part. This derivation can be viewed as Liapunov-Schmidt decomposition, albeit a somewhat atypical one. 
\item In Section \ref{EU}  we state the main estimates we need and then, given those estimates, prove our main results  using a modified contraction mapping argument. 
Specifically we prove Theorem \ref{main result} and Corollary \ref{main result descale}, which are the technical versions of Theorem~\ref{main result nontech}.
\item  Sections \ref{BE}, \ref{NE} and \ref{proofs} contain the proof of the main estimates; these are the technical heart of the paper. 
\item
Finally Section \ref{CR} presents some comments on our results, avenues for further investigation and concluding remarks.
\end{itemize}
\section{Nondimensionalization and the equations for relative displacements.}
We can simplify \eqref{nl} somewhat  by putting $\bar{y}_j = \bar{x}_j - j l_s$. Then \eqref{nl} is equivalent to
\be\label{nl2}
m_j  {d^{2} \bar{x}_j \over d \bar{t}^2} = - k_s \bar{r}_{j-1} - b_s\bar{r}_{j-1} + k_s\bar{r}_j + b_s \bar{r}^2_j
\ee
where
$
\bar{r}_j:=\bar{x}_{j+1}-\bar{x}_j .
$
Note  that the system is in equilibrium when $\bar{x}_j = 0$ for all $j$. 

Next we nondimensionalize by taking
$
\bar{x}_{j}(\bar{t}) = a_1 x_j(a_2\bar{t}) 
$
where $a_1,a_2$ are nonzero constants.
This converts \eqref{nl2} to
\bes
m_j  a_2^2 {d^{2} {x}_j \over d {t}^2} = - k_s {r}_{j-1} - a_1 b_s{r}_{j-1} + k_s{r}_j + a_1 b_s {r}^2_j
\quad \text{where} \quad
r_j = x_{j+1} - x_j.
\ees
Note that here $t = a_2 \bar{t}$. 

Selecting $a_1$ and $a_2$ such that
$
m_1 a_2^2 = k_s$ and $a_1 b_s = k 
$
yields
\be\begin{split}\label{dimer newton}
\ddot{x}_j  &= -r_{j-1} - r_{j-1}^2 + r_j + r_{j}^2\quad \text{when $j$ is odd}\\
{1 \over w} \ddot{x}_j  &= -r_{j-1} - r_{j-1}^2 + r_j + r_{j}^2\quad \text{when $j$ is even}.
\end{split}\ee
In the above,
\bes\label{this is w}
w:={m_1 \over m_2}>1\ees 
because $m_1>m_2$.

It is both traditional and technically advantageous to express the equations of motion for lattices in terms of the relative displacements, $r_j$, instead of in the displacements from equilibrium, $x_j$. We find that
\be\begin{split}\label{r eqn}
\ddot{r}_j&=-(1+w)(r_j+r_j^2) + w(r_{j+1}+r_{j+1}^2) + (r_{j-1}+r_{j-1}^2)\quad \text{when $j$ is odd}\\
\ddot{r}_j&=-(1+w)(r_j+r_j^2) + (r_{j+1}+r_{j+1}^2) + w(r_{j-1}+r_{j-1}^2)\quad \text{when $j$ is even.}\\
\end{split}\ee


\section{Derivation of the traveling wave equations}\label{TWE}
We are interested in traveling wave solutions
and so  we make the  ansatz
\be\label{ansatz}
r_j(t) = \begin{cases}p_1(j-ct) &\text{ when $j$ is odd} \\ p_2(j-ct) &\text{when $j$ is even.} \end{cases}
\ee 
Here $c \in \R$ is the wave speed and  $p_1, p_2:\R \to \R$.
Putting this into \eqref{r eqn} gives us the following advance-delay-differential system of equations for $p_1$ and $p_2$:
\be\begin{split}\label{p eqn}
c^2{p}_1''&=-(1+w)(p_1+p_1^2) + wS^{1}(p_{2}+p_{2}^2) + S^{-1}(p_2+p_{2}^2)\\
c^2p''_2&=-(1+w)(p_2+p_2^2) + S^1(p_1+p_1^2) + wS^{-1}(p_1 + p_1^2).
\end{split}\ee
Above, $S^d$ is the ``shift by $d$" operator. Specifically:
$$
S^{d} f(\cdot):= f(\cdot+d).
$$

If we let
\bes
L:=\left[ \begin{array}{cc} 
1+w & -(w S^1 + S^{-1})\\
-(w S^{-1} + S^1) & 1+w
\end{array}
\right]
\ees
then we can compress\footnote{Note the for $2$-vectors $\xb$ and $\yb$ we use the notational convention that $\xb.\yb$ is component-wise multiplication. Likewise $\xb^{.2}$ is component-wise squaring. } \eqref{p eqn} to 
\be\label{p eqn small}
c^2 \pb'' + L(\pb + \pb^{.2}) = 0.
\ee

\subsection{Diagonalization of the linear part}
We can diagonalize \eqref{p eqn small} using Fourier analysis and the first step is to compute the action of $L$ on complex exponentials. 
We find that for any vector $\v \in \R^2$ and $k \in \R$ that
$$
L [ e^{ikx} \v] = [\tilde{L}(k) \v] e^{ikx}
$$
where
$$
\tilde{L}(k):=\left[ \begin{array}{cc} 1+w & -\tilde{\beta}(k) \\  -\tilde{\beta}(-k) & 1+w \end{array} \right]
\mand
\tilde{\beta}(k) := w e^{ik} + e^{-ik}.
$$

A routine calculation shows that the eigenvalues of $\tilde{L}(k)$ are given by
\be\label{this is lambda}
\tilde{\lambda}_\pm(k):=1+w\pm\tilde{\varrho}(k)\quad \text{where} \quad \tilde{\varrho}(k):= \sqrt{
(1 - w)^2 +  {4 w} \cos^2(k)
}.
\ee
The following lemma contains many of the properties of $\tilde{\lambda}_\pm(k)$ we will need (the proof is in Section \ref{proofs}).
\begin{lemma}\label{lambda lemma} The following hold for all $w > 1$.
\begin{enumerate}[(i)]
\item $\tilde{\lambda}_-(0) = 0$ and $\tilde{\lambda}_+(0) = 2 + 2w.$
\item There exists $\tau_0 > 0$ such that $\tilde{\lambda}_\pm(z)$ are uniformly bounded complex analytic functions\footnote{We extend $\tlambda_\pm(k)$ and $\tilde{\varrho}(k)$ to functions of  $z = k + i \tau \in \C$ in a simple way by using the extension of cosine to complex inputs.} in the closed strip $\overline{\Sigma}_{\tau_0}:=\left\{ z \in \C : |\Im z| \le \tau_0 \right\}$.
\item $\tilde{\lambda}_\pm(z)$ are even and $\tilde{\lambda}_\pm(z+\pi) = \tilde{\lambda}_\pm(z)$ for all $z \in \overline{\Sigma}_{\tau_0}$.
\item For all $k\in \R$ we have
\be \label{lambda bounds} 
0\le \tilde{\lambda}_-(k) \le 2<2w \le \tilde{\lambda}_+(k)  \le 2 + 2 w. 
\ee
\item For all $k \in \R$ we have \be\label{lambda derivative bounds}
|\tilde{\lambda}'_\pm(k)|\le2 \mand |{\tlambda'_\pm}(k)| \le 2c_w^2|k|
\ee
where
$$
c_w:=\sqrt{{1 \over 2} \lambda_-''(0)}=\sqrt{{2 w \over 1+w}}=\text{``the (nondimensionalized) speed of sound."}
$$
Additionally $c_w>1$.
\item There exists $c_- \in (0,1)$ and $l_0>0$ such that for all $c \ge c_-$
there exists a unique nonnegative  $k_c$ for which \be\label{this is kc}
c^2 k_c^2 - \tlambda_+(k_c) =0.\ee
Moreover \be \label{size of kc} k_c \in [\sqrt{2w}/c,\sqrt{2+2w}/c]\ee and
\be\label{this is l0}
| 2 c^2 k_c - \tilde{\lambda}_+(k_c)|\ge l_0.
\ee Lastly, the map $c \mapsto k_c$ is $C^\infty$.

\end{enumerate}
\end{lemma}

\begin{remark} The eigenvalues $\tlambda_\pm(k)$ 
are tied to the dispersion relation for \eqref{r eqn}. 
To be precise, we have plane wave solutions for the linearization of \eqref{r eqn} at $r = 0$ 
of the form
$$
r_j(t) = \begin{cases}v_1 e^{i(k j - \omega t)} & \text{when $j$ is odd}\\
v_2 e^{i(k j - \omega t)}  & \text{when $j$ is even}
\end{cases}
$$
if and only if $\omega$ and $k$ satisfy the dispersion relation
\be\label{dr}
(\omega^2 - (1+w))^2 - \tbeta(k)\tbeta(-k)=0
\ee
and $(v_1,v_2)^t$ is a an appropriately chosen eigenvector of $\tilde{L}(k)$. (See, for instance,  \cite{brillouin}.)

The set of such $\omega$ and $k$ which meet \eqref{dr} has two connected components with $\omega\ge0$. (It is obviously even in $\omega$.) These are
\bes\begin{split}
\omega &= \sqrt{\tlambda_-(k)}, \quad \text{in which case we say $\omega$ is on the ``acoustic" branch of \eqref{dr},}\\
\text{and}\quad \omega &= \sqrt{\tlambda_+(k)}, \quad \text{in which case we say $\omega$ is on the ``optical" branch.}
\end{split}\ees

In the monatomic problem, the dispersion relation has but one branch and is akin to $\omega^2 -\sin^2(k)=0$, see \cite{friesecke-pego1}. Plotting $\tlambda_-(k)$ and $\sin^2(k)$ will show that the two functions are qualitatively much alike. The phase speed of the monatomic problem is maximum at $k = 0$; this maximum speed is called the speed of sound. The same is true for the phase speed
of waves associated to the acoustic branch. Specifically,  corollary to the second estimate  in \eqref{lambda derivative bounds} is the fact that that the phase speed, $\omega/k$, of plane waves associated to the acoustic branch is no bigger than $c_w$.

The solitary waves for the monatomic problem have  speeds which are strictly supersonic; the reason for this is discussed further below. In any case, the nonlinear waves move faster than the linear waves.\footnote{This fact is crucial to the proof of the stability of the solitary waves in \cite{friesecke-pego2} \cite{friesecke-pego3} \cite{friesecke-pego4} and \cite{hoffman-wayne2}.}
But in the dimer problem we consider there are optical branch linear waves with all possible phase speeds, as can be seen by plotting $\sqrt{\tlambda_+(k)}/k.$
The point is this: if we search for a localized acoustic wave that travels with a supersonic speed $c>c_w$ then there will necessarily be an optical branch linear wave whose phase is exactly $c$. 
 See Figure~\ref{sketch}.
Equating the
solitary wave's speed to the optical phase speed yields the relation 
$
c^2 k^2 - \tlambda_+(k) = 0.
$
As stated in Lemma \ref{lambda lemma}, there is  a unique nonnegative solution of this: $k = k_c$. \begin{figure}[t]
\centering
    \includegraphics[width=5in]{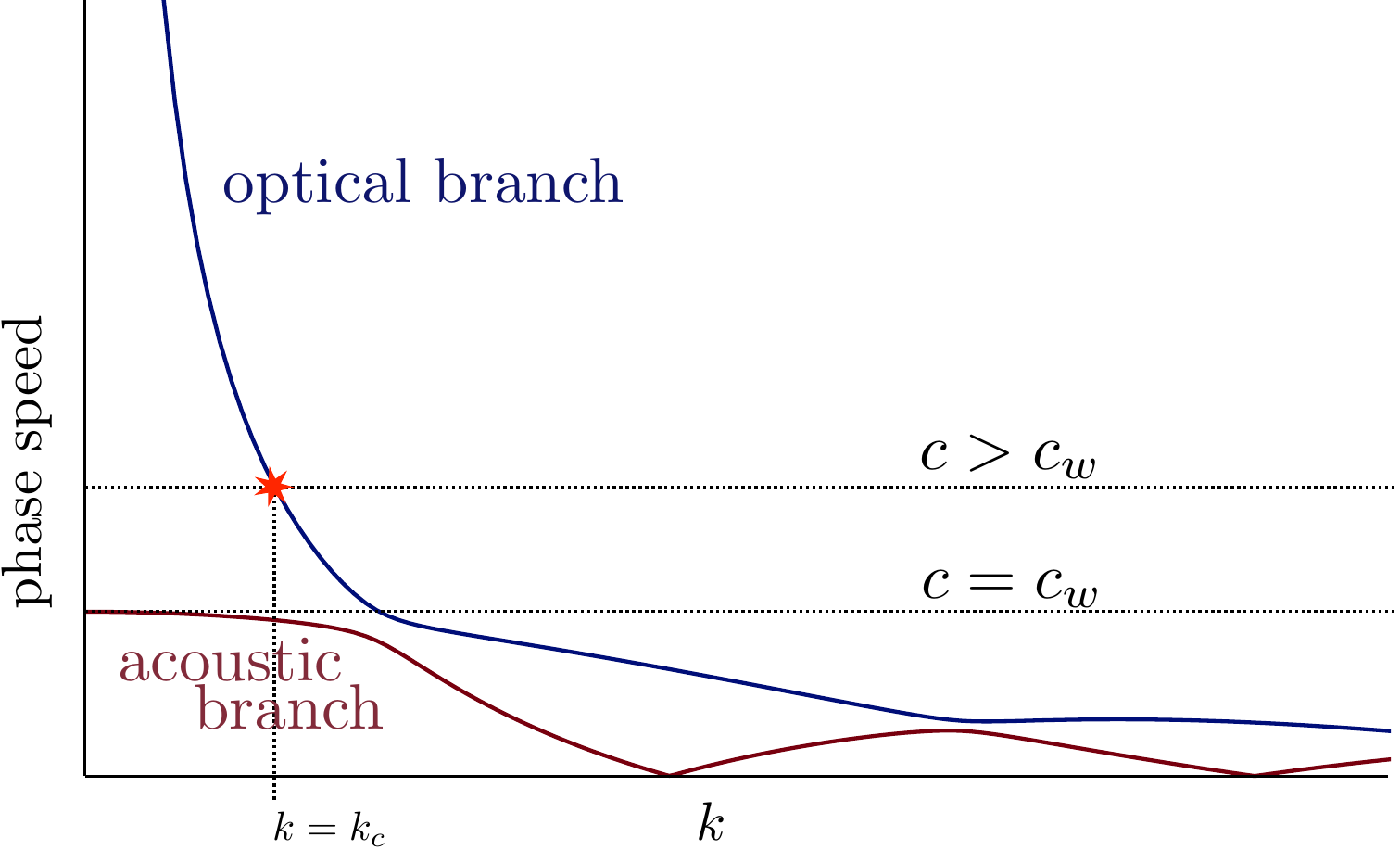}
 \caption{  \it Sketch of the phase speeds for the different branches of the dispersion relation. 
 \label{sketch}}
\end{figure}
In a very rough sense, then, we expect the localized acoustic wave with speed $c$
to excite a mode in the optical branch  with wavenumber $k_c$.
Making this intuition rigorous is, of course, a substantial part of our analysis.
\end{remark}

The inequalities in  \eqref{lambda bounds} imply  $\tilde{\lambda}_-(k) < \tilde{\lambda}_+(k)$ for all $k\in \R$. Thus $\tilde{L}(k)$ is diagonalizable for all $k$.
Towards this end, we compute  that the eigenvectors of $\tilde{L}(k)$ are scalar multiples of
$\ds
 \left[ \begin{array}{c} \tilde{\beta}(k) \\ \tilde{\varrho}(k) \end{array} \right]$
(for $\tilde{\lambda}_-(k)$) and 
 $
 \ds\left[ \begin{array}{c} \tilde{\beta}(k) \\ -\tilde{\varrho}(k) \end{array} \right] 
$ (for $\tilde{\lambda}_+(k)$). 

We can diagonalize $\tilde{L}$ by dropping these into a matrix. It will be advantageous to renormalize them first, though.
Let 
\bes\label{this is gamma}
\tilde{\gamma}(k) : =e^{-ik}+e^{ik} { \tilde{\varrho}(k) \over \tilde{\beta}(k)}
\ees
and put
\bes\label{this is T}
\tilde{J}_2(k):=\left[\begin{array}{cc}
\tilde{\gamma}(k) \tilde{\beta}(k) & \tilde{\gamma}(k) \tilde{\beta}(k) \\
\tilde{\gamma}(k) \tilde{\varrho}(k) & -\tilde{\gamma}(k) \tilde{\varrho}(k) 
\end{array}
\right].
\ees
Its inverse is
\bes\label{this is inv T}
\tilde{J}_1(k):=\tilde{J}_2^{-1}(k)=\left[\begin{array}{cc}
(2\tilde{\gamma}(k) \tilde{\beta}(k))^{-1} & (2\tilde{\gamma}(k) \tilde{\varrho}(k))^{-1} \\
(2\tilde{\gamma}(k) \tilde{\beta}(k))^{-1}&  -(2\tilde{\gamma}(k) \tilde{\varrho}(k))^{-1} 
\end{array}
\right].
\ees
Then we have
$$
\tilde{J}_1(k) \tilde{L}(k) \tilde{J_2}(k) = \tilde{\Lambda}(k) := \left[\begin{array}{cc} \tilde{\lambda}_-(k) & 0 \\ 0 & \tilde{\lambda}_+(k) \end{array}\right].
$$

The reason we choose to normalize the eigenvectors with $\tilde{\gamma}$ is obviously non-obvious. Here is what is special about $\tilde{\gamma}$:
\be\label{gamma is special}
\tilde{\gamma}(-k) \tilde{\beta}(-k) = \tilde{\gamma}(k) \tilde{\varrho}(k).
\ee
This property---easily checked---will imply a certain symmetry below, specifically in the proof of Lemma \ref{EO}.

A short computation indicates that neither $\tilde{\gamma}$ nor $\tilde{\beta}$ vanish when $k \in \R$.
In fact, we have:
\begin{cor} \label{J cor}
If $w > 1$ then there exists $\tau_1\in(0,\tau_0]$ such that $\tilde{\gamma}(z)$, $\tilde{\gamma}^{-1}(z)$, $\tbeta(z)$ and $\tilde{\beta}^{-1}(z)$ are  uniformly bounded analytic functions for $z \in \overline{\Sigma}_{\tau_1}$.
Moreover
$\tilde{J}_1(z)$ and $\tilde{J}_2(z)$ are
uniformly bounded matrix valued analytic functions for $z \in \overline{\Sigma}_{\tau_1}$.
\end{cor}
We do not provide a proof, as it is more or less immediate from the definitions and Lemma~\ref{lambda lemma}.
Note also that
\be\label{J zero}
\tilde{J}_1(0) = {1 \over 4(1+w)} \left[\begin{array}{cr} 1 & 1 \\ 1 & - 1
\end{array}
\right] \mand \tilde{J}_2(0) = {2(1+w)} \left[\begin{array}{cr} 1 & 1 \\ 1 & - 1
\end{array}
\right] .
\ee

Since $\tilde{J}_1$ and $\tilde{J}_2$ diagonalize $\tilde{L}$, we can use Fourier multiplier operators to diagonalize $L$.
We use following normalizations and notations for the Fourier transform and its inverse:
$$
\hat{f}(k):=\Fo[f](k):={1 \over 2 \pi} \int_\R f(x)e^{-ikx} dx \mand
\check{g}(x):=\Fo^{-1}[g](x) := \int_\R g(k)e^{ikx} dk.
$$
Likewise, we use the following normalizations and notations for the Fourier series of
a $2 P$-periodic function:
$$
\hat{f}(k):={1 \over 2P} \int_{-P}^P f(x) e^{-ik\pi x/P} dx \mand {f}(x) = \sum_{k \in \Z} \hat{f}(k) e^{ik \pi x/P}.
$$
We have used the same ``hat" notation for the Fourier transform and the coefficients of the Fourier series; context
will always make it clear which we mean. 

\begin{definition}\label{Fo mult definition} 
Suppose that we have $\tmu:\R \to \C$. The ``Fourier multiplier with symbol $\tmu$" is defined as follows.
\begin{enumerate}[(i)]
\item
If $f:\R \to \C$ has a well-defined Fourier transform then
\be\label{fmreal}
\mu f (x):= \int_\R e^{ikx} \tmu(k)  \hat{f}(k) dk.
\ee
\item If  $f:\R \to \C$ is $2P$-periodic then
\be\label{fmper}
\mu f(x) := \sum_{k \in \Z} e^{ik\pi x/P} \tmu(k \pi /P) \hat{f}(k).  
\ee
\item
If $f = f_1 + f_2$ where $f_1:\R \to \C$ has a well-defined Fourier transform and $f_2:\R \to \C$ is $2P$-periodic then 
we have $\mu f = \mu f_1 + \mu f_2$ where $\mu f_1$ is computed with \eqref{fmreal} and $\mu f_2$ is computed with \eqref{fmper}.
\end{enumerate}
\end{definition}
\begin{remark}
An alternate way to express  \eqref{fmper} goes as follows.
Since $f(x)$ is $2P$-periodic we know $f(x) = \phi( \omega x)$ for some 
$2 \pi$-periodic function $\phi(y)$ and $\omega = \pi/P$.
In this case $\mu f(x) = (\mu^\omega \phi)(\omega x)$ where
$\mu^\omega$ is a Fourier multiplier with symbol $\tmu^\omega(k) = \tmu(\omega k)$.
The nice thing here is that
$$
\mu^\omega \phi(y) = \sum_{k \in \Z} e^{iky} \tmu^\omega(k ) \hat{\phi}(k),
$$
a slightly less complicated formula than \eqref{fmper}.

Note also that if we put a $2P-$periodic function into the Fourier transform integral $\Fo$ then we can interpret the 
output
as a distribution. Specifically, it will be a superposition of delta-functions situated on $(\pi/P) \Z \subset \R$.
In this case we can apply the formula \eqref{fmreal}. The outcome of this
coincides exactly with \eqref{fmper}. In this way we can see that the formula in Part (iii) is the ``correct" way to apply
Fourier multipliers to sums of decaying and periodic functions.
\end{remark}


%

So let $\lambda_\pm$, $\varrho$, $\beta$, $\gamma$, $\Lambda$ and $J_n$ be the Fourier multiplier operators with symbols $\tlambda_\pm$, $\tilde{\varrho}$, $\tbeta$, $\tilde{\gamma}$, $\tilde{\Lambda}$ and $\tilde{J}_n$, respectively.
Put $\pb : = J_2 \h$. If $\pb$ solves \eqref{p eqn small} then we find that $\h$ solves
\be\label{h eqn small}
c^2\h'' + \Lambda \h + \Lambda B(\h,{\h}) = 0
\ee
with
$$
B(\h,\grave{\h}):=\left( \begin{array}{c} b_1(\h,\grave{\h}) \\ b_2(\h,\grave{\h}) \end{array} \right):= J_1 (J_2 \h.J_2 \grave{\h}).
$$
Written out component-wise this is
\be\label{h eqn big}
c^2 h_1'' + \lambda_- h_1 + \lambda_- b_1(\h,\h)=0 \mand  c^2 h_2'' + \lambda_+ h_2 + \lambda_+ b_2(\h,\h)=0.
\ee
\begin{remark} Since $\lambda_-$ is associated to the acoustic branch of the dispersion relation and $\lambda_+$ to the optical branch, we informally think of $h_1$
as the ``acoustic part" of the solution and $h_2$ as the ``optical part." What we shall see down the line is that the dominant part of $h_1$ is a $\sech^2$ traveling wave, whereas the periodic
solutions will be largest in $h_2$.
\end{remark}

\subsection{The Friesecke-Pego cancelation}
Applying $\Fo$ to the first equation in \eqref{h eqn big} gives
\be\label{Fo h 1}
\left(-c^2 k^2  +  \tilde{\lambda}_-(k) \right) \hat{h}_1 +\tilde{\lambda}_-(k) \hat{b_1(\h,\h)} = 0.
\ee
Since $\tilde{\lambda}_-(0) = 0$ and is even, it is roughly quadratic in $k$ near the origin.  Obviously so is $k^2$.
Which indicates that we can cancel a $k^2$ out of the above if we chose.  We do not do precisely this, but instead employ a similar
approach inspired by the proof of the existence of low energy solitary waves for monotomic FPU in \cite{friesecke-pego1}.

We know that $\tlambda_-(0)=0$ . We also know from \eqref{lambda derivative bounds} that $|\tlambda'_-(k)| \le 2c_w^2 |k|$. And so the FTOC\footnote{The Fundamental Theorem of Calculus, of course} implies:
\be\label{super sonic}
c^2 k^2 - \tlambda_-(k) > 0 \quad \text{for all $k \ne 0$ provided
 $c^2 \ge c_w^2$.}
\ee
This 
allows us to divide through in \eqref{Fo  h 1}, not by $k^2$, but rather by $-c^2 k^2  + \tilde{\lambda}_-(k)$.  

So put
$$
\tilde{\varpi}_c(k) :=  -{\tilde{\lambda}_-(k)  \over c^2 k^2 -  \tilde{\lambda}_-(k)}.
$$
This function has a removable singularity at $k =0$ when $c^2>c_w^2$ and no other singularities for $k \in \R$. 
Then \eqref{Fo h 1} is equivalent to
$\ds
\hat{h}_1 +\tilde{\varpi}_c(k) \hat{b_1(\h,\h)} = 0.
$

Let $\varpi_c f$ be the Fourier multiplier operator with symbol $\tilde{\varpi}_c $. The above reasoning shows we can rewrite \eqref{h eqn big} as 
\be\label{h eqn big 2}
h_1 +\varpi_c b_1(\h,\h) = 0 \mand c^2 h_2'' + \lambda_+ h_2 + \lambda_+ b_2(\h,\h) = 0\ee
or alternately as:
$$
\H_c(\h):= \left[ \begin{array}{cc} 1 & 0 \\
0 & c^2 \partial_x^2 + \lambda_+ 
\end{array}
\right] \h +  \left[ \begin{array}{cc} \varpi_c & 0 \\
0 & \lambda_+
\end{array}
\right] B(\h,\h) = 0.
$$
\begin{remark} 
We will henceforth require $c^2>c_w^2$ so that our map $\varpi_c$ is well-defined; this is the technical
reason why we look for (and why the authors of \cite{friesecke-pego1} looked for) nonlinear waves which are supersonic.
Unraveling the scalings that lead to \eqref{r eqn} from \eqref{nl} shows that the traveling wave solutions
under investigation here will, 
in the physical coordinates, travel with a speed faster than $c_{sound}:=\sqrt{k_s / \bar{m}}$,
 where $\bar{m}$ is the average of the two masses.
\end{remark}

We have the following nice symmetry result for $\H_c$.
\begin{lemma}\label{EO} If $h_1$ is even and $h_2$ is odd,
the first and second components of ${\H}_c(\h)$ are, respectively, even and odd.
\end{lemma}

\begin{proof}
For a function $f(y)$ let $Rf(y)  = f(-y)$. If $f$ is even then $f = Rf$. If $f$ is odd then $f = - Rf$.
If $\mu$ is a Fourier multiplier with symbol $\tilde{\mu}$ then 
\be\label{R Fo}
R (\mu f)(x) = (R\mu) Rf (x).
\ee By $R \mu$ we mean the Fourier multiplier with symbol $R\tilde{\mu}(k)=\tilde{\mu}(-k)$.

If the symbol $\mu$ is even then this implies $R (\mu f) = \mu (Rf).$ This in turn implies that such a $\mu$ will map even functions to even functions and odd functions to odd functions. Thus $\lambda_\pm$, $\partial_x^2$ and $\varpi_c$ all ``preserve parity."

Informally, we say $\h \in E \times O$ if $h_1$ is even and $h_2$ is odd. The preceding comments imply that we will have our result if we can show that $B(\h,\h)$ maps $E \times O$ to itself since the remaining parts of $\H_c$ will not flip an $E$
to an $O$ or vice versa. 
Note that
$
\h \in E \times O$ if and only if $R\h = I_1 \h 
$
where $I_1 = \diag(1,-1)$.  Thus our goal is to show that if $R \h = I_1 \h$ then $R B(\h,\h) = I_1 B(\h,\h)$.

So suppose that $R \h = I_1 \h$. Using \eqref{R Fo} we have
$\ds
RB(\h,\h) = (RJ_1) R[(J_2 \h)^{.2}].
$
It is easy to see that $R(fg) = Rf Rg$ and so the above gives
$
RB(\h,\h) = (RJ_1) [(R(J_2 \h))^{.2}].
$
Then we use \eqref{R Fo} again to get
$
RB(\h,\h) = (RJ_1) [((RJ_2) (R\h))^{.2}].
$
Since $R \h = I_1 \h$ we have
$
RB(\h,\h) = (RJ_1) [((RJ_2) (I_1 \h))^{.2}].
$
Then associativity gives:
\be\label{RJ2}
RB(\h,\h) = (RJ_1) [((RJ_2 I_1) \h)^{.2}].
\ee

The multiplier for $R J_2$ is, using \eqref{R Fo},
\bes
R\tilde{J}_2(k) = \tilde{J}_2(-k)=\left[\begin{array}{cc}
\tilde{\gamma}(-k) \tilde{\beta}(-k) & \tilde{\gamma}(-k) \tilde{\beta}(-k) \\
\tilde{\gamma}(-k) \tilde{\varrho}(-k) & -\tilde{\gamma}(-k) \tilde{\varrho}(-k) 
\end{array}
\right].
\ees
Now we use the special property of $\tgamma$ in \eqref{gamma is special} to convert this to:
\bes
R\tilde{J}_2(k) =\left[\begin{array}{cc}
\tilde{\gamma}(k) \tilde{\varrho}(k) & \tilde{\gamma}(k) \tilde{\varrho}(k) \\
\tilde{\gamma}(k) \tilde{\beta}(k) & -\tilde{\gamma}(k) \tilde{\beta}(k) 
\end{array}
\right].
\ees
If we let $I_2:=\left[\begin{array}{cc}0 & 1 \\ 1 & 0 \end{array} \right]$ then the above relation 
implies, after a short calculation, that 
$
R\tilde{J}_2 = I_2 \tilde{J}_2 I_1.
$
Since $\tilde{J}_1 = \tilde{J}_2^{-1}$, $I_1^{-1} = I_1$ and $I_2^{-1} = I_2$, this gives.
$
R\tilde{J}_1 = I_1 \tilde{J}_1 I_2.
$

Using these relations in \eqref{RJ2} gives us
$
RB(\h,\h) = I_1 J_1 I_2 [ (I_2 J_2 I_1^2) \h)^{.2} ].
$
Since $I_1^2$ is the identity this is
$
RB(\h,\h) = I_1 J_1 I_2 [ (I_2 J_2) \h)^{.2} ].
$
Also, it is easy to check that $(I_2 \f)^{.2}= I_2 (\f^{.2})$. Thus we have
$
RB(\h,\h) = I_1 J_1 I^2_2 [ (J_2 \h)^{.2} ].
$
Since $I_2^2$ is the identity this is
$
RB(\h,\h) = I_1 J_1  [ (J_2 \h)^{.2} ] = I_1 B(\h,\h).
$
This was our goal and so we are done.
\end{proof}

In light of this result, we restrict our attention henceforth to looking for solutions which are even in the first component and odd in the second.


\subsection{Long wave scaling}
Now make the long wave scaling (inspired by the classical multiscale derivation of the Korteweg-de Vries
equation from monotomic FPU in \cite{kruskal})
$$
h_1(x):=\ep^2 \theta_1( \ep  x),\quad h_2:= \ep^2 \theta_2(\ep x) \mand c^2 = c_w^2 + \ep^2
$$
where $0 < \ep \ll 1$.

\begin{remark}\label{operator scaling}
Note that if $\mu$ is a Fourier multiplier operator with symbol $\tilde{\mu}(k)$ and if $f(x) = \phi(\omega x)$ then
$
\mu f(x) = (\mu^\omega \phi) (\ep x)
$
where $\mu^\omega $
is a Fourier multiplier with symbol
$
{\tmu^\omega}(k) := \tilde{\mu}(\omega k).
$
\end{remark}

After the scaling, \eqref{h eqn big 2} becomes
\be\label{theta eqn big}
\theta_1 +\varpi^\ep b_1^\ep(\thetab,\thetab) = 0 \mand
\ep^2 (c_w^2 +\ep^2) \theta_2'' + \lambda_+^\ep \theta_2 + \ep^2 \lambda^\ep_+b^\ep_2(\thetab,\thetab) = 0
\ee
with $\ds\thetab:=\left(\begin{array}{c} \theta_1\\ \theta_2 \end{array}\right),$
\be\label{lw pi}
\tilde{\varpi^\ep}(K):=\ep^2 \tilde{\varpi}_{{\sqrt{c^2_w+\ep^2}} }(\ep K) = -{\ep^2 \tilde{\lambda}_-(\ep K) \over (c_w^2 + \ep^2)\ep^2K^2 - \tilde{\lambda}_-(\ep K) } 
,
\ee
\be\label{lw lambda}
\tilde{\lambda}^\ep_\pm(K):=\tilde{\lambda}_\pm(\ep K),
\ee
\be\label{lw Jn}
\tilde{J}^\ep_j(K):=\tilde{J}_j(\ep K)
\ee
and
\be\label{lw B def}
B^\ep(\thetab,\grave{\thetab}):=\left(\begin{array}{c} b_1^\ep(\thetab,\grave{\thetab}) \\b_2^\ep(\thetab,\grave{\thetab})\end{array} \right) := J^\ep_1 \left[ J_2^\ep \thetab . J_2^\ep \grave{\thetab} \right].
\ee
Of course  $\varpi^\ep$, $\lambda^\ep_\pm$ and $J_n^\ep$ are Fourier multiplier operators with the  symbols taken in the obvious way. Note
that since we assume the $\ep \in (0,1)$, the scaling implies, via Lemma \ref{lambda lemma} and Corollary \ref{J cor}, that $\tilde{\lambda}^\ep_\pm(Z)$
and $\tilde{J}_n^\ep(Z)$ are analytic for $|\Im(Z)|\le \tau_1 \le \tau_1/\ep$.

\begin{remark} The equation for $\theta_2$ (which is the part associated to the optical branch) is classically singular when $\ep \sim 0$ because of
the term $$\ep^2 (c_w^2 +\ep^2) \theta_2''.$$ 

Had we not performed the Friesecke-Pego cancelation earlier, the $\theta_1$ equation would have a similarly singular term. 
But the cancelation desingularizes that equation as we shall demonstrate below. Of course this why they made that cancelation in their work \cite{friesecke-pego1} and why we do so here. But because $\tlambda_+(0)\ne 0$
there is no chance to make a similar cancelation in the second component.  
\end{remark}

We can write \eqref{theta eqn big} as:
\be\label{theta eqn small}
\Theta_\ep(\thetab):=\left[ \begin{array}{cc} 1 & 0 \\
0 & \ep^2(c_w^2 + \ep^2) \partial_X^2 + \lambda_+^\ep \end{array} \right]\thetab+\left[ \begin{array}{cc} \varpi^\ep & 0 \\
0 & \ep^2\lambda^\ep_+\end{array} \right] B^\ep(\thetab,\thetab) =0.
\ee

The long wave scaling does not effect the symmetry mapping properties that $\H_c$ had. To wit:
\begin{lemma}\label{EO2} If $\theta_1$ is even and $\theta_2$ is odd,
the first and second components of ${\Theta}_\ep(\h)$ are, respectively, even and odd.
\end{lemma}


\section{The formal long wave limit}\label{LWL}

In this section we naively set $\ep = 0$ in \eqref{theta eqn big}. 
This is mostly routine. For instance, given the definitions in \eqref{lw lambda} and \eqref{lw Jn}, we set
$
\tilde{\lambda}_+^0(K):=\tilde{\lambda}_+(0)$ and
$\tilde{J}^0_j(K):= \tilde{J}_j(0)$. So we define, (using \eqref{J zero}):
$$
\lambda^0_+ := 2+2w, \quad
J^0_1:={1 \over 4(1 + w)}
\left[\begin{array}{cr} 1 & 1 \\ 1 & -1
\end{array}
\right] \mand  J^0_2:=2(1+w) \left[ \begin{array}{rr} 1& 1 \\ 1 & -1 \end{array}\right]. 
$$
This in turn leads to the definition
$$
B^0(\thetab,\grave{\thetab}):=\left(\begin{array}{c} b_1^0(\thetab,\grave{\thetab}) \\b_2^0(\thetab,\grave{\thetab})\end{array} \right) := J^0_1 \left[ J_2^0 \thetab . J_2^0 \grave{\thetab} \right]
= 2(1+w)\left(\begin{array}{c}
 \theta_1 \grave{\theta}_1 +  \theta_2 \grave{\theta}_2 \\
\theta_1 \grave{\theta}_2 + \theta_2 \grave{\theta}_1 \end{array}\right).
$$

Blindly setting $\ep=0$ in $\tilde{\varpi}^\ep(K)$ will not work; a ``$0/0$" situation occurs.
Computing the Maclaurin expansion of $\tilde{\lambda}_+(k)$ gives
$$
\tilde{\lambda}_-(k) = c_w^2 k^2 - \alpha_w k^4 + \cdots
$$
where
$$
\alpha_w:={c_w^2 \over 3 } {1 -  w + w^2 \over (1 + w)^2}>0.
$$
Thus we have
$$
\tilde{\varpi}^\ep(K) =  -{\ep^2 \tilde{\lambda}_-(\ep K) \over (c_w^2 + \ep^2)\ep^2K^2 - \tilde{\lambda}_-(\ep K) } 
= - { \ep^4 c_w^2 K^2  + \O_f(\ep^6) \over (c_w^2 + \ep^2)\ep^2 K^2 - \ep^2 c_w^2 K^2 + \alpha_w \ep^4 K^4+ \O_f(\ep^6) }
$$
By ``$\O_f(\ep^n)$" we mean terms which are formally of order $\ep^n$. 

After some cancelations this becomes
$$
\tilde{\varpi}^\ep(K) 
= - { \ep^4 c_w^2 K^2  + \O_f(\ep^6) \over   \ep^4 K^2  + \alpha_w \ep^4 K^4+ \O_f(\ep^6) }= - {  c_w^2 K^2  + \O_f(\ep^2) \over    K^2  + \alpha_w  K^4+ \O_f(\ep^2) }.
$$
Now we set $\ep =0$ to get
\be\label{this is pi0}
\tilde{\varpi}^0(K) := -{c_w^2 \over 1 + \alpha_w K^2} \mand {\varpi}^0 := -c_w^2(1 - \alpha_w \partial_X^2)^{-1}.
\ee

\begin{remark} We will give precise estimates on the operator norms of  $\varpi^\ep - \varpi^0$, $\lambda_+^\ep-\lambda_+^0$ and so on in Section \ref{BE}.
\end{remark}

With all of this, if we put $\ep = 0$ in \eqref{theta eqn big} we arrive at:
$$
\theta_1 - {4w\over 1 - \alpha_w \partial_X^2}  \left[ \theta_1^2 +  \theta_2^2  \right]=0 \mand (2+2w)\theta_2 = 0.
$$
To solve this we take $\theta_2 = 0$ and $\theta_1$ a solution of
$$
\theta_1 - {4w \over 1 - \alpha_w \partial_X^2}  \left[ \theta_1^2   \right]=0.
$$
Applying $1 - \alpha_w \partial_X^2$  to the above results in 
\be\label{TWEKDV}
\alpha_w \theta_1''-\theta_1 +4w \theta_1^2   =0.
\ee
This is a rescaling of the nonlinear differential equation whose solutions give the profile for the KdV solitary waves.
It has an explicit solution given by
\be\label{the soliton}
\theta_1(X)=\sigma(X):=\sigma_0 \sech^2 (2 q_0 X) \quad \text{where} \quad \sigma_0:={3 \over 8w}\mand q_0:={1 \over 2 \sqrt{\alpha_w}}.
\ee
Note that if we put $\sigmab:=(\sigma,0)^t$ then we have shown:
\be\label{what sigma does}
\sigma+\varpi^0 b_1^0(\sigmab,\sigmab) = 0. 
\ee

\begin{remark} Observe that in the limit $\ep \to 0^+$ the part of the solution associated to optical branch, $\theta_2$, plays no role. All the action is happening for the acoustic part $\theta_1$.
\end{remark}
\begin{remark} In \cite{gaison-etal} the authors use homogenization theory to show that solutions of \eqref{r eqn} with initial data of the form $r(j,0) = \ep^2 R(\ep j)$ and $\dot{x}(j,0) = \ep^2 V(\ep j)$   are well-approximated over long times by a pair of solutions of KdV. Specifically they show that
$$
r(j,t) = \ep^2 U_-(\ep(j - c_w t),\ep^3t) +\ep^2 U_+(\ep(j + c_w t),\ep^3t) + \O(\ep^{3/2})\quad \text{for all $|t| \le C\ep^{-3}$}
$$
where $U_\pm$ solve the KdV equations
$$
\pm \partial_T U_\pm +  \alpha_w \partial_X^3 U_\pm + 4 w \partial_X (U^2_\pm) = 0.
$$
Making a traveling wave ansatz for these of the form $U_\pm(X,T) = \theta(X \mp T)$ 
results in 
 \eqref{TWEKDV}; the coefficients match exactly. Which is to say the results here are consistent with those in \cite{gaison-etal}.
\end{remark}

\section{Periodic solutions}\label{PS}

In this section we prove the existence of spatially periodic solutions of \eqref{theta eqn big}.
To this end, we first compute the linearization of that equation about $\thetab(X) = 0$ to get
$$
\theta_1 = 0 \mand \ep^2(c_w^2 + \ep^2) \theta_2'' + \lambda_+^\ep \theta_2 = 0.
$$
We are looking for solutions where $\theta_2$ is odd and periodic. Thus we can take $\theta_2(X) = \sin(K_\ep X)$ for some $K_\ep \in \R$. 
Inserting this into the second equation above, and recalling that $\lambda_+^\ep$ is a Fourier multiplier operator with symbol $\tlambda_+(\ep K)$, gives
$$
\ep^2 (c_w^2 + \ep^2) K_\ep^2 - \tlambda_+(\ep K_\ep) = 0.
$$
If we put $c = c_\ep:= \sqrt{c_w^2 + \ep^2}$ and $$K_\ep:=k_{c_\ep}/\ep,$$ this last equation is exactly  \eqref{this is kc}. Which is to say, by virtue of Part (vi) of Lemma \ref{lambda lemma},
that $K_\ep$ is its  unique nonnegative solution.

Thus we have odd periodic solutions of the the linearization of \eqref{theta eqn big} of the form
$$
\thetab(X) = \nub_\ep(X):=\sin(K_\ep X) \jb.
$$
We can extend the existence of periodic solutions for the linear problem to the full nonlinear problem \eqref{theta eqn big} by means of 
 the technique of ``bifurcation from a simple eigenvalue," developed by Crandall \& Rabinowitz in  \cite{crandall-rabinowitz} and Zeidler in \cite{zeidler2}.  
Here is what we find:
\begin{theorem} \label{periodic solutions}For all $w>1$ there exist $\ep_0 > 0$, $a_0>0$ and $0 < C_1 < C_2$ such the following holds for all for $\ep \in (0,\ep_0)$.
There exist maps
\be\label{maps}
\begin{split}
K^a_\ep&: [-a_0,a_0] \longrightarrow \R\\
\psi_{\ep,1}^a&: [-a_0,a_0] \longrightarrow C^\infty_\per\cap\left\{\text{even functions}\right\}\\
\psi_{\ep,2}^a&: [-a_0,a_0] \longrightarrow C^\infty_\per\cap\left\{\text{odd functions}\right\}
\end{split}\ee
with the following properties.
\begin{enumerate}[(i)]
\item Putting
$$\thetab(X)=
a\varphib^a_\ep(X):=
a\left( \begin{array}{c}
0\\
\sin(K_\ep^a X) 
\end{array}
\right)+a\left( \begin{array}{c}
 \psi^a_{\ep,1}(K_\ep^a X)\\
 \psi^a_{\ep,1}(K_\ep^a X)
\end{array}
\right)=:a {\nub}(K_\ep^a X) + \psib^a_\ep(K_\ep^a X)
$$
solves \eqref{theta eqn big} for all $|a|\le a_0$.
\item $K_\ep^0 = K_\ep$ where $K_\ep$ is the unique positive solution of 
\be\label{Kep}
\txi_{\sqrt{c_w^2+\ep^2}}(\ep{K}):= -\ep^2 (c_w^2 + \ep^2)K^2 + \tlambda_+(\ep K) =0.
\ee
Moreover $K_\ep = \O(1/\ep)$ in the sense that
\bes
C_1/\ep < K_\ep < C_2 /\ep \quad \text{ for all $\ep \in (0,\ep_0)$}.
\ees
\item $\psi^0_{\ep,1} = \psi^0_{\ep,2} = 0$.
\item $\ds \int_{-\pi}^{\pi} \psi^a_{\ep,2}(y) \sin(y) \ dy = 0$ for all $|a|\le a_0$.
\item For all $r \ge 0$, there exists $C_r>0$ such that for all $|a|,|\ga|\le a_0$ we have
\be\label{periodic bound}
\left \vert  \ep K_\ep^a \right \vert + \left\|  \psib^a_{\ep} \right\|_{C^r_\per \times C^r_\per}\le C_r
\ee
and
\be\label{periodic lip}
\left \vert  K_\ep^a -K_\ep^\ga\right \vert + \left\|  \psib^a_{\ep} - \psib^\ga_\ep \right\|_{C^r_\per\times C^r_\per} \le C_r|a-\ga|.
\ee
\end{enumerate}
\end{theorem}

The remainder of Section \ref{PS} is dedicated to the proof of this theorem.

\subsection{Frequency freezing}
We begin by making the additional scaling 
\be\label{bifur param scaling}
\thetab(X) := \phib(\omega{X}) \quad\text{ with } \phib:=\left(\begin{array}{c} \phi_1\\ \phi_2 \end{array}\right), \ \omega \in \R,  
\ee
where $\phib(Y)$ is $2\pi$-periodic.  By Remark \ref{operator scaling}, our system \eqref{theta eqn small} becomes
\be\label{phi eqn small}
\Phib_{\ep}(\phib,\omega) :=
\left[ \begin{array}{cc} 1 & 0 \\
0 & \ep^2\omega^2(c_w^2 + \ep^2) \partial_Y^2 + \lambda_+^{\ep\omega} \end{array} \right]\phib+\left[ \begin{array}{cc} \varpi^{\ep,\omega} & 0 \\
0 & \ep^2\lambda^{\ep\omega}_+\end{array} \right] B^{\ep\omega}(\phib,\phib) =0.
\ee 
where $\varpi^{\ep,\omega}$ is the multiplier with symbol
$$
\tilde{\varpi}^{\ep,\omega}(K) := \tilde{\varpi}^{\ep}(\omega{K}),
$$
and the multipliers $\lambda_+^{\ep\omega}$ and $B^{\ep\omega}$ conform to their prior definitions.

Since $B^{\ep\omega}$ is quadratic in $\phib$, it is easy to see that 
$$
D_{\phib}\Phib_{\ep}(0,\omega) = \begin{bmatrix*}
1 &0 \\
0 &\ep^2\omega^2(c_w^2+\ep^2)\partial_Y^2 + \lambda_+^{\ep\omega}
\end{bmatrix*}.
$$
When $\omega = K_{\ep}$, one may show that zero is a simple eigenvalue of $D_{\phib}\Phib_{\ep}(0,K_{\ep})$ when the operator is restricted to a suitable function space;
this is essentially just the calculation carried out at the start of this section.  Consequently,
the classical bifurcation results in \cite{crandall-rabinowitz}  and \cite{zeidler2} can be used to show that there exists a nontrivial family of solutions to $\Phib_{\ep}(\phib,\omega) = 0$
branching out of $(0,K_\ep)$.
 Unfortunately, those classical results do not provide, in an easy way, estimates on the solution which are uniform in $\ep$.
 And so, while our strategy is modeled on the proofs of the results in \cite{crandall-rabinowitz}  and \cite{zeidler2}, we carry out the proof from scratch and always with our eyes
 on how quantities depend on $\ep$.
 Our first step is to convert \eqref{phi eqn small} to a fixed point equation.

\subsection{Conversion to a fixed-point problem}\label{conversion} Let 
$$
\Y = (H_{\per}^2 \cap \{\text{even functions}\}) \times (H_{\per}^2 \cap \{\text{odd functions}\}), 
$$
where $H_{\per}^r$ is the Sobolev space of $2\pi$-periodic functions $\phi$ such that 
$$
\norm{\phi}_{H_{\per}^r} := \left(\sum_{k \in \Z} (1+k^2)^r|\hat{\phi}(k)|^2\right)^{1/2} < \infty.
$$

With $\nub(Y) := \sin(Y)\bf{j}$, we have the direct sum decomposition $\Y = \mathcal{N} \oplus \ortho$, where $\ortho \subseteq \Y$ is the orthogonal complement of $\mathcal{N} := \spn(\{\nub\})$ in the standard $H_{\per}^2 \times H_{\per}^2$ inner product, i.e.,
$$
\ip{\phib}{\psib}_{H_{\per}^2\times H_{\per}^2} := \ip{\phi_1}{\psi_1}_{H_{\per}^2} + \ip{\phi_2}{\psi_2}_{H_{\per}^2} \quad\text{ with }\quad \ip{\phi}{\psi}_{H_{\per}^2} := \sum_{k \in \Z} (1+k^2)^2\hat{\phi}(k)\overline{\hat{\psi}(k)}.
$$
We may then write any $\phib \in \Y$ as 
\be\label{ortho decomp}
\phib = a\nub + a\psib \quad\text{ for some }\quad a \in \R, \ \psib \in \ortho.  
\ee
Observe that if $\psib \in \ortho$, then $\hat{\psib}(\pm1)\cdot\jb = 0$.  Set 
\be\label{X!}
\X = \ortho \times \R.
\ee

Since the trivial solution $\phib = 0$ already solves \eqref{phi eqn small} for any choice of $\omega$, we will assume $a \ne 0$. After factoring and dividing by $a$, the problem 
$$
\Phib_{\ep}(\phib,K_{\ep}+t) = \Phib_{\ep}(a\nub+a\psib,K_{\ep}+t)=0
$$
becomes
\be\label{ortho decomp 1-1}
\psi_1 + a\varpi^{\ep,{K_{\ep}}+t}b_1^{\ep({K_{\ep}}+t)}(\nub+\psib,\nub+\psib) = 0 
\ee
and
\be\label{ortho decomp 1-2}
(\ep^2({K_{\ep}}+t)^2(c_w^2+\ep^2)\partial_Y^2 + \lambda_+^{\ep({K_{\ep}}+t)})(\sin(Y)+\psi_2) + a\ep^2\lambda_+^{\ep({K_{\ep}}+t)}b_2^{\ep({K_{\ep}}+t)}(\nub+\psib,\nub+\psib)=0.
\ee

Let $\Pi_1$ be the Fourier multiplier with symbol 
$$
\tilde{\Pi}_1 := \begin{cases}
1, &|k| = 1 \\
0, &|k| \ne 1 
\end{cases}
$$
and set $\Pi_2 = 1-\Pi_1$.  Then \eqref{ortho decomp 1-1} and \eqref{ortho decomp 1-2} are equivalent to
\be\label{ortho decomp 2-1}
\psi_1 + a\varpi^{\ep,{K_{\ep}}+t}b_1^{\ep({K_{\ep}}+t)}(\nub+\psib,\nub+\psib) = 0,
\ee
\be\label{ortho decomp 2-2}
(\ep^2({K_{\ep}}+t)^2(c_2^2+\ep^2)\partial_X^2 + \lambda_+^{\ep({K_{\ep}}+t)})\sin(Y)\!+\! a\ep^2\Pi_1\lambda_+^{\ep({K_{\ep}}+t)}b_2^{\ep({K_{\ep}}+t)}(\nub+\psib,\nub+\psib)=0
\ee
and
\be\label{ortho decomp 2-3}
(\ep^2({K_{\ep}}+t)^2(c_2^2+\ep^2)\partial_X^2 + \lambda_+^{\ep({K_{\ep}}+t)})\psi_2 + a\ep^2\Pi_2\lambda_+^{\ep({K_{\ep}}+t)}b_2^{\ep({K_{\ep}}+t)}(\nub+\psib,\nub+\psib)=0.
\ee

Condition \eqref{ortho decomp 2-1} immediately gives a fixed-point equation for $\psi_1$, and we see that \eqref{ortho decomp 2-2} holds if and only if the Fourier transform of its left side evaluated at $k=\pm1$ is zero.  Because $b_2^{\ep({K_{\ep}}+t)}(\nub+\psib,\nub+\psib)$ is odd by (the proof of) Lemma \ref{EO}, we need only consider this Fourier transform at $k=1$.  With 
$$
\tilde{\xi}_{\sqrt{c_w^2+\ep^2}}(k) = -(c_w^2+\ep^2)(\ep{k})^2 + \tlambda_+(\ep{k})
$$
as in \eqref{Kep} and $\cep := \sqrt{c_w^2+\ep^2}$, set 
\be\label{xi ep t}
\txi^{\ep,t}(k) := \tilde{\xi}_{\cep}(\ep({K_{\ep}}+t)k)
\ee
so that \eqref{ortho decomp 2-2} is equivalent to
\be\label{ortho decomp 2-2 equiv}
\frac{1}{2i}\tilde{\xi}^{\ep,t}(1)+ a\ep^2\Fo\left[\Pi_1\lambda_+^{\ep({K_{\ep}}+t)}b_2^{\ep({K_{\ep}}+t)}(\nub+\psib,\nub+\psib)\right](1) = 0.
\ee

Taylor's theorem tells us that 
$$
\txi^{\ep,1}(1) = \txi_{\cep}(\ep K_{\ep}+ \ep{t}) = \txi_{\cep}'(\ep K_{\ep})(\ep{t}) + R_{\ep}(\ep{t})(\ep{t})^2,
$$
and Part (vi) of Lemma \ref{lambda lemma} provides a number $l_0 > 0$ such that 
 $$
 |\txi_{\cep}'(\ep{K_{\ep}})| \ge l_0
 $$
for all $\ep$ sufficiently close to zero.  So, we may rewrite \eqref{ortho decomp 2-2 equiv} as
\be\label{ortho decomp 2-2 equiv again}
t = -\frac{\ep}{\txi_{\cep}'(\ep{K_{\ep}})}R_{\ep}(\ep{t})t^2 - \frac{2i\ep{a}}{\txi_{\cep}'(\ep{K_{\ep}})}\Fo\left[\Pi_1\lambda_+^{\ep({K_{\ep}}+t)}b_2^{\ep({K_{\ep}}+t)}(\nub+\psib,\nub+\psib) \right](1)
\ee

Finally, we will show that $\txi^{\ep,t}(k) \ne 0$ for $k \ne \pm1$, which means that the multiplier $(\xi^{\ep,t})^{-1}$ with symbol $(\txi^{\ep,t})^{-1}$ is well-defined on the range of $\Pi_2$ for suitably small $\ep$ and $t$.  Then \eqref{ortho decomp 2-3} becomes
\be\label{ortho decomp 3-1}
\psi_2 = -a\ep^2\left(\xi^{\ep,t}\right)^{-1}\Pi_2\lambda_+^{\ep({K_{\ep}}+t)}b_2^{\ep({K_{\ep}}+t)}(\nub+\psib,\nub+\psib).
\ee 

We have arrived at our ultimate fixed-point problem.  Set $\Psib_{\ep} := (\Psi_{\ep,1}, \Psi_{\ep,2},\Psi_{\ep,3})$ with
\begin{align*}
\Psi_{\ep,1}(\psib,t,a) &:=-a\varpi^{\ep,{K_{\ep}}+t}b_1^{\ep({K_{\ep}}+t)}(\nub+\psib,\nub+\psib) \\
\\
\Psi_{\ep,2}(\psib,t,a) &:= -a\ep^2(\xi^{\ep,t})^{-1}\Pi_2\lambda_+^{\ep({K_{\ep}}+t)}b_2^{\ep({K_{\ep}}+t)}(\nub+\psib,\nub+\psib) \\
\\
\Psi_{\ep,3}(\psib,t,a) &:= -\frac{\ep}{\txi_{\cep}'(\ep{K_{\ep}})}R_{\ep}(\ep{t})t^2 - \frac{2i\ep{a}}{\txi_{\cep}'(\ep{K_{\ep}})}\Fo\big[\Pi_1\lambda_+^{\ep({K_{\ep}}+t)}b_2^{\ep({K_{\ep}}+t)}(\nub+\psib,\nub+\psib)\big]\!(1)
\end{align*}

We will solve this problem by applying the following lemma, whose proof is given in Section \ref{fixed point proof}, to the map $\Psib_{\ep}$ for $\ep$ sufficiently small.

\begin{lemma}\label{fixed point lemma}
Let $\X$ be a Banach space and let $\ball(r) = \{x \in \X : \norm{x} \le r\}$.  For $0 < \ep < \ep_0$ let $F_{\ep} \colon \X \times \R \to \X$ be maps with the property that for some $C_1,a_1,r_1 > 0$, if $x,y \in \ball(r_1)$ and $|a| \le a_1$, then
\begin{align}
\sup_{0 < \ep < \ep_0} \norm{F_{\ep}(x,a)} &\le C_1\left(|a|+|a|\norm{x}+\norm{x}^2\right) \label{bound1} \\
\nonumber \\
\sup_{0 < \ep < \ep_0} \norm{F_{\ep}(x,a)-F_{\ep}(y,a)} &\le C_1\left(|a| + \norm{x}+\norm{y}\right)\norm{x-y} \label{bound2}
\end{align}
Then there exist $a_0 \in (0,a_1], r_0 \in (0,r_1]$ such that for each $0 < \ep < \ep_0$ and $|a| \le a_0$, there is a unique $x_{\ep}^a \in \ball(r_0)$ such that $F_{\ep}(x_{\ep}^a,a) = x_{\ep}^a$. 

Suppose as well that the maps $F_{\ep}(\cdot,a)$ are Lipschitz on $\ball(r_0)$ uniformly in $a$ and $\ep$, i.e., there is $L_1 > 0$ such that 
\be\label{lip bound}
\sup_{\substack{0 < \ep < \ep_0 \\ \norm{x} \le r_0}} \norm{F_{\ep}(x,a)-F_{\ep}(x,\grave{a})} \le L_1|a-\grave{a}|
\ee
for all $|a|,|\grave{a}| \le a_0$. Then the mappings $[-a_0,a_0] \to \X \colon a \mapsto x_{\ep}^a$ are also uniformly Lipschitz; that is, there is $L_0 > 0$ such that 
\be\label{fp lip concl}
\sup_{0 < \ep < \ep_0} \norm{x_{\ep}^a-x_{\ep}^{\grave{a}}} \le L_0|a-\grave{a}|
\ee
for all $|a|,|\grave{a}| \le a_1$. 
\end{lemma}

\subsection{Application of Lemma \ref{fixed point lemma}} We begin with a general observation about Fourier multipliers.  The proof of this lemma follows from direct calculations with the norm
$$
\norm{\psi}_{H_{\per}^r}^2 = \sum_{k \in \Z} (1+k^2)^r|\hat{\psi}(k)|^2,
$$
and so we omit it.  Throughout this section, we denote by $\bdop(\mathcal{U},\mathcal{V})$ the space of bounded linear operators between normed spaces $\mathcal{U}$ and $\mathcal{V}$ and set $\bdop(\mathcal{U}) := \bdop(\mathcal{U},\mathcal{U})$. 

\begin{lemma}\label{fm props}
Let $\mu$ be a Fourier multiplier with symbol $\tmu \in L^{\infty}(\R)$ and let $\omega \in \R$.  As in Remark \ref{operator scaling}, let $\mu^{\omega}$ be the Fourier multiplier with symbol $\tmu^{\omega}(k) = \tmu(\omega{k})$.  Then
\begin{enumerate}[(i)]
\item $\sup_{r,\omega \in \R} \norm{\mu^{\omega}}_{\bdop(H_{\per}^r)} \le \norm{\tmu}_{L^{\infty}(\R)}.$
\item If $\tmu$ is Lipschitz, i.e., there is $\Lip(\tmu)> 0$ such that $|\tmu(k) - \tmu(\grave{k})| \le \Lip(\tmu)|k-\grave{k}|$, then
$$
\norm{(\mu^{\omega}-\mu^{\grave{\omega}})\psi}_{H_{\per}^r} \le \Lip(\tmu)|\omega-\grave{\omega}|\norm{\psi}_{H_{\per}^{r+1}}
$$
for all $\omega,\grave{\omega},r \in \R$ and $\psi \in H_{\per}^{r+1}$.  
\item If there exist $C,p > 0$ such that 
$$
|\tmu(k)| \le \frac{C}{(1+k^2)^p}
$$
for all $k \in \R$, then $\norm{\mu}_{\bdop(H_{\per}^r,H_{\per}^{r+2p})} \le C$.
\end{enumerate}
\end{lemma}

\begin{remark}
Informally, Part (ii) of Lemma \ref{fm props} means that taking a Lipschitz estimate for the map $\omega \mapsto \mu^{\omega}$ costs us a derivative.
\end{remark}

The following two lemmas on the Fourier multipliers $\varpi^{\ep,K_{\ep}+t}$ and $(\xi^{\ep,t})^{-1}\Pi_2$ are the keys to our application of Lemma \ref{fixed point lemma} to the maps $\Psib_{\ep}$.  They follow directly from the corresponding results for the symbols $\tvarpi^{\ep,K_{\ep}+t}$ and $(\txi^{\ep,t})^{-1}$, which are stated below as Lemmas \ref{varpi symbol props} and \ref{xi symbol props} and proved in Section \ref{mult R props}.

\begin{lemma}\label{varpi props}
\begin{enumerate}[(i)]
\item There exist $\ep_{11}, C_{\varpi \max} > 0$ such that 
$$
\sup_{\substack{0 < \ep < \ep_{11} \\ |t| \le 1 \\ r \in \R}} \norm{\varpi^{\ep,K_{\ep}+t}}_{\bdop(H_{\per}^r, H_{\per}^{r+2})} \le C_{\varpi \max}.
$$

\item There exists $C_{\varpi \Lip} > 0$ such that if $|t|,|\grave{t}| \le 1$, then
$$
\sup_{\substack{0 < \ep < \ep_{11} \\ r \in \R}} \norm{\varpi^{\ep,K_{\ep}+t}-\varpi^{\ep,K_{\ep}+\grave{t}}}_{\bdop(H_{\per}^r)} \le C_{\varpi \Lip}|t-\grave{t}|.
$$
\end{enumerate}
\end{lemma}

\begin{lemma}\label{xi props}
\begin{enumerate}[(i)]
\item There exist $\ep_{12}, C_{\xi \max} > 0$ such that
$$
\sup_{\substack{0 < \ep < \ep_{12} \\ |t| \le 1 \\ r \in \R}} \norm{\ep^2(\xi^{\ep,t})^{-1}\Pi_2}_{\bdop(H_{\per}^r,H_{\per}^{r+2})} \le C_{\xi \max}.
$$
\item There exists $C_{\xi \Lip} > 0$ such that 
$$
\sup_{\substack{0 < \ep < \ep_{12} \\ r \in \R}} \norm{\ep^2(\xi^{\ep,t})^{-1}\Pi_2 - \ep^2(\xi^{\ep,t})^{-1}\Pi_2}_{\bdop(H_{\per}^r)} \le C_{\xi \Lip}|t-\grave{t}|
$$
for all $|t|, |\grave{t}| \le 1$.
\end{enumerate}
\end{lemma}

\begin{lemma}\label{varpi symbol props}
There exists $\ep_{11} > 0$ such that the following hold.
\begin{enumerate}[(i)]
\item There is $C_{\tvarpi \max} > 0$ such that 
\be\label{varpi bounded}
\sup_{\substack{0 < \ep < \ep_{11} \\ k \in \Z \\ |t| \le 1}} |\tvarpi^{\ep,K_{\ep}+t}(k)| \le \frac{C_{\tvarpi \max}}{1+k^2}.
\ee
\item There is $C_{\tvarpi \Lip} > 0$ such that 
\be\label{varpi lip}
\sup_{\substack{0 < \ep < \ep_{11} \\ k \in \Z}} |\tvarpi^{\ep,K_{\ep}+t}(k)-\tvarpi^{\ep,K_{\ep}+\grave{t}}(k)| \le C_{\tvarpi \Lip}|t-\grave{t}|
\ee
for all $|t|,|\grave{t}| \le 1$.
\end{enumerate}
\end{lemma}

\begin{lemma}\label{xi symbol props}
There exists $\ep_{12} > 0$ such that the following hold.
\begin{enumerate}[(i)]
\item There is $C_{\txi \max} > 0$ such that 
\be\label{xi bounded}
\sup_{\substack{0 < \ep < \ep_{12} \\ k \in \Z\setminus\{-1,1\} \\ |t| \le 1}} \left|\frac{1}{\txi^{\ep,t}(k)}\right| \le \frac{C_{\txi \max}}{1+k^2}.
\ee
\item There is $C_{\txi \Lip} > 0$ such that
\be\label{xi lip}
\sup_{\substack{0 < \ep < \ep_0 \\ k \in \Z\setminus\{-1,1\}}} \left|\frac{1}{\txi^{\ep,t}(k)} - \frac{1}{\txi^{\ep,\grave{t}}(k)}\right| \le C_{\txi \Lip}|t-\grave{t}|.
\ee
\end{enumerate}
\end{lemma}

Last, Taylor's theorem provides the following useful decomposition of $\txi^{\ep,t}$, which we prove in Section \ref{mult R props}. 

\begin{lemma}\label{xi decomp}
For $0 < \ep < \ep_{12}$ and $\tau \in \R$, we have
\be\label{xi 1}
\txi_{\cep}(\ep{K_{\ep}}+\tau) = \txi_{\cep}'(\ep{K_{\ep}})\tau + R_{\ep}(\tau)\tau^2,
\ee
where the functions $R_{\ep}$ have the following property: there exist $C_{R \max},C_{R \Lip} > 0$ such that when $0 < \ep < \ep_{12}$,
\be\label{R bounds}
\sup_{0 < \ep < \ep_{12}} |R_{\ep}(\tau)| \le C_{R \max} \mand \sup_{0 < \ep < \ep_{12}} |R_{\ep}(\tau)-R_{\ep}(\grave{\tau})| \le C_{R \Lip}|\tau-\grave{\tau}|
\ee
for all $\tau,\grave{\tau} \in \R$.
\end{lemma}

We are now ready to apply Lemma \ref{fixed point lemma} to our map $\Psib_{\ep}$.

\begin{proposition}\label{fp app}
Let $\ep_0 = \min\{\ep_{11},\ep_{12}\}$. The maps $\Psib_{\ep}, 0 < \ep < \ep_0,$ satisfy the conditions \eqref{bound1}, \eqref{bound2}, and \eqref{lip bound} of Lemma \ref{fixed point lemma} on the space $\X$ defined in \eqref{X!} when $a_0 = r_0 = 1$.  
\end{proposition}

\begin{proof}
We begin with some additional notation.  Set $\H^r := H_{\per}^r \times H_{\per}^r, \norm{\psib}_r := \norm{\psib}_{\H^r},$ and
\be\label{psi T ops}
T_0^{\ep}(t) :=  \begin{bmatrix*}
\varpi^{\ep,{K_{\ep}}+t} &0 \\
0 &\ep^2\left(\xi^{\ep,t}\right)^{-1}\Pi_2
\end{bmatrix*},
\quad
T_1^{\ep}(t) := \begin{bmatrix*}
1 &0 \\
0 &\lambda_+^{\ep({K_{\ep}}+t)}
\end{bmatrix*}
J_1^{\ep({K_{\ep}}+t)},
\quad
T_2^{\ep}(t) := J_2^{\ep({K_{\ep}}+t)}.
\ee

Lemmas \ref{varpi props} and \ref{xi props} combine to produce constants $C_0,C_1,C_2 > 0$ such that the following estimates hold:
\be\label{T0 bound}
\sup_{\substack{0 < \ep < \ep_0 \\ |t| \le 1 \\ r \in \R}} \norm{T_0^{\ep}(t)}_{\bdop(\H^{r},\H^{r+2})} \le C_0  
\ee
\be\label{T12 bound}
\sup_{\substack{0 < \ep < \ep_0 \\ |t| \le 1 \\ r \in \R}} \norm{T_i^{\ep}(t)}_{\bdop(\H^r)} \le C_i,  \ i = 1,2 
\ee
\be\label{T0 lip}
\sup_{\substack{0 < \ep < \ep_0 \\ r \in \R}} \norm{T_0^{\ep}(t)-T_0^{\ep}(\grave{t})}_{\bdop(\H^r)} \le C_0|t-\grave{t}|, \ |t|,|\grave{t}| \le 1 
\ee
\be\label{T12 lip}
\sup_{\substack{0 < \ep < \ep_0 \\ r \in \R}} \norm{T_i^{\ep}(t)-T_i^{\ep}(\grave{t})}_{\bdop(\H^r,\H^{r-1})} \le C_i|t-\grave{t}|, \ |t|,|\grave{t}| \le 1, i = 1,2. 
\ee

Define
\begin{align}
\bG_{\ep}(\psib,t) &= T_1^{\ep}(t)(T_2^{\ep}(t)(\nub+\psib))^{.2} \label{big G} \\
\bF_{\ep}(\psib,t) &= T_0^{\ep}(t)\bG_{\ep}(\psib,t) \label{fep}.
\end{align}

The estimates \eqref{T12 bound} along with the Sobolev embedding estimate
\be\label{sob}
\norm{\phib.\psib}_r \le C_{\text{sob},r}\norm{\phib}_r\norm{\psib}_r, \ \phib,\psib \in \H^r, \ r \ge 1,
\ee
give $M_{12,r} > 0$ such that 
\be\label{G bound}
\sup_{\substack{0 < \ep < \ep_0 \\ |t| \le 1}} \norm{\bG_{\ep}(\psib,t)}_r \le M_{12,r}(\norm{\psib}_{r}^2 + \norm{\psib}_{r} + 1).
\ee
We then use \eqref{T0 bound} to find
$$
\sup_{\substack{0 < \ep < \ep_0 \\ |t| \le 1}} \norm{\bF_{\ep}(\psib,t)}_{r} \le \sup_{\substack{0 < \ep < \ep_0 \\ |t| \le 1}} \norm{T_0^{\ep}(t)}_{\bdop(\H^{r-2},\H^r)}\norm{\bG_{\ep}(\psib,t)}_{r-2} \le C_0M_{12}(\norm{\psib}_{r-2}^2 + \norm{\psib}_{r-2}+1).
$$
Set $M_{012} = C_0M_{12}$.  Since
\be\label{first two comps}
\begin{pmatrix*}
\Psi_{\ep,1}(\psib,t,a) \\
\Psi_{\ep,2}(\psib,t,a)
\end{pmatrix*}
=
-a{\bf{F}}_{\ep}(\psib,t),
\ee
we find
\be\label{bootstrap bound}
\sup_{\substack{0 < \ep < \ep_0 \\ |t| \le 1}}\normlarge{\begin{pmatrix*}
\Psi_{\ep,1}(\psib,t,a) \\
\Psi_{\ep,2}(\psib,t,a)
\end{pmatrix*}}_r \le M_{012}|a|(\norm{\psib}_{r-2}^2 + \norm{\psib}_{r-2} + 1).
\ee
We will return to the estimate \eqref{bootstrap bound} when we prove the bounds \eqref{periodic bound} for our fixed points.  For now, we take $r=2$ to obtain a constant $M_2 > 0$ such that
$$
\sup_{\substack{0 < \ep < \ep_0 \\ \norm{\psib}_2, |t|, |a| \le 1}}\normlarge{\begin{pmatrix*}
\Psi_{\ep,1}(\psib,t,a) \\
\Psi_{\ep,2}(\psib,t,a)
\end{pmatrix*}}_2 \le M_2|a|.
$$
This implies the first estimate \eqref{bound1} of Lemma \ref{fixed point lemma} for the components $\Psi_{\ep,1}$ and $\Psi_{\ep,2}$.  

To prove the second estimate \eqref{bound2} of Lemma \ref{fixed point lemma}, we first rewrite
\be\label{F lip 1}
\bF_{\ep}(\psib,t) - \bF_{\ep}(\grave{\psib},\grave{t}) = T_0^{\ep}(t)(\bG_{\ep}(\psib,t)-\bG_{\ep}(\grave{\psib},\grave{t})) + (T_0^{\ep}(t)-T_0^{\ep}(\grave{t}))\bG_{\ep}(\grave{\psib},\grave{t})
\ee
and then find
\begin{align*}
\bG_{\ep}(\psib,t)-\bG_{\ep}(\grave{\psib},\grave{t}) &= T_1^{\ep}(t)( [T_2^{\ep}(t)(\nub+\psib)+T_2^{\ep}(\grave{t})(\nub+\grave{\psib})].[(T_2^{\ep}(t)-T_2^{\ep}(\grave{t}))(\nub+\grave{\psib})]) \\
&+T_1^{\ep}(t)[T_2^{\ep}(t)(\nub+\psib)+T_2^{\ep}(\grave{t})(\nub+\grave{\psib})].[T_2^{\ep}(t)(\psib-\grave{\psib})] \\
&+(T_1^{\ep}(t)- T_1^{\ep}(\grave{t}))(T_2^{\ep}(\grave{t})(\nub+\grave{\psib}))^{.2}.
\end{align*}
We estimate the third term above; estimates for the first two terms follow by similar techniques.  We have
\begin{align*}
\norm{(T_1^{\ep}(t)- T_1^{\ep}(\grave{t}))(T_2^{\ep}(\grave{t})(\nub+\grave{\psib}))^{.2}}_{r-2} &\le C_1|t-\grave{t}|\norm{(T_2^{\ep}(\grave{t})(\nub+\grave{\psib}))^{.2}}_{r-1} \text{ by \eqref{T12 lip}} \\
\\
&\le C_1C_{\text{sob},r-1}\norm{T_2^{\ep}(\grave{t})(\nub+\grave{\psib})}_{r-1}^2 \text{ by \eqref{sob}} \\
\\
&\le C_1C_{\text{sob},r-1}C_2^2\norm{\nub+\grave{\psib}}_{r-1}^2 \text{ by \eqref{T12 bound}} \\
\\
&\le C_1C_{\text{sob},r-1}C_2^2(\norm{\grave{\psib}}_{r-1}^2+2\norm{\nub}_{r-1}\norm{\grave{\psib}}_{r-1} + \norm{\nub}_{r-1}^2) \\
\\
&\le C_1C_{\text{sob},r-1}C_2^2\max\{2\norm{\nub}_{r-1},1\}(\norm{\grave{\psib}}_{r-1}^2+\norm{\grave{\psib}}_{r-1}+1).
\end{align*}
After comparable work on the other two terms, we ultimately arrive at a constant $L_{12,r} > 0$ such that 
\be\label{G lip}
\sup_{0 < \ep < \ep_0} \norm{\bG_{\ep}(\psib,t)-\bG_{\ep}(\grave{\psib},\grave{t})}_{r-2} \le L_{12,r}(\norm{\grave{\psib}}_r^2 + \norm{\psib}_r + \norm{\grave{\psib}}_r+1)(\norm{\psib-\grave{\psib}}_{r-1} + |t-\grave{t}|).
\ee

We will need this estimate below when we work on $\Psi_{\ep,3}$.  For now, we return to \eqref{F lip 1} and find
$$
\norm{\bF_{\ep}(\psib,t) - \bF_{\ep}(\grave{\psib},\grave{t})}_r \le \norm{T_0^{\ep}(t)}_{\bdop(\H^{r-2},\H^r)}\norm{\bG_{\ep}(\psib,t)-\bG_{\ep}(\grave{\psib},\grave{t})}_{r-2} +\norm{T_0^{\ep}(t)-T_0^{\ep}(\grave{t})}_{\bdop(\H^r)}\norm{\bG_{\ep}(\grave{\psib},\grave{t})}_r.
$$
Combining \eqref{T0 lip}, \eqref{G bound}, and \eqref{G lip} produces $L_{012,r} > 0$ such that
\be\label{F lip}
\sup_{0 < \ep < \ep_0} \norm{\bF_{\ep}(\psib,t) - \bF_{\ep}(\grave{\psib},\grave{t})}_r  \le L_{012,r}( \norm{\psib}_r^2+ \norm{\grave{\psib}}_r^2+\norm{\psib}_r + \norm{\grave{\psib}}_r)(\norm{\psib-\grave{\psib}}_{r-1} + |t-\grave{t}|).
\ee
Taking $r = 2$ and assuming $\norm{\psib}_2,\norm{\grave{\psib}}_2 \le 1$, we find $L_2 > 0$ such that 
$$
\sup_{0 < \ep < \ep_0} \norm{\bF_{\ep}(\psib,t) - \bF_{\ep}(\grave{\psib},\grave{t})}_2  \le L_2(\norm{\psib-\grave{\psib}}_2 + |t-\grave{t}|).
$$
This together with \eqref{first two comps} proves the second estimate \eqref{bound2} of Lemma \ref{fixed point lemma} for $\Psi_{\ep,1}$ and $\Psi_{\ep,2}$.

Now we proceed to study $\Psi_{\ep,3}$.  Set
$$
T_4(t) = \begin{bmatrix*}
0 &0 \\
0 &\Pi_1
\end{bmatrix*}
$$
and keep $T_1$ and $T_2$ as in \eqref{psi T ops}. Here, however, we will only care about the case $r=2$. Using the general bound
\be\label{ft est}
|\Fo[f](k)| \le \norm{f}_{L_{\per}^{\infty}} \le C_{\text{sob},2}\norm{f}_{H_{\per}^2},
\ee
we find
\begin{align*}
|\Psi_{\ep,3}(\psib,t,a)| &\le \frac{\ep_0C_{R \max}}{l_0}t^2 + \frac{2C_{\text{sob},2}\ep_0|a|}{l_0}\norm{\Pi_1\lambda_+^{\ep(K_{\ep}+t)}b_2^{\ep(K_{\ep}+t)}(\nub+\psib,\nub+\psib)}_{H_{\per}^2} \\
\\
&\le  \frac{\ep_0C_{R \max}}{l_0}t^2 + \frac{2C_{\text{sob},2}\ep_0|a|}{l_0}\norm{T_4\bG_{\ep}(\psib,t)}_2 \\
\\
&\le  \frac{\ep_0C_{R \max}}{l_0}t^2 + \frac{2C_{\text{sob},r}\norm{T_4}_{\bdop(\H^r)}M_{12}\ep_0|a|}{l_0}(\norm{\psib}_2^2 + \norm{\psib}_2 + 1) \text{ by \eqref{G bound}}.
\end{align*}
Thorough rearrangement of this last line, as well as the assumption $\norm{\psib}_2, |a| \le 1$, produces a constant $M_3 > 0$ such that 
$$
\sup_{0 < \ep < \ep_0} |\Psib_{\ep,3}(\psib,t,a)| \le L_3(|a| + \norm{\psib}_2 + t^2), \ \norm{\psib}_2,|t|,|a| \le 1
$$
and this is sufficient to obtain the estimate \eqref{bound1} of Lemma \ref{fixed point lemma} for $\Psi_{\ep,3}$.  

The proof of estimate \eqref{bound2} for $\Psi_{\ep,3}$ is similar to the work above; we omit the details but mention that it uses the Fourier transform estimate \eqref{ft est}, the uniform bounds on the functions $R_{\ep}$ from Lemma \ref{xi decomp}, and the Lipschitz estimate \eqref{G lip} for the functions $\bG_{\ep}$.

Last, the final bound \eqref{lip bound} of Lemma \ref{fixed point lemma} is easily established for the components $\Psi_{\ep,i}$ using the uniform bounds on the operators $T_i^{\ep}$ developed above; again, we omit the details.
\end{proof}

Lemma \ref{fixed point lemma} thus provides a number $a_0 > 0$ and, for all $0 < \ep < \ep_0$ and $|a|\le a_0$, a unique pair $(\psib_{\ep}^a,t_{\ep}^a) \in \{(\psib,t) \in \X  :  \norm{(\psib,t)}_{\X} \le 1\}$ such that $\Psib_{\ep}(\psib_{\ep}^a,t_{\ep}^a,a) = (\psib_{\ep}^a,t_{\ep}^a)$.  We may reverse each step of the conversion in Section \ref{conversion} and we recall the scaling \eqref{bifur param scaling} and the decomposition \eqref{ortho decomp} to find that the function
$$
\thetab(X) := a\varphib_{\ep}^a(X) := a\nub(\ep({K_{\ep}}+t_{\ep}^a)X) + a\psib(\ep({K_{\ep}}+t_{\ep}^a)X)
$$
solves \eqref{phi eqn small}.  Defining $K_{\ep}^a := {K_{\ep}} + t_{\ep}^a$, we have the maps \eqref{maps} and property (i) of Theorem \ref{periodic solutions}.  We prove the rest of the theorem below.

\begin{proof} (of Theorem \ref{periodic solutions}, Parts (ii), (iii), (iv) and (v)) When $a = 0$, the fixed-point property of $(\psi_{\ep,1}^0,\psi_{\ep,2}^0,t_{\ep}^0)$ and the definition of $\Psib_{\ep}$ give
\be\label{also (iii)}
(\psi_{1,\ep}^0,\psi_{2,\ep}^0,t_{\ep}^0) = \Psib_{\ep}(\psib_{\ep}^0,t_{\ep}^0,0) = \left(0,0,-\frac{\ep}{\txi_{\cep}'(\ep{K_{\ep}})}R_{\ep}(\ep{t}_{\ep}^0)(t_{\ep}^0)^2\right).
\ee
We see immediately that $\psi_{\ep,1}^0 = \psi_{\ep,2}^0 = 0$, which is Part (iii), and also
$$
t_{\ep}^0 = -\frac{\ep}{\txi_{\cep}'(\ep{K_{\ep}})}R_{\ep}(\ep{t}_{\ep}^0)(t_{\ep}^0)^2.
$$
Scaling both sides by $\ep$ and rearranging, we find
$$
0 = \tilde{\xi}_{\cep}'(\ep{K_{\ep}})(\ep{t}_{\ep}^0) + R_{\ep}(\ep t_{\ep}^0)(\ep t_{\ep}^0)^2 = \tilde{\xi}_{\cep}(\ep{K_{\ep}}+\ep t_{\ep}^0)  
$$
by \eqref{xi 1}.  We may assume that we have taken $\ep_0$ to be so small that $\ep{K_{\ep}} + \ep t > 0$ for any $0 < \ep < \ep_0$ and $|t| \le 1$, thus $\ep{K_{\ep}} + \ep t_{\ep}^0 > 0$.  By the uniqueness of positive roots of $\txi_{\cep}$ given in Part (vi) of Lemma \ref{lambda lemma}, we have $\ep{K_{\ep}} + \ep t_{\ep}^0 = \ep{K_{\ep}}$, hence $t_{\ep}^0 = 0$ and $K_{\ep}^0 = K_{\ep}$.  So, Part (ii) holds.

For Part (iv), since $\psib_{\ep}^a \in \ortho$, we know $\hat{\psib_{\ep}^a}(\pm1)\cdot\jb = 0$, thus
$$
\int_{-\pi}^{\pi} \psi_{\ep,2}^a(y)\sin(y) \ dy = \sum_{k \in \Z} \hat{\psi_{\ep,2}^a}(k)\hat{\sin}(k) = 0.
$$

Last, for Part (v), by \eqref{size of kc} in Lemma \ref{lambda lemma} we have positive constants $m_*(w)$ and $m^*(w)$, depending only on $w$, such that
$$
m_*(w) \le \ep K_{\ep} \le m^*(w), \ 0 < \ep < 1.
$$
This shows $K_{\ep} = \O(1/\ep)$ and also allows us to estimate
\be\label{ep K bound}
|\ep K_{\ep}^a| = |\ep K_{\ep} + \ep t_{\ep}^a| \le m^*(w) + 1.
\ee

Next, relying on the notation of the proof of Proposition \ref{fp app}, when $r=2$ we have
$$
\sup_{\substack{0 < \ep < \ep_0 \\ |a| \le a_0}} \norm{\psib_{\ep}^a}_2 \le r_0 \le 1
$$
by Lemma \ref{fixed point lemma}, and when $r > 2$, \eqref{bootstrap bound} implies the bootstrap estimate
$$
\norm{\psib_{\ep}^a}_r = \normlarge{\begin{pmatrix*}\Psi_{\ep,1}(\psib_{\ep}^a,t_{\ep}^a,a) \\ 
\Psi_{\ep,2}(\psib_{\ep}^a,t_{\ep}^a,a)
\end{pmatrix*}}_r
\le M_{012}(\norm{\psib_{\ep}^a}_{r-1}^2 + \norm{\psib_{\ep}^a}_{r-1}+1).
$$
We induct on $r$, bound $|\ep K_{\ep}^a|$ by \eqref{ep K bound}, and use the Sobolev embedding theorem to produce \eqref{periodic bound}.

For the uniform Lipschitz bound \eqref{periodic lip}, we first apply the uniform Lipschitz condition \eqref{fp lip concl} guaranteed by Lemma \ref{fixed point lemma} to the fixed points $(\psib_{\ep}^a,t_{\ep}^a)$ and compute
\be\label{r=2 lip}
|K_{\ep}^a - K_{\ep}^{\grave{a}}| + \norm{\psib_{\ep}^a - \psib_{\ep}^{\grave{a}}}_{2} \le |t_{\ep}^a - t_{\ep}^{\grave{a}}| +  \norm{\psib_{\ep}^a - \psib_{\ep}^{\grave{a}}}_{2} \le L_1|a-\grave{a}|
\ee
for some $L_1 > 0$.  For $r > 2$, the estimate \eqref{F lip} gives
\begin{align*}
\norm{\psib_{\ep}^a-\psib_{\ep}^{\grave{a}}}_{r} &\le \norm{\bF_{\ep}(\psib_{\ep}^a,t_{\ep}^a)-\bF_{\ep}(\psib_{\ep}^{\grave{a}},t_{\ep}^{\grave{a}})}_{r} \\
&\le L_{012,r}(\norm{\psib_{\ep}^a}_r^2+\norm{\grave{\psib}_{\ep}^a}_r^2+\norm{\psib_{\ep}^a}_r + \norm{\grave{\psib}_{\ep}^a}_r)(\norm{\psib_{\ep}^a-\grave{\psib}_{\ep}^a}_{r-1} + |t_{\ep}^a-\grave{t}_{\ep}^a|)
\end{align*}
for each $0 < \ep < \ep_0$.  We bound the factor
$$
\norm{\psib_{\ep}^a}_r^2+\norm{\grave{\psib}_{\ep}^a}_r^2+\norm{\psib_{\ep}^a}_r + \norm{\grave{\psib}_{\ep}^a}_r
$$
by \eqref{periodic bound} and estimate $|t_{\ep}^a-t_{\ep}^{\grave{a}}| \le L_1|a-\grave{a}|$ as before to find
$$
\norm{\psib_{\ep}^a-\psib_{\ep}^{\grave{a}}}_r \le L_r(\norm{\psib_{\ep}^a-\psib_{\ep}^{\grave{a}}}_{r-1} + |a-\grave{a}|)
$$
for some $L_r > 0$ and all $0 < \ep < \ep_0$.  Taking the existing Lipschitz estimate on $|K_{\ep}^a-K_{\ep}^{\grave{a}}|$ from \eqref{r=2 lip}, using the Sobolev embedding theorem, and inducting on $r$ produces the final Lipschitz estimate \eqref{periodic lip} of Part (v).
\end{proof}

\subsection{Proof of Lemma \ref{fixed point lemma}}\label{fixed point proof}
We set 
$$
r_0 = \min\left\{\frac{1}{6C_1},r_1\right\} \mand a_0 = \min\left\{\frac{r_0}{6C_1}, \frac{1}{6C_1}, a_1\right\}.
$$
Then whenever $0 < \ep < \ep_0, \norm{x} \le r_0, |a| \le a_0$, we have
$$
\norm{F_{\ep}(x,a)} \le C_1(a_0 + a_0r_0 + r_0^2)  \le C_1\left(\frac{r_0}{6C_1} + \frac{r_0}{6C_1} + \frac{r_0}{6C_1}\right) = \frac{r_0}{2} < r_0.
$$
Moreover, 
$$
C_1(|a| + \norm{x}+\norm{y}) \le C_1\left(\frac{1}{6C_1} + \frac{1}{6C_1}+\frac{1}{6C_1}\right) = \frac{1}{2}
$$
whenever $|a| \le a_0, \norm{x},\norm{y} \le r_0$.  So, \eqref{bound2} gives
\be\label{fp contr}
\norm{F_{\ep}(x,a)-F_{\ep}(y,a)} \le \frac{1}{2}\norm{x-y}
\ee
for all such $a,x,y$.  Thus have the uniform contraction condition.

We conclude that for each $0 < \ep < \ep_0$ and $|a| \le a_0$, $F_{\ep}(\cdot,a)$ maps $\ball(r_0)$ into itself and is a contraction (with uniform constant 1/2).  By Banach's fixed point theorem, for each $0 < \ep < \ep_0$ and $|a| \le a_0$, we then have a unique $x_{\ep}^a \in \ball(r_0)$ such that $F_{\ep}(x_{\ep}^a,a) = x_{\ep}^a$.

For the Lipschitz estimate on the mappings $a \mapsto x_{\ep}^a$, compute, for $|a| \le a_0$,
\begin{align*}
\norm{x_{\ep}^a - x_{\ep}^{\grave{a}}} &= \norm{F_{\ep}(x_{\ep}^a,a) - F_{\ep}(x_{\ep}^{\grave{a}},\grave{a})} \\
\\
&\le \norm{F_{\ep}(x_{\ep}^a,a)-F_{\ep}(x_{\ep}^a,\grave{a})} + \norm{F_{\ep}(x_{\ep}^a,\grave{a})-F_{\ep}(x_{\ep}^{\grave{a}},\grave{a})} \\
\\
&\le L_1|a-\grave{a}| + \frac{1}{2}\norm{x_{\ep}^a-x_{\ep}^{\grave{a}}} \text{ by \eqref{lip bound} and \eqref{fp contr}}.
\end{align*}
Hence
$$
\norm{x_{\ep}^a-x_{\ep}^{\grave{a}}} \le 2L_1|a-\grave{a}|
$$
for all $|a| \le a_0$ and $0 < \ep < \ep_0$.

\section{The nanopteron equations}\label{N}

\subsection{Beale's ansatz}
Following \cite{beale2}, we let 
$$\etab(X):=\left(\begin{array}{c} \eta_1(X) \\ \eta_2(X) \end{array} \right)$$
and look for a solution of \eqref{theta eqn small} of the form
\be\label{beale}
\thetab=\sigmab+a \varphib^a_\ep +  \etab.
\ee
In the above there are three unknowns: 
\begin{itemize}
\item
the function $\eta_1$ (which will be an even
exponentially decaying function),\item the function $\eta_2$ (which will be an odd exponentially decaying function) and \item the amplitude of the periodic part, $a \in \R$.
\end{itemize}

\begin{remark}Since $\sigmab=\sigma \ib$ and $\varphib_\ep^0 = \sin(K_\ep X) \jb$, we see that the principal contribution in the first slot is 
connected to the acoustic branch and to the optical branch in the second slot, as described above.\end{remark}

One finds that $\etab$ solves the system:
\be\label{eta eqn 1}\begin{split}
\eta_1 + {4(1+w)} \varpi^0 \left( \sigma \eta_1\right)&= j_1+j_2 + j_3 + j_4+j_5\\
\ep^2 (c_w^2+\ep^2) \eta_2'' + \lambda^\ep_+\eta_2 &= \ep^2 (l_1 + l_2 + l_3 + l_4 + l_5)
\end{split}\ee
where
\begin{align*}
j_1&:=-\left( \sigma+\varpi^\ep b_1^\ep(\sigmab,\sigmab)\right)                                      &l_1:=- \lambda_+^\ep b_2^\ep(\sigmab,\sigmab)\\
j_2&:=-\left( 2\varpi^\ep b_1^\ep(\sigmab,\etab)- 2 \varpi^0 b_1^0(\sigmab,\etab)\right) &l_2:=-2 \lambda_+^\ep  b_2^\ep(\sigmab,\etab)\\
j_3&:=-2a\varpi^\ep b_1^\ep(\sigmab,\varphib^a_\ep)                                                     &l_3:=-2a \lambda_+^\ep b_2^\ep(\sigmab,\varphib^a_\ep)\\
j_4&:=-2a \varpi^\ep b_1^\ep(\etab,\varphib^a_\ep)                                                         &l_4:=-2a \lambda_+^\ep b_2^\ep(\etab,\varphib^a_\ep)\\
j_5&:=-2\varpi^\ep b_1^\ep(\etab,\etab)                                                                           &l_5:=-\lambda_+^\ep b_2^\ep(\etab,\etab).
\end{align*}
We used the fact that
$
2 \varpi^0 b_1^0 (\sigmab,\etab) = {4(1+w)} \varpi^0 \left( \sigma \eta_1\right).
$

The operator
\be\label{A def}
\A := 1 +{4(1+w)} \varpi^0 \left( \sigma \cdot\right)
\ee
was studied in \cite{friesecke-pego1} and is invertible the class of even functions. This is made precise below in Theorem \ref{A inv}.
Thus we can rewrite the first equation in \eqref{eta eqn 1} as
\be\label{N1 eqn}
\eta_1 = \A^{-1}\left( j_1 + j_2+j_3+j_4+j_5\right)=:N_1^\ep(\etab,a).
\ee

\subsection{The solvability condition of $\B_\ep$}
On the other hand
\be\label{Bep def}
\B_\ep:=\ep^2 (c_w^2+\ep^2) \partial_X^2 + \lambda^\ep_+
\ee
is not so nice.  
If we take the Fourier transform of the equation \be\label{Bep}
\B_\ep f = g\ee we find
that
\be\label{Bepfo}
\tilde{\B}_\ep(K)  \hat{f}(K) = \hat{g}(K)
\ee
where $\tilde{\B}_\ep(K)=-\ep^2 (c_w^2 + \ep^2) K^2 + \tilde{\lambda}_+(\ep K).$

In \eqref{Kep} in 
Theorem \ref{periodic solutions}, we saw that there exists a unique $K_\ep>0$ such that
$
\tilde{\B}_\ep(\pm K_\ep)=0.
$
Also
 $K_\ep = \O(1/\ep)$.
%
Since we have $\tilde{\B}_\ep(\pm K_\ep)=0$ we see, by virtue of \eqref{Bepfo}, that
\be\label{solve 1}
\B_\ep f = g \implies 
\hat{g}(\pm K_\ep) =0.
\ee
Which is to say that $\B_\ep$ in not surjective. (It is injective.) 
The appropriate way to view \eqref{solve 1} is as a pair of solvability conditions for \eqref{Bep};
it turns out that if the integral conditions are met then there is a solution $f$ of $\B_\ep f = g$. In this case we write $f = \B_\ep^{-1} g$. This is made precise below in Lemma \ref{B inv}.

Note that if $f$ is odd, so is $g = \B_\ep f$. And therefore so is $\hat{g}(K)$.
Which means that we can eliminate one of the solvability conditions in \eqref{solve 1}.
In particular, if $f$ is odd then
\be\label{solve 2}
\B_\ep f = g \implies
\iota_\ep[g]:=\int_\R g(X) \sin(K_\ep X) dX = 0.
\ee

\subsection{The modified equation for $\eta_2$ and an equation for $a$}
Thus \eqref{solve 2} implies any solution of \eqref{eta eqn 1} has
\be\label{solve 3}
\iota_\ep[l_1 + l_2 + l_3 + l_4 + l_5] = 0.
\ee
Following \cite{beale2} and \cite{amick-toland}, we will use this condition to ``select the amplitude $a$."

Toward this end, we let
$$
\chi_\ep(X):=\lambda_+^\ep J_1^\ep \left( J_2^0 \sigmab. J_2^\ep \nub_\ep \right)\cdot \jb\quad \text{where}\quad
\nub_\ep(X):= \varphib^0_\ep(X)= \sin(K_\ep X) \jb.
$$
We claim that
$$
l_{31}:=l_3 + 2a  \chi_\ep
$$
is ``small", though we hold off on a precise estimate for the time being. Roughly what we mean is that $l_{31}$ contains terms which are either of size comparable to $\ep$, or are quadratic in~$a$.
We also claim that
$$
\kappa_\ep := \iota_\ep[\chi_\ep]
$$
is large in the sense that it is strictly bounded away from zero by an amount that does not depend on $\ep$. Both these claims are verified below (in \eqref{l31 est} and \eqref{key kappa estimate}).
With this definition we can rewrite \eqref{solve 3}  as
\be\label{solve 4}
a ={1 \over2 \kappa_\ep} \iota_\ep[l_1 + l_2 + l_{31} + l_4 + l_5]=: N^\ep_3(\etab,a).
\ee

Next we modify the second equation in \eqref{eta eqn 1} to 
\be\begin{split}\label{mod 1}
\B_\ep \eta_2 = &-2\ep^2 a \chi_\ep + \ep^2 \left(l_1 + l_2 + l_{31} + l_4 + l_5\right) 
\\&-{1 \over \kappa_\ep} 
\iota_\ep \left[-2\ep^2 a \chi_\ep + \ep^2 \left(l_1 + l_2 + l_{31} + l_4 + l_5) \right)\right] \chi_\ep.
\end{split}\ee
By design,
$$\iota_\ep \left[-2\ep^2 a \chi_\ep + \ep^2 \left(l_1 + l_2 + l_{31} + l_4 + l_5\right)-{1 \over \kappa_\ep} 
  \iota_\ep \left[-2\ep^2 a \chi_\ep + \ep^2 \left(l_1 + l_2 + l_{31} + l_4 + l_5) \right)\right] \chi_\ep
\right]=0.
$$
Which is to say that the right hand side of \eqref{mod 1} meets the solvability condition \eqref{solve 2} and we can apply $\B_\ep^{-1}$ to it.
Also, if \eqref{solve 4} is met then the term in the second row of \eqref{mod 1} vanishes and so 
 the right hand side of \eqref{mod 1} agrees exactly with the right hand side of the second equation in \eqref{eta eqn 1}.
Also note that
$\ds
2\ep^2 a \chi_\ep-{1 \over \kappa_\ep} \iota_\ep[2\ep^2 a \chi_\ep ]\chi_\ep=0.
$

So if we put
$$
\P_\ep f := \B_\ep^{-1}\left( f - {1 \over \kappa_\ep} \iota[f] \chi_\ep \right)
$$
then \eqref{mod 1} is equivalent to
\be\begin{split}\label{mod 2}
 \eta_2 &=\ep^2 \P_\ep \left(l_1 + l_2 + l_{31} + l_4 + l_5\right) =:N^\ep_2(\etab,a).\end{split}\ee
\begin{remark}\label{abstract remark}
Stating things more abstractly, what we know is  that the cokernel of $\diag(\A,\B_\ep)$ is nontrivial, due to the 
solvability conditions \eqref{solve 1}. The classical method for the analysis of nonlinear problems where
the cokernel (or, more typically, the kernel) of the linear part is nontrivial is
the Liapunov-Schmidt decomposition, like we used in the construction of the periodic solutions. But in our case we have the additional complication that $\diag(\A,\B_\ep)$ is injective. Which means that its Fredholm index is negative. It is this feature that  results in having different pieces of our problem living in different sorts of function spaces (namely localized and periodic) as opposed to the whole argument taking place in $E^1_q \times O^1_q$. 

Another less precise, but perhaps more evocative way, of saying this is to say that we want to do a regular old Liapunov-Schmidt analysis but the function we want to be the basis for the kernel of $\B_\ep$---specifically $\sin(K_\ep X)$---is not in our function space. And so we need to come up with a way to include periodic functions in the solution at the same time as the localized functions. Which leads us to Beale's ansatz \eqref{beale}. At the end of the day,  the equation for
the periodic amplitude $a$, \eqref{solve 4}, can viewed as being the replacement for
``the projection of \eqref{eta eqn 1} onto the kernel" which would appear in a more standard Liapunov-Schmidt analysis.
Likewise
$\eta_1$, \eqref{N1 eqn} and, more relevantly, $\eta_2$, \eqref{mod 2} are the replacements
for ``the projection onto the orthogonal complement of the kernel."
\end{remark}

\subsection{The final system}
In short, if we can solve the system
\be\label{big sys}
\begin{split}
\eta_1 &= N_1^\ep(\eta_1,\eta_2,a) \\
\eta_2 &= N_2^\ep(\eta_1,\eta_2,a) \\
a & = N_3^\ep(\eta_1,\eta_2,a)
\end{split}
\ee
then we will have found a solution of our problem. 
Observe that \eqref{big sys} is written such that solutions are fixed points of the map $N^\ep:=(N^\ep_1,N^\ep_2,N^\ep_3)$.
We would achieve our goal if we could show that $N^\ep$ is a contraction on a suitable function space. It turns out that the right hand side has some problems
in that regard, due principally to the terms $j_4$ and $l_4$. These have a Lipschitz constant with respect to $a$ that depends
in a bad way on $\etab$. Nevertheless, a modified contraction mapping
argument will get the job done. 
But first we need many estimates.

\section{Existence/uniqueness/regularity/magnitude}\label{EU}
\subsection{Function spaces}
For $r \ge 0$ and $p \in [1,\infty]$, let
$W^{r,p}(\R)$ be the usual Sobolev space of $r$-times (weakly) differentiable functions 
in $L^p(\R)$. The norms on these spaces will be denoted by $\|\cdot \|_{W^{r,p}}$.
Put $H^r(\R) := W^{r,2}(\R)$, per the usual convention. 

For $r \ge 0$ and $q\ge0$, let
$$
H^r_q:=\left \{ f  \in H^r(\R) :  \cosh( q \cdot) f(\cdot) \in H^r(\R)  \right\}.
$$
$H^r_q$ consists of those functions in $H^r(\R)$ which, roughly speaking, behave like $e^{-q|\cdot|}$ as $|\cdot| \to \infty$.
Let $$E^r_q:=H^r_q \cap \left\{ \text{even functions}\right\}\mand O^r_q:=H^r_q\cap \left\{ \text{odd functions}\right\}.$$
Each of these is a Hilbert space with inner product given by
$$
(f,g)_{r,q}:= ( \cosh(q \cdot) f, \cosh(q \cdot) g)_{H^r(\R)}
$$
where $(\cdot,\cdot)_{H^r(\R)}$ is the usual $H^r(\R)$ inner product.
Of course we denote $\| f\|_{r,q}:=\sqrt{ (f,f)}_{r,q}$. We abuse notation slightly and, for elements $\ub$ of $H^r_q \times H^r_q$, write $\| \ub \|_{H^r_q \times H^r_q} = \| \ub \|_{r,q}$.
We will show that \eqref{big sys} has a solution in $E^1_q \times O^1_q\times \R$ for some $q > 0$.

\subsection{Key estimates}
As mentioned above, the existence proof is an iterative argument modeled on the proof of Banach's contraction mapping theorem. The following proposition 
contains all the necessary estimates for proving existence and uniqueness. It also contains estimates
which will be used in a bootstrap argument which will show that the solution is smooth and, more interestingly, that the amplitude of the periodic
piece ``$a$" is small beyond all orders of $\ep$. 
\begin{proposition}\label{main mover}
For all $w >1$ there exists $\ep_\star\in(0,1)$, $q_\star>0$ and $C_\star>1$ such that we have the following 
properties.
\begin{enumerate}[ (i) ]
\item ({\bf  Mapping estimates})
For all 
$$
\etab \in E^1_q \times O^1_q,\quad 0 < \ep \le\ep_\star,\quad {1 \over 2}q_\star \le q \le q_\star \mand -a_0 \le a  \le a_0
$$
we have $N_1^\ep(\etab,a) \in E^1_q$ and  $N_2^\ep(\etab,a) \in O^1_q$ 
together with the estimate:
\be\label{ball1}
\|N_1^\ep(\etab,a)\|_{1,q}+\|N_2^\ep(\etab,a)\|_{1,q}+\left\vert N_3^\ep(\etab,a)\right\vert 
\le C_\star\left(\ep + \ep\|\etab\|_{1,q}+\ep |a| +\|\etab\|_{1,q}^2  + a^2\right).
\ee
\item ({\bf  Lipschitz-type estimates})
For all 
$$
\etab,\grave{\etab} \in E^1_{q'} \times O^1_{q'}, \quad
0 < \ep \le\ep_\star,\quad {1 \over 2}q_\star \le q < q' \le q_\star \mand -a_0 \le a \le \ga \le a_0
$$
we have
\be\label{lip1}\begin{split}
   &\|N_1^\ep(\etab,a)-N_1^\ep(\grave{\etab},\ga)\|_{1,q}
+\|N_2^\ep(\etab,a)-N_2^\ep(\grave{\etab},\ga)\|_{1,q}
+ 
\left \vert N_3^\ep(\etab,a)-N_3^\ep(\grave{\etab},\ga)\right \vert\\
\le &{C_\star \over |q-q'|} 
\left(\ep + \|{\etab}\|_{1,q'}+\|\grave{\etab}\|_{1,q'}
+|a|+|\ga|\right)(|a-\ga|+\|\etab-\grave{\etab}\|_{1,q}).
\end{split}\ee
\item ({\bf Bootstrap estimates})
For all $r \ge 1$ there exists $C_{\star,r}>1$ such that
for all 
$$
\etab \in E^r_q \times O^r_q,\quad 0 < \ep \le\ep_\star,\quad {1 \over 2}q_\star \le q \le q_\star \mand -a_0 \le a  \le a_0
$$
we have $N_1^\ep(\etab,a) \in E^{r+1}_q$ and  $N_2^\ep(\etab,a) \in O^{r+1}_q$ 
together with the estimates:
\begin{multline}\label{boot1}
\|N_1^\ep(\etab,a)\|_{r+1,q}+
\|N_2^\ep(\etab,a))\|_{r+1,q}\\
\le C_{\star,r} \left( \ep + \|\etab\|_{r,q} + \ep^{1-r} |a|   +\ep^{-r}a^2+\ep^{-r}|a|\|\etab\|_{r,q}+\|\etab\|_{r,q}^2 \right)
\end{multline}
and
\be\label{boot2}
\left \vert N_3^\ep(\etab,a) \right \vert \le C_{\star,r}\left( \ep^{r+1} +\ep^r \|\etab\|_{r,q} + \ep|a| + a^2 +|a|\|\etab\|_{r,q} + \ep^{r}\|\etab\|_{r,q}^2\right).
\ee
\end{enumerate} 
\end{proposition}
The proof of this proposition is lengthy, byzantine and postponed to until Sections \ref{BE} and \ref{NE} below.
Onward to existence.

\subsection{Existence}
Let $$\mathcal{X}_q:=E^1_q \times O^1_q \times \R.$$ 
This is a Banach space with
norm $\| \cdot \|_{\X_q}$ defined in the obvious way.

Fix $w > 1$ and take $\ep_\star$ and $q_\star$ as in Proposition \ref{main mover}.
If we put $\nb=(\etab,a)$ and $\grave{\nb}=(\grave{\etab},\ga)$ then the  estimate \eqref{ball1} is compressed to
\be\label{ball2}
\| N^\ep(\nb)\|_{\X_q} \le C_\star(\ep + \ep \|\nb\|_{\X_q} + \|\nb\|^2_{\X_q}).
\ee
Similarly, \eqref{lip1} implies
\be\label{lip2}
\| N^\ep(\nb)-N^\ep(\grave{\nb})\|_{\X_q} \le {C_\star \over |q-q'|}(\ep + \ep \|\nb\|_{\X_{q'}} + \|\grave{\nb}\|_{\X_{q'}})\| \nb - \grave{\nb}\|_{\X_{q}}.
\ee
Here we have the same restrictions on $q,q',\ep$ as in the proposition, of course.

Put
\bes
\bar{\ep}:=\min\left(\ep_\star,{1 \over 2(C_\star  +2C_\star  ^2)} , {q_\star \over 8(C_\star  +4 C^2_*)}\right).
\ees
Henceforth we assume that $\ep \in (0,\bar{\ep}]$.
Suppose that 
\be\label{the lid}\| \nb\|_{q_\star} \le 2C_\star  \ep.\ee
 Then \eqref{ball2}, \eqref{the lid} and the definition of $\bar{\ep}$ imply
\be\label{big estimate}\begin{split}
\|N^\ep(\nb)\|_{\X_{q_\star}} 
\le&C_\star  \left(\ep + 2C_\star   \ep^2+ 4C_\star  ^2\ep^2\right)
\le C_\star  \ep \left(1 + (2C_\star   + 4C_\star  ^2)\bar{\ep}\right)
\le 2C_\star   \ep.
\end{split}\ee

Now select $\nb^1 \in \X_{q_\star}$ with $\|\nb^1\|_{q_\star} \le 2C_\star  \ep$. For $j \ge 1$, put
\be\label{iterate}
\nb^{j+1} = N^\ep(\nb^{j}).
\ee
A simple induction argument using \eqref{big estimate} shows that, for all $j \in \N$, we have
\be\label{uniform bound}
\| \nb^{j+1}\|_{\X_{q_\star}} \le 2C_\star   \ep.
\ee
Thus we see that $\left\{\nb^{j} \right\}_{j \in \N}$ is a uniformly bounded sequence in $\X_{q_\star}$ (and therefore uniformly bounded
in all spaces $\X_q$ with $q \le q_\star$ too).

We now demonstrate that this sequence is Cauchy in $\X_{3q_\star/4}$. Fix $j \ge 2$. Then \eqref{iterate} and  \eqref{lip2} (with $q = 3q_\star/4$ and $q' = q_\star$) imply
\bes
\begin{split}
\| \nb^{j+1}-{\nb^j}\|_{\X_{3q_\star/4}} &= \|N^\ep(\nb^{j}) - N^\ep(\nb^{j-1}) \|_{\X_{3q_\star/4}}\\
&\le 4C_\star  q_\star^{-1}\left(\ep+ 
 \| \nb^j\|_{\X_{q_\star}}+\| \nb^{j-1} \|_{\X_{q_\star}}\right) \left\| \nb^j-{\nb}^{j-1} \right\|_{\X_{3q_\star/4}}.
\end{split}
\ees
We use use \eqref{uniform bound} in  the first term to get
\bes
\begin{split}
\| \nb^{j+1}-\nb^{j}\|_{\X_{3q_\star/4}} 
&\le 4C_\star  q_\star^{-1}\left(\ep+ 
4 C_\star   \ep\right) \left\| \nb^j-{\nb}^{j-1} \right\|_{\X_{3q_\star/4}} \\
\end{split}
\ees

Using the fact that $\ep \in (0,\bar{\ep}]$ and the definition of $\bar{\ep}$ we see that
$
4 C_\star  q_\star^{-1}\left(\ep+ 
4 C_\star   \ep\right) \le {1 /2}.
$
Thus
\bes
\begin{split}
\| \nb^{j+1}-\nb^{j}\|_{\X_{3q_\star/4}} 
&\le {1 \over 2}\left\| \nb^j-{\nb}^{j-1} \right\|_{\X_{3q_\star/4}} \quad \text{for $j \ge 2$}.
\end{split}
\ees
Also, \eqref{uniform bound} and the triangle inequality give:
$
\| \nb^{2}-\nb^{1}\|_{\X_{3q_\star/4}} \le 4 C_\star   \ep.
$
A classic induction argument then shows that
\be\label{small}
\| \nb^{j+1}-\nb^{j}\|_{\X_{3q_\star/4}} \le 8 C_\star   \ep 2^{-j}.
\ee
for all $j \ge 1$. 

Now fix $m >n\ge1$. The triangle inequality, followed by \eqref{small} and the geometric series summation formula give:
$$
\| \nb^{m}-\nb^{n}\|_{\X_{3q_\star/4}}
\le \sum_{j=n}^{m-1}\| \nb^{j+1}-\nb^{j}\|_{\X_{3q_\star/4}} \le 8 C_\star   \ep\sum_{j=n}^{m-1}  2^{-j}\le 8 C_\star   \ep\sum_{j=n}^{\infty}  2^{-j}= {16 C_\star \ep  \over 2^n}.
$$
Thus we can make $\| \nb^{m}-\nb^{n}\|_{\X_{3q_\star/4}}$ as small as we like by taken $m,n$ sufficiently large, which means
that the sequence is Cauchy. Which means it converges. Call the limit
$
\nb_\ep =(\etab_\ep,a_\ep) \in \X_{3q_\star/4}.
$
Because of \eqref{uniform bound}, we have
\be\label{solution is small}
\| \nb_\ep\|_{\X_{3q_\star/4}} \le 2C_\star \ep.
\ee

Now we claim that \be\label{win}
\nb_\ep = N^\ep(\nb_\ep)
\ee
which would imply that $\nb_\ep$ is the solution we are looking for.
Since the convergence of $\nb^j$ 
is in $\X_{3q_\star/4}$, if we knew that $N^\ep$ was continuous on that space we would have our claim by passing the limit 
through $N^\ep$ in \eqref{iterate}. But $N^\ep$ is not
obviously continuous. One can see this in the fact that the Lipschitz constant in  \eqref{lip2} depends on $\| \nb \|_{\X_{q'}}$ with $q'>q$.
We do know that $N^\ep(\nb_\ep) \in \X_{3q_\star/4}$ by virtue of \eqref{ball2}.

But nonetheless we have \eqref{win}. Since $\nb^j$ converges in $\X_{3q_\star/4}$, the scheme \eqref{iterate} implies
\be\label{converge}
N^\ep(\nb^j) \to \nb_\ep
\ee
too. This convergence takes place in $\X_q$ for all $q \in [0,3q_\star/4]$.
So look at
$$
\| N^\ep(\nb_\ep) - \nb_\ep\|_{\X_{q_\star/2}}.
$$
Note that we are estimating this in the bigger space $\X_{q_\star/2}$, not $\X_{3q_\star/4}$. The triangle inequality shows that
$$
\| N^\ep(\nb_\ep) - \nb_\ep\|_{\X_{q_\star/2}} \le \| N^\ep(\nb_\ep) - N^\ep(\nb^j) \|_{\X_{q_\star/2}} + \| N^\ep(\nb^j) - \nb_\ep\|_{ \X_{q_\star/2}}.
$$
The second term can be made as small as we like by taking $j$ big enough because of \eqref{converge}. 
For the first term we use \eqref{lip2}:
$$
\| N^\ep(\nb_\ep) - N^\ep(\nb^j) \|_{\X_{q_\star/2}}  \le 4 C_* q_\star^{-1} \left(\ep + \|\nb_\ep\|_{\X_{3q_\star/4}} +  \|\nb^j\|_{\X_{3q_\star/4}}   \right)\|\nb_\ep - \nb^j\|_{\X_{q_\star/2}}.
$$
Using \eqref{uniform bound} and \eqref{solution is small} this becomes:
$$
\| N^\ep(\nb_\ep) - N^\ep(\nb^j) \|_{\X_{q_\star/2}}  \le 4 C_* q_\star^{-1}\ep  \left(1 + 4C_\star \ep  \right)\|\nb_\ep - \nb^j\|_{\X_{q_\star/2}}.
$$
Since $\nb^j \to \nb_\ep$ in $\X_{3q_\star/4}$ it also converges in $\X_{q_\star/2}$. 
And thus we can make the above term as small as we want by taking $j$ sufficiently large.
Which is to say that 
$
\| N^\ep(\nb_\ep) - \nb_\ep\|_{\X_{q_\star/2}} =0.
$
Thus we have \eqref{win}. Which is to say, there exists a solution of \eqref{big sys}.

\subsection{Uniqueness}
Suppose that $\grave{\nb}_\ep \in \X_{3q_\star/4}$ has the property that $\grave{\nb} = N^\ep(\grave{\nb}_\ep)$
and $\| \grave{\nb}_\ep \|_{ \X_{3q_\star/4}} \le 2 C_\star   \ep$ and $\ep \in (0,\bar{\ep}]$. Clearly
$$
\gnb_\ep -\nb_\ep = N^\ep(\gnb_\ep) - N^\ep(\nb_\ep).
$$
Applying \eqref{lip2} with $q = q_\star/2$ and $q' = 3q_\star/4$ gives:
$$
\| \gnb_\ep -\nb_\ep \|_{\X_{q_\star/2}} \le 4 C_\star  q_\star^{-1}\left(\ep + 
 \| \gnb_\ep \|_{\X_{3q_\star/4}}+\| \nb_\ep\|_{\X_{3q_\star/4}}\right) \left\| \gnb_\ep - \nb_\ep\right\|_{\X_{q_\star/2}}.
$$
Since $\| \grave{\nb}_\ep \|_{ \X_{3q_\star/4}} \le 2 C_\star   \ep$ and $\| {\nb}_\ep \|_{ \X_{3q_\star/4}} \le 2 C_\star   \ep$
we have
$$
\| \gnb_\ep -\nb_\ep \|_{\X_{q_\star/2}} \le 4 C_\star  q_\star^{-1}\left(\ep + 
4 C_\star  \ep\right)  \left\| \gnb_\ep - \nb_\ep\right\|_{\X_{q_\star/2}}.$$
As above, we saw that $\ep \in (0,\bar{\ep}]$ implies $4 C_\star  q_\star^{-1}\left(\ep + 
4 C_\star  \ep\right) \le 1/2$. Thus we have
$$
\| \gnb_\ep -\nb_\ep \|_{\X_{q_\star/2}}  \le {1 \over 2}\| \gnb_\ep -\nb_\ep \|_{\X_{q_\star/2}} $$
which implies $\grave{\nb}_\ep=\nb_\ep.$ 
And so $\nb_\ep=(\etab_\ep,a_\ep)$ is the unique fixed point of $N^\ep$ in
the ball of radius $2C_\star  \ep$ in $\X_{3q_\star/4}$.

\subsection{Regularity of $\eta_\ep$ and the size of $a_\ep$}
We claim that for 
 all $r \ge 1$, there exists $C_r > 0$ such that for all $\ep \in (0,\bar{\ep}]$ the fixed points $(\etab_\ep,a_\ep)$ constructed above satisfy
\be\label{conclusion}
\| \etab_\ep\|_{r,{3q_\star/4}} \le C_r \ep \mand |a_\ep| \le C_r \ep^r.
\ee

We prove this by induction. The original construction of $(\eta_\ep,a_\ep)$ was done in the ball of radius $2C_\star \ep$ in the space $\X_{{3q_\star/4}}$
and so we have the $r=1$ base  case:
$$
\| \etab_\ep\|_{1,{3q_\star/4}} \le 2C_\star \ep \mand |a_\ep| \le 2 C_\star \ep.
$$

Now suppose that \eqref{conclusion} holds for some  $r\ge1$. We know that
$(\etab_\ep,a_\ep) = N^\ep(\etab_\ep,a_\ep).
$
Therefore, using \eqref{boot1} we see:
\bes\begin{split}
\| \etab_\ep\|_{r+1,3q_\star/4} 
=&\|N_1^\ep(\etab_\ep,a_\ep)\|_{r+1,{3q_\star/4}} +\|N_2^\ep(\etab_\ep,a_\ep)\|_{r+1,{3q_\star/4}} \\ \le &
C_{\star,r}\left( \ep + \|\etab_\ep\|_{r,{3q_\star/4}} + \ep^{1-r} |a_\ep|+ \ep^{-r} a_\ep^2 + \ep^{-r} |a_\ep|\|\etab_\ep\|_{r,{3q_\star/4}} + \|\etab_\ep\|^2_{r,{3q_\star/4}}\right).
\end{split}\ees
Using the inductive hypothesis \eqref{conclusion} gives:
\bes
\| \etab_\ep\|_{r+1,3q_\star/4}  \le
C_{\star,r}\left( \ep  + C_r \ep + C_r \ep + C_r^2 \ep^r + C_r^2 \ep + C_r^2 \ep^2\right) \le C_{r+1} \ep.
\ees
We are half way done.

Using \eqref{boot2}
we have
$$
|a_\ep|=\left \vert N_3^\ep (\etab_\ep,a_\ep) \right \vert \le C_{\star,r} \left( \ep^{r+1}  + \ep^r \|\etab_\ep\|_{r,{3q_\star/4}}
+ \ep|a_\ep| + |a_\ep|^2 + |a_\ep|\|\etab_\ep\|_{r,{3q_\star/4}} + \ep^r \| \etab_\ep\|^2_{r,{3q_\star/4}} \right)
$$
Using the inductive hypothesis \eqref{conclusion} gives:
\bes
|a_\ep|\le
C_{\star,r}\left(\ep^{r+1} + C_r \ep^{r+1} + C_r \ep^{r+1} + C_r^2 \ep^{2r} + C_r^2 \ep^{r+1} + C_r^2 \ep^{r+2} \right) \le C_{r+1} \ep^{r+1}.
\ees
Thus we have established \eqref{conclusion} with for $r+1$ and we are done.

\subsection{The main result}
Summing up, we have proven our main result, stated here in full technicality.
\begin{theorem}\label{main result}
For all $w>1$ there exists $\bar{\ep}>0$ and $\bar{q}>0$ such that the following holds for all $\ep \in (0,\bar{\ep})$.
\begin{enumerate}[(i)]
\item There exists $\ds \etab_\ep \in \cap_{r \ge 0} \left( E^r_{\bar{q}} \times O^r_{\bar{q}} \right)$ and $a_\ep \in [a_0,a_0]$ such that
$$
\thetab(X) = \thetab_\ep(X):= \sigmab(X)+\etab_\ep(X) + a_\ep \varphib^{a_\ep}_\ep(X)
$$
solves \eqref{theta eqn small}.
\item For all $r \ge 0$ there exists $C_r>0$ such that, for all $\ep \in (0,\bar{\ep})$:
$$
\| \etab_\ep\|_{r,\bar{q}} \le C_r \ep \mand |a_\ep| \le C_r \ep^{r}.
$$
\item $\thetab_\ep$ is unique in the sense that $\etab_\ep$ and $a_\ep$ are the only choices for
which $\thetab_\ep$ is a solution of \eqref{theta eqn small} and the estimates in (ii) hold.
\end{enumerate}
\end{theorem}

\begin{remark} The uniqueness result above does not rule out two interesting possibilities.
The first is that there could be a different choice for $\etab$ and $a$ where
$\etab \in E^1_{q} \times O^1_q$ with $q \in[0,\bar{q})$.
That is, $\etab \to 0$ at infinity at a rate slower than $e^{-\bar{q}|X|}$.
We consider this to be unlikely; our conjecture is that the solution is in fact unique in the class of $L^2 \times L^2$ functions.

The other possibility is there are solutions of \eqref{theta eqn small} which converge to $a \varphib_\ep^a (X\pm X_0)$
as $X \to \pm \infty$, for some $X_0 \in \R$. That is to say, the solution $\thetab$ converges to a phase-shifted member of the family of periodic solutions. This almost certainly will happen; the analogous result is shown to be true 
for gravity-capillary waves in \cite{beale2} and \cite{sun} and the singularly perturbed KdV-type equation studied in \cite{amick-toland}. To prove such a result can be achieved (following \cite{amick-toland}) by making an adjustment to Beale's ansatz \eqref{beale}. Specifically, replacing $a\varphib_\ep^a(X)$ with $a \varphib_\ep^a(X + \sgn(X) X_0) \Xi(X)$ where $\Xi(X)$ is a smooth positive $C^\infty$ function which is zero at $X =0$ and exactly equal to one for $|X|$ large. Obviously this generates more than a few extra terms in \eqref{eta eqn 1} and complications in proving estimates down the line! 
\end{remark}

Theorem \ref{main result} implies, after undoing all the changes of variables that led from \eqref{r eqn} to \eqref{theta eqn small}:
\begin{cor} \label{main result descale} For all $w > 1$ there exists $\bar{\ep}>0$ and $\bar{q}$ such that 
the following holds for all $\ep \in (0,\bar{\epsilon})$. Let $c_\ep := \sqrt{c_w^2 + \ep^2}$. There is a solution of \eqref{r eqn}
of the form
$$
r(j,t)= {3 \over 4} \ep^2(1+w) \sech^2\left({\ep \over \sqrt{\alpha_w}}\left(j\pm c_\ep t\right)\right) +  v^\ep_j(\ep(j\pm  c_\ep t))
+ p_j^\ep(j \pm c_\ep t)
$$
where:
\begin{enumerate}[(i)]
\item $v_{j+2}^\ep (X) = v^\ep_{j}(X)
$
and
$
p^\ep_{j+2}(X) =p^\ep_j (X)
$
for all $j \in \Z$ and $X \in \R$.
\item For all $r \ge 0$ we have $\|v_1^\ep\|_{H^r_{\bar{q}}}+ \|v_2^\ep\|_{H^r_{\bar{q}}} \le C_r \ep^3$.
$C_r>0$ depends only on $r$ and $w$ and not on $\ep$.
\item $p_1^\ep$ and $p_2^\ep$ are periodic with period $P_\ep \in I_w$ where $I_w$ 
is a closed bounded subset of $\R^+$. $I_w$ depends only on $w$ and not on $\ep$.
\item For all $r \ge 0$ we have $\|p_1^\ep\|_{W^{r,\infty}}+ \|p_2^\ep\|_{W^{r,\infty}} \le C_r \ep^r$.
$C_r>0$ depends only on $r$ and $w$ and not on $\ep$.
\end{enumerate}
\end{cor}
It is this corollary which is paraphrased nontechnically in Theorem \ref{main result nontech}.


\section{Basic estimates}\label{BE}

\subsection{Estimates on $\sigmab$.}

Since $\sigma(X) = \sigma_0 \sech^2(2 q_0 X)$, for all $r \ge 0$ there exists $C_r>0$ such that
\be\label{sigma bound}
\| \sigma \|_{r,q} = \|\sigmab\|_{r,q} \le C_r
\ee
holds for all $q \in [0,q_0]$.  In fact $\sigma$ is in $H^r_q$ for all $q \in [0,2 q_0)$, but 
by restricting the interval for $q$ we can ensure that the constant $C_r$ does not depend on $q$. 
The constant does depend on $r$, of course.
Obviously it does not depend on $\ep$ since $\sigma$ does not.

\subsection{Estimates for $a \varphib_\ep^a$.}
The estimates for $\varphib_\ep^a$ in Theorem \ref{periodic solutions} are valid for rescaled versions which are $2\pi$-periodic. 
They are not scaled in this  way when they appear in the expressions $j_n$ and $l_n$ and so we need to ``translate" the estimates from Theorem \ref{periodic solutions}.
The chief difficulty here---which is in fact one of the chief difficulties in the whole argument---is that the frequency of $\varphib^a$ depends on $a$. This frequency mismatch will
 ultimate lead to the loss of spatial decay in the Lipschitz estimates \eqref{lip1}. Here is the result.
 \begin{lemma}\label{phi lemma} For all $r \ge 0$ there exists $C_r>0$ such that for all $\ep \in (0,1)$ and $a,\ga \in [-a_0,a_0]$ we have
 \item \be\label{phi bound}
\| \varphib_\ep^a \|_{W^{r,\infty}} +\| J_2^\ep \varphib_\ep^a \|_{W^{r,\infty}}  \le C_r\ep^{-r}
\ee
and, for all $X \in \R$,
\be\label{Lip est phi}
\left\vert \partial_X^r J_2^\ep (\varphib_\ep^a - \varphib_\ep^\ga) \right\vert \le C_r \ep^{-r}|a-\ga|(1+|X|). 
\ee
 \end{lemma}

\begin{proof}
The  estimate \eqref{phi bound} follows directly from the estimates in Theorem \ref{periodic solutions}, the fact that 
$\tilde{J}_2^\ep(k)$ is uniformly bounded and the fact that $K_\ep = \O(1/\ep)$. We skip the details and instead focus on \eqref{Lip est phi}.
We make the decomposition 
$$
\varphib_\ep^a(X) - \varphib_\ep^{\grave{a}}(X) = \Delta_1 + \Delta_2 + \Delta_3
$$
where
\begin{multline}\label{Deltas}
\Delta_1 := \nub(K_\ep^a X) - \nub(K_\ep^{\ga} X), \quad \Delta_2:=\psib_\ep^a(K_\ep^a X) - \psib_\ep^\ga(K_\ep^a X)\\ \mand \Delta_3:=\psib_\ep^\ga(K_\ep^a X) - \psib_\ep^\ga(K_\ep^\ga X).
\end{multline}

We start with  $J_2^\ep \Delta_1$. 
Since $J_2^\ep [\ub e^{i\omega X}] = [\tilde{J}_2(\ep \omega) \ub] e^{i \omega X}$ and since $\nub(X) = \sin(X) \jb $, we see
\be\label{Jnub}
J_2^\ep [\nub ( \omega  X)] = {1 \over 2i} \left[ \tilde{J}_2(\ep  \omega ) \jb \right] e^{i \omega X} -{1 \over 2i} \left[ \tilde{J}_2(-\ep  \omega )\jb \right] e^{-i \omega X}. 
\ee
Thus
\begin{multline*}
J_2^\ep \Delta_1 = {1 \over 2i} \left[ \tilde{J}_2(\ep  K_\ep^a ) \jb \right] e^{i K_\ep^a X}- {1 \over 2i} \left[ \tilde{J}_2(\ep  K_\ep^\ga ) \jb \right] e^{i K_\ep^\ga X}\\ -{1 \over 2i} \left[ \tilde{J}_2(-\ep  K_\ep^a )\jb \right] e^{-i K_\ep^a X}
 +{1 \over 2i} \left[ \tilde{J}_2(-\ep  K_\ep^\ga )\jb \right] e^{-i K_\ep^\ga X}. 
\end{multline*}
We add a lot of zeros and do a lot of rearranging to get:
\be\label{guide}\begin{split}
J_2^\ep \Delta_1 &= 
{1 \over 2i} \left[\left( \tilde{J}_2(\ep  K_\ep^a ) - \tilde{J}_2(\ep  K_\ep^\ga ) \right) \jb \right]e^{i K_\ep^a X}\\
&+{1 \over 2i} \left[ \tilde{J}_2(\ep  K_\ep^\ga ) \jb \right] \left( e^{i K_\ep^a X} -e^{i K_\ep^\ga X}\right)\\&
+{1 \over 2i} \left[\left( \tilde{J}_2(-\ep  K_\ep^a ) - \tilde{J}_2(-\ep  K_\ep^\ga ) \right) \jb \right]e^{-i K_\ep^a X}\\
&+{1 \over 2i} \left[ \tilde{J}_2(-\ep  K_\ep^\ga ) \jb \right] \left( e^{-i K_\ep^a X} -e^{-i K_\ep^\ga X}\right).
\end{split}\ee

We know from Corollary \ref{J cor} that $\tilde{J}_2(k)$ is analytic and, since it is periodic for $k \in \R$, globally Lipschitz on $\R$.
Thus we can estimate the term in the first  line as
$$
\left \vert {1 \over 2i} \left[\left( \tilde{J}_2(\ep  K_\ep^a ) - \tilde{J}_2(\ep  K_\ep^\ga ) \right) \jb \right]e^{i K_\ep^a X} \right \vert 
\le C \ep|K_\ep^a - K_\ep^\ga|. 
$$
The uniform Lipschitz estimate \eqref{periodic lip} for $K_\ep^a$ in  Theorem \ref{periodic solutions} then gives
$$
\left \vert {1 \over 2i} \left[\left( \tilde{J}_2(\ep  K_\ep^a ) - \tilde{J}_2(\ep  K_\ep^\ga ) \right) \jb \right]e^{i K_\ep^a X} \right \vert 
\le C \ep|a-\ga|.
$$
Exactly the same reasoning leads to the following estimate on the third line:
$$
\left \vert {1 \over 2i} \left[\left( \tilde{J}_2(-\ep  K_\ep^a ) - \tilde{J}_2(-\ep  K_\ep^\ga ) \right) \jb \right]e^{-i K_\ep^a X}\right \vert \le C \ep |a-\ga|.
$$

To estimate the second line of \eqref{guide}, first we use the fact that $\tilde{J}_2(k)$ is uniformly bounded for $k \in \R$:
$$
\left \vert {1 \over 2i} \left[ \tilde{J}_2(\ep  K_\ep^\ga ) \jb \right] \left( e^{i K_\ep^a X} -e^{i K_\ep^\ga X}\right) \right \vert \le C \left \vert e^{i K_\ep^a X} -e^{i K_\ep^\ga X}\right \vert.
$$
Then we use the global Lipschitz estimate for the complex exponential:
 $|e^{i y} - e^{i y'}| \le 2|y-y'|$ for $y,y' \in \R$.
 This gives
$$
\left \vert {1 \over 2i} \left[ \tilde{J}_2(\ep  K_\ep^\ga ) \jb \right] \left( e^{i K_\ep^a X} -e^{i K_\ep^\ga X}\right) \right \vert \le C \left \vert K_\ep^a X  -K_\ep^\ga X\right \vert.
$$
Then, as above, the Lipschitz estimate \eqref{periodic lip} for $K_\ep^a$  gives:
$$
\left \vert {1 \over 2i} \left[ \tilde{J}_2(\ep  K_\ep^\ga ) \jb \right] \left( e^{i K_\ep^a X} -e^{i K_\ep^\ga X}\right) \right \vert \le C \left \vert a- \ga \right \vert|X|.
$$
In exactly the same fashion we can estimate the term in the fourth line to get:
$$
\left \vert {1 \over 2i} \left[ \tilde{J}_2(-\ep  K_\ep^\ga ) \jb \right] \left( e^{-i K_\ep^a X} -e^{-i K_\ep^\ga X}\right)\right \vert \le C \left \vert a- \ga \right \vert|X|.
$$

Thus all together we have:
\be\label{D1 est 1}
\left \vert J_2^\ep \Delta_1(X)\right \vert \le C \ep |a - \ga| + C |a-\ga||X| \le C|a-\ga| (1 + |X|).
\ee

We also want to estimate $\partial_X^r J_2^\ep \Delta_1$. 
Each term in  $J_2^\ep \Delta_1$ contains $e^{i K_\ep^a X}$ or $e^{i K_\ep^\ga X}$ and thus taking $r$ derivatives with respect to $X$ will produce
additional terms like $(K_\ep^a)^r$.
We know that $K_\ep^0 = \O(1/\ep)$ and the Lipschitz estimate \eqref{periodic lip} for  $K_\ep^a$ implies that $K_\ep^a = \O(1/\ep)$ as well.
Thus  $(K_\ep^a)^r = \O(\ep^{-r})$.
This results in the following estimate:
\be\label{D1 est 2}
\left \vert \partial_X^r J_2^\ep \Delta_1(X)\right \vert \le  C_r\ep^{-r}|a-\ga| (1 + |X|).
\ee

Now look at $J_2^\ep \Delta_2$.  
We know that $\psib_\ep^a(y)$ is $2\pi-$periodic and, moreover, smooth in $X$. Thus we can expand it in its Fourier series:
$
\psib_\ep^a(y) = \sum_{ j \in \Z} \hat{\psib_\ep^a}(j) e^{ijy}.
$
Noting that 
both terms in $\Delta_2$ are periodic with the same frequency, we see that:
$$
\Delta_2(X) = \sum_{j\in \Z} (\hat{\psib_\ep^a}(j)-\hat{\psib_\ep^\ga}(j)) e^{i jK_\ep X}.
$$
Applying ${J}^\ep_2$ gives
$$
J_2^\ep \Delta_2(X) = \sum_{j\in \Z} \tilde{J}_2(\ep j K_\ep^a) (\hat{\psib_\ep^a}(j)-\hat{\psib_\ep^\ga}(j)) e^{i j K^a_\ep X}.
$$
Since $\psib_\ep^a$ is smooth, classical Fourier series estimates
give
$$
\left \vert  \hat{\psib_\ep^a}(j)-\hat{\psib_\ep^\ga}(j)\right \vert\le C_r(1+|j|^r)^{-1} \| \psib^a_\ep - \psib^\ga_\ep\|_{C^r_\per \times C^r_\per}
$$
where we make take $r$ as large as we wish. 
The uniform Lipschitz estimate \eqref{periodic lip} for $\psib^a_\ep$ in Theorem \ref{periodic solutions} then implies:
$$
\left \vert  \hat{\psib_\ep^a}(j)-\hat{\psib_\ep^\ga}(j)\right \vert\le C(1+j^2)^{-1} |a-\ga|.
$$
Thus, since $\left\{(1+j^2)\right\}_{j \in \Z}$ is summable, 
\be\label{D2 est 0}
\left \vert J_2^\ep \Delta_2(X) \right \vert \le C|a-\ga|.
\ee
As above if we differentiate  $J_2^\ep \Delta_2$ $r$ times with respect to $X$ (each of which produces one power of $K_\ep^a$) and repeat the same steps we find:
\be\label{D2 est 2}
\left\vert \partial_X^rJ_2^\ep \Delta_2(X) \right \vert \le C_r\ep^{-r}|a-\ga|.
\ee

To handle $\Delta_3$ is basically a combination of how we dealt with $\Delta_1$ and $\Delta_2$.
Using the Fourier expansion for $\psib_\ep^a$ from above we see that:
$$
J_2^\ep\Delta_3=\sum_{ j \in \Z}\left( 
  [\tilde{J}_2(\ep jK_\ep^a)\hat{\psib_\ep^\ga}(j)] e^{ijK_\ep^aX}
-  [\tilde{J}_2(\ep jK_\ep^\ga)\hat{\psib_\ep^\ga}(j)] e^{ijK_\ep^\ga X}
\right).
$$
Adding zero and rearranging terms gives:
\bes\begin{split}
J_2^\ep\Delta_3=&\sum_{ j \in \Z}
  [\tilde{J}_2(\ep jK_\ep^a)\hat{\psib_\ep^\ga}(j)]\left( e^{ijK_\ep^aX}-  e^{ijK_\ep^\ga X}
\right)\\
+& \sum_{ j \in \Z}
  \left[\left(\tilde{J}_2(\ep jK_\ep^a)-\tilde{J}_2(\ep jK_\ep^\ga)\right)\hat{\psib_\ep^\ga}(j)\right] e^{ijK_\ep^\ga X}.
\end{split}\ees

Using (as we did when estimating $\Delta_1$ above) the fact that $\tilde{J}_2$ and $e^{iy}$ are globally Lipschitz together with the estimate $|K_\ep^a -K_\ep^\ga| \le C|a-\ga|$ implied by Theorem \ref{periodic solutions}, we have
$$
\left\vert J_2^\ep\Delta_3(X)\right\vert \le C|a-\ga|(1+|X|)\sum_{j \in \Z}\left \vert \hat{\psib_\ep^\ga}(j)\right \vert |j|.
$$
Next (as we did when estimating $\Delta_2$) we use the rapid decay of the Fourier coefficients of $\psib_\ep^\ga$ to conclude
that $\ds \sum_{j \in \Z}\left \vert \hat{\psib_\ep^\ga}(j)\right \vert |j| \le C$. This gives
\be\label{D3 est 1}
\left\vert J_2^\ep\Delta_3(X)\right\vert  \le C|a-\ga|(1+|X|).
\ee
In exactly the same fashion, we can establish
\be\label{D3 est 2}
\left\vert \partial_X^r J_2^\ep\Delta_3\right\vert  \le C_r \ep^{-r}|a-\ga|(1+|X|).
\ee
Thus all together we have shown \eqref{Lip est phi}.
\end{proof}

\subsection{Product estimates} Since our nonlinearity is quadratic we need good estimates for products of functions. In particular we need estimates
that keep track of decay rates.
First we note the famous Sobolev inequality $\|f\|_{L^\infty(\R)} \le \| f\|_{H^1(\R)}$ implies
\be\label{set}
\| \cosh(q \cdot) f \|_{W^{r,\infty}} \le C_r\|f \|_{r+1,q}
\ee
for all $r \ge 0$ and $q \ge 0$. Then we have:

\begin{lemma}\label{rq estimates}
For all $r \ge 0$ there exists $C_r>0$ such that following estimates hold for all $q,q' \ge 0$.
If $q \ge q'$ then 
\be\label{product inequality not borrow}
\| f g\|_{r,q} \le C_r\|f\|_{r,q'} \|\cosh(|q-q'| \cdot) g\|_{W^{r,\infty}} .
\ee

If $q \le q'$ then
\be\label{product inequality borrow}
\| f g\|_{r,q} \le  C_r \|f\|_{r,q'} \|\sech(|q'-q| \cdot) g\|_{W^{r,\infty}}.
\ee

Lastly, if $r \ge 1$ and $0 \le q' \le q$:
\be\label{SET}
\| f g\|_{r,q} \le C_r \| f \|_{r,q'}\|g\|_{r,q-q'}.
\ee

\end{lemma}
\begin{proof}
Definitionally
$
\| f g\|_{r,q} = \| \cosh(q \cdot) f g \|_{H^r}.
$
We multiply by one inside as follows:
$$
\| f g\|_{r,q} = \| \left( \cosh(q \cdot) \sech(q' \cdot) \cosh((q'-q)\cdot) \right)(\cosh(q'\cdot) f) (\sech((q'-q)\cdot) g) \|_{H^r}.
$$
The estimate $\| uv\|_{H^r} \le C \| u\|_{H^r} \| v\|_{W^{r,\infty}}$ is well-known and using it here gives:
$$
\|f g\|_{r,q} \le \|  \cosh(q \cdot) \sech(q' \cdot) \cosh((q'-q)\cdot) \|_{W^{r,\infty}}\| \cosh(q'\cdot) f\|_{H^r} \| \sech((q'-q)\cdot) g\|_{W^{r,\infty}}.
$$
Routine calculus methods shows that the condition $q \le q'$ implies
$$
\|  \cosh(q \cdot) \sech(q' \cdot) \cosh((q'-q)\cdot) \|_{W^{r,\infty}} \le C_r
$$
for a constant $C_r$ which depends only on $r$. This gives \eqref{product inequality borrow}.

If instead we multiply by one inside like:
$$
\| f g\|_{r,q} = \| \left( \cosh(q \cdot) \sech(q' \cdot) \sech((q'-q)\cdot) \right)(\cosh(q'\cdot) f) (\cosh((q'-q)\cdot) g) \|_{H^r}.
$$
then the estimate
$$
\|  \cosh(q \cdot) \sech(q' \cdot) \sech((q'-q)\cdot) \|_{W^{r,\infty}} \le C_r,
$$
which holds when $q\ge q,'$
gives \eqref{product inequality not borrow}.

The remaining estimate \eqref{SET} follows from  \eqref{product inequality not borrow} and \eqref{set}.
\end{proof}

\begin{remark}
Note that we we sometimes refer to \eqref{product inequality borrow} as a ``decay borrowing" estimate, since it allows growth in $g$ at the expense of extra decay in $f$.
On the other hand, the estimates \eqref{product inequality not borrow} and \eqref{SET} require both $f$ and $g$ to decay.
\end{remark}

\subsection{Estimates for $\A$}
The next result confirms the earlier claim that $\A$ is invertible on even functions.
\begin{proposition}\label{A inv} There exists $q_1>0$ such that $\A$ is a bijection from $E^r_q$ to itself
for all $q \in [0,q_1]$ and $r \ge 1$. Additionally, for each $r\ge1$, there exists $C>0$ such that 
\be\label{A inv bound}
\| \A^{-1} f \|_{r,q} \le C \| f\|_{r,q}
\ee
for all $q \in [0,q_1]$ and $f \in E^r_q$.
\end{proposition}
\begin{proof}
This is shown to be true in \cite{friesecke-pego1} for the special case when $q = 0$. The extension to $q>0$ can be achieved by using the now classical technique of operator conjugation \cite{pego-weinstein}. We omit the details.
\end{proof}

\subsection{Estimates for $\iota_\ep$}
The following estimate is a version of the famous Riemann-Lebesgue Lemma:
\begin{lemma}\label{rl lemma}
There exists $C>0$ such that for any $f\in H^r_q$, with $r \ge 0$, $q > 0$ and $|\omega|\ge1$ we have:
\bes\label{rl bound}
\left \vert \int_\R f(x) e^{i \omega x} dx \right \vert \le {C \over \omega^r \sqrt{q} } \| f\|_{r,q}.
\ees

\end{lemma}
\begin{proof}
Assume that $f$ is a Schwartz class function and $|\omega|\ge1$. Then integration by parts gives:
\begin{equation*}
\begin{split}
I:=\left \vert \int_\R f(x) e^{i \omega x} dx \right \vert = & \left \vert  \int_\R f(x) {1 \over  \omega^r} {d^r \over dx^r}[e^{i\omega x}]dx \right \vert =  |\omega|^{-r} \left \vert  \int_\R f^{(r)}(x)  e^{i\omega x}dx \right \vert.
\end{split}
\end{equation*}
Next we use the triangle inequality to get:
$$
I \le |\omega|^{-r} \int_\R |f^{(r)}(x)| dx.
$$
Multliplication by one and Cauchy-Schwartz yields:
$$
I \le |\omega|^{-r} \int_\R |f^{(r)}(x)| \cosh(qx) \sech(qx) dx \le |\omega|^{-r} \| f^{(r)}\|_{0,q} \| \sech(q \cdot)\|_{L^2}.
$$
Of course $ \| f^{(r)}\|_{0,q}  \le \| f\|_{r,q}$ and  $\| \sech(q \cdot)\|_{L^2}=q^{-1/2} \| \sech( \cdot) \|_{L^2}$.
This establishes the conclusion for Schwartz class functions. A classical density argument completes the proof.
\end{proof}

Since $K_\ep = \O(1/\ep)$,  Lemma \ref{rl lemma} implies
\be\label{iota bound}
\left \vert  \iota_\ep[f]\right \vert\le {C \ep^r\over \sqrt{q}} \| f\|_{r,q}.
\ee

\subsection{Estimates for Fourier multipliers}
The following result of  Beale (specifically, Lemma 3 of \cite{beale2}) will be used repeatedly.
\begin{theorem} \label{beale1}
Suppose that $\tilde{\mu}(z)$ is a complex valued function which has the following properties:
\begin{enumerate}[(i)]\item
$\tilde{\mu}(z)$ is meromorphic on 
the closed strip $\overline{\Sigma}_q = \left\{ |\Im z | \le q\right\} \subset \C$ where $q > 0$;
\item
there exists $m \ge 0$ and $c_*,\zeta_*>0$ such that $|z|>\zeta_*$ and $z \in \overline{\Sigma}_q$ imply
$|\tilde{\mu}(z)|\le c_*/|\Re z|^{m}$;
\item the set of singularities of $\tilde{\mu}(z)$ in $\overline{\Sigma}_q$ (which we denote $P_\mu$) is finite
and, moreover, is contained in the interior $\Sigma_q$;
\item all singularities of $\tilde{\mu}(z)$ in $\overline{\Sigma}_q$ are simple poles.
\end{enumerate}
Let 
$$
U^r_{\mu,q}:= \left\{ f \in H^r_q : z \in P_\mu  \implies  \hat{f}(z) = 0\right\}.
$$
Then the Fourier multiplier operator $\mu$ with symbol $\tilde{\mu}$ 
is a bounded injective map from $U^r_{\mu,q}$ into $H^{r+m}_q$. Additionally, for all $m'\in[0,m]$, 
we have the estimates:
\be\label{mu inv est}
\| \mu f \|_{r+m',q} \le C_{\mu,m'} \| f\|_{r,q}
\ee
where
\be\label{constant}
C_{\mu,m'}:= \sup_{k \in \R} \left \vert 
{(1+|k|^2)^{m'/2}  \tilde{\mu}(k\pm iq)}
\right \vert.
\ee
\end{theorem}

The first consequence of this is:
\begin{cor}\label{all ops} For all $w > 1$ we have the following.
\begin{enumerate}[(i)]\item
There exists $C>0$ and $\tau_2 >0$ such that
the following holds for all $\tau \in (0,\tau_2]$ and $r \in \R$.
The operators $\lambda_\pm$ are  bounded linear maps
from $H^{r}_{\tau}\to H^r_\tau$. 
Likewise, 
$J_1$ and $J_2$ are bounded maps from 
$H^{r}_{\tau}\times H^r_\tau \to H^r_\tau \times H^r_\tau$.
We have the estimates:
\be\label{all ops 0}
\| \lambda_+ f \|_{r,\tau} +\| \lambda_- f \|_{r,\tau} \le C\| f\|_{r,\tau} \mand  \| J_1 \f \|_{r,\tau} + \|J_2 \f\|_{r,\tau} \le C\|\f\|_{r,q}.
\ee
\item There exists $C>0$, $\ep_1\in(0,1)$ and $q_2 >0$ such that
the following holds for all $q \in [0,q_2]$, $r \in \R$ and $\ep \in [0,\ep_1]$.
The operators $\lambda^\ep_\pm$  are bounded linear maps
from $H^{r}_{q}\to H^r_q$. 
Likewise $J^\ep_1$ and $J^\ep_2$ are bounded maps from 
$H^{r}_{q}\times H^r_q \to H^r_q \times H^r_q$. We have the estimates:
\be\label{all ops bound}
\| \lambda^\ep_+f  \|_{r,q} +\| \lambda^\ep_- f \|_{r,q} \le C\| f\|_{r,q} \mand  \| J^\ep_1 \f \|_{r,q} + \|J^\ep_2 \f\|_{r,q} \le C\|\f\|_{r,q}.
\ee
\end{enumerate}
\end{cor}
We do not provide the details of the proof. All the operators have symbols which are bounded analytic functions on strips  (Lemma \ref{lambda lemma}, Corollary \ref{J cor})
and thus  everything follows directly from Theorem \ref{beale1}.

\subsection{Estimates for $\chi_\ep$}
Restrict $q \in [0,q_2]$.
Using the definition of $\chi_\ep$ and the triangle inequality gives
$
\| \chi_\ep\|_{r,q}\le\left \|\lambda_+^\ep J_1^\ep \left( J_2^0 \sigmab. J_2^\ep \nub_\ep\right)\right\|_{r,q}.
$
Corollary \ref{all ops} tells us that for $\lambda_+^\ep$ and $J_1^\ep$ are operators from $H^r_q$ to itself which 
are bounded independently of $\ep$. Thus
$
\| \chi_\ep\|_{r,q}\le C\left \|J_2^0 \sigmab. J_2^\ep \nub_\ep\right\|_{r,q}.
$
Using the product inequality \eqref{product inequality not borrow} gives
$
\| \chi_\ep\|_{r,q}\le C\left \|J_2^0 \sigmab \right\|_{r,q} \left\|J_2^\ep \nub_\ep\right\|_{W^{r,\infty}}.
$
Using \eqref{sigma bound} and \eqref{phi bound} gives:
\be\label{chi est}
\| \chi_\ep\|_{r,q} \le {C_r \ep^{-r}} 
\ee
for any $r \ge 0$, $q \in [0,q_2]$ and $\ep \in (0,\ep_1]$.

\subsection{Estimates for $\kappa_\ep$}
A rather tedious computation shows the unsurprising result that $\chi_\ep(X)$ is an odd function of $X$; we omit it.
Given this, we have
$$
\kappa_\ep = \iota_\ep[\chi_\ep] ={2\pi i} \hat{\chi}_\ep(K_\ep).
$$
Since $\chi_\ep(X) = \lambda_+^\ep J_1^\ep(J_2^0 \sigmab J_2^\ep. \nub_\ep)\cdot \jb$
and $\lambda_\ep^\ep$ and $J^\ep_1$ are Fourier multipliers we have
$$
\hat{\chi}_\ep(K_\ep) = \tilde{\lambda}_+(\ep K_\ep)  \left(\tilde{J}_1(\ep K_\ep) \Fo[ J_2^0 \sigmab. J_2^\ep \nub_\ep](K_\ep)\right)\cdot \jb.
$$
By definition
$$
\Fo[J_2^0 \sigmab. J_2^\ep \nub_\ep](K_\ep) = {1 \over 2\pi} \int_\R J_2^0 \sigmab(X). J_2^\ep\nub_\ep(X) e^{- i K_\ep X} dX.
$$
and
$$\ds
J_2^0 \sigmab(X) = 2(1+w)\sigma(X) \left( \begin{array}{c} 1 \\1 \end{array}\right).
$$
Thus 
$$
\Fo[J_2^0 \sigmab. J_2^\ep \nub_\ep](K_\ep) = {1+w\over \pi} \int_\R \sigma(X)J_2^\ep\nub_\ep(X) e^{- i K_\ep X} dX.
$$

Since $\nub_\ep(X) =(2i)^{-1} [ e^{iK_\ep X}-e^{-iK_\ep X}] \jb$ 
and $J_2^\ep$  is a Fourier multiplier, the last becomes
$$
\Fo[J_2^0 \sigmab. J_2^\ep \nub_\ep](K_\ep) = {1+w \over 2\pi i} \int_\R \sigma(X) \left(\tilde{J}_2(\ep K_\ep) \jb e^{i K_\ep X} -  \tilde{J}_2(-\ep K_\ep) \jb e^{-i K_\ep X}\right) e^{- i K_\ep X} dX.
$$
After rearranging terms in this we have
\begin{multline*}
\Fo[J_2^0 \sigmab. J_2^\ep \nub_\ep](K_\ep) = 
  {1+w \over 2\pi i}  \left( \int_\R \sigma(X) dX\right) \left( \tilde{J}_2(\ep K_\ep) \jb \right) \\
-  {1+w \over 2\pi i} \left( \int_\R \sigma(X) e^{-2iK_\ep X} dX\right) \left( \tilde{J}_2(-\ep K_\ep) \jb \right).
\end{multline*}

Recalling the definition of the Fourier transform, the above can be written as:
$$
\Fo[J_2^0 \sigmab. J_2^\ep \nub_\ep](K_\ep)= 
  -{i(1+w)} \hat{\sigma}(0)\left( \tilde{J}_2(\ep K_\ep) \jb \right) 
+i(1+w)\hat{\sigma}(2K_\ep) \left( \tilde{J}_2(-\ep K_\ep) \jb \right).
$$
Since $\sigma(X)>0$ for all $X$, we have $\hat{\sigma}(0)>0$. Since $\sigma(X)$ is analytic and square integrable, classical Fourier analysis can be used to show that there is a constant $C>0$
for which $|
\hat{\sigma}(k) | \le C e^{-C|k|}.
$
Since $K_\ep = \O(1/\ep)$ this means
that
$
|\hat{\sigma}(2 K_\ep) | \le C e^{-C/\ep}.
$
That is to say, it is exponentially small in $\ep$. Thus we have shown that
\be\label{most important part}
\left \vert \kappa_\ep - \kappa^*_\ep
\right \vert \le C e^{-C/\ep}
\ee
where
$$
\kappa^*_\ep:= 2 \pi(1+w) \hat{\sigma}(0) \tilde{\lambda}_+(\ep K_\ep) [\tilde{J}_1(\ep K_\ep) \tilde{J}_2(\ep K_\ep) \jb] \cdot \jb.
$$

It has been some time, but $\tilde{J}_1 = \tilde{J}_2^{-1}$. Thus
$$
\kappa^*_\ep= 2\pi (1+w)\hat{\sigma}(0) \tilde{\lambda}_+(\ep K_\ep).
$$
We also saw in Lemma \ref{lambda lemma} that 
$2w < \tilde{\lambda}_+(k)$ for all $k \in \R$. Thus $\kappa^*_\ep$ is strictly bounded away from zero.
This, with $\eqref{most important part}$, demonstrates that there is constant $C>0$ and $\ep_2\in(0,1)$ for which
\be\label{key kappa estimate}
|\kappa_\ep| \ge C \quad \text{for all $\ep \in (0,\ep_2]$}.
\ee

\subsection{Estimates and solvability conditions for $\B_\ep$}
Theorem \ref{beale1}
allows us establish the features of $\B_\ep$
described in the previous section. In particular we have:
\begin{lemma}\label{B inv}
There exists $\ep_3\in(0,1)$ and $q_3 >0$ such that for all $q \in (0,q_3]$ 
there exists $C_q>0$ such that for all $\ep \in (0,\ep_2]$
the following hold.
\begin{enumerate}[(i)]
\item There exists $f \in H^{r+2}_q$ such that $\B_\ep f = g \in H^r_q$ if and only if $\hat{g}(\pm K_\ep)  =0$. 
\item If $\hat{g}(\pm K_\ep)  =0$ then the solution $f$ is unique. We denote the solution by $f=\B_\ep^{-1} g$.
\item If $\hat{g}(\pm K_\ep)  =0$ then the solution $f$ satisfies the estimates
\be\label{B inv bound}
\| f \|_{r+j,q} = \|\B_\ep^{-1} g \|_{r,q} \le {C_q \over \ep^{j+1} } \| g\|_{r,q} \quad \text{where $j=0,1$ or $2$}.
\ee
\item For all $g \in O^r_q$ there exists a unique $f\in O^{r+2}_q$ such that
$\ds
\B_\ep f= g - {1\over \kappa_\ep} \iota_\ep[g] \chi_\ep.
$
We denote the solution by $f = \P_\ep g$. 
\item We have the estimates
\be\label{P bound}
\| \P_\ep g \|_{r+j,q} \le {C_q \over \ep^{j+1} }\| g\|_{r,q} \quad \text{where $j=0,1$ or $2$}.
\ee
\item Lastly,
$$
C_q \to \infty\quad  \text{as $q \to 0^+$}.
$$
\end{enumerate}
\end{lemma}

To prove this, we need the following result, which is proved in Section \ref{proofs}.
\begin{lemma}\label{xi lemma}
Let $$\tilde{\xi}_c(z) := -c^2 z^2 + \tilde{\lambda}_+(z).$$
There exists $\delta>0$, $\ell_0>0$, $R>0$, $\tau_3>0$ and $C>0$ such that the following hold when $|c-c_w| \le \delta $ and $|\tau|\le \tau_3$.
\begin{enumerate}[(i)]
\item $|\txi_c(z)|$ is analytic on the closed strip $\overline{\Sigma}_{\tau_3}:=\left\{|\Im z| \le \tau_3\right\}$ and is even. 
\item $|\txi_c(z)| \ge C|z|^2$ for $|z|\ge R$.
\item If $j = 0,1$ or $2$ then
\be\label{xi estimate}
\inf_{k \in \R} (1+k^2)^{-j/2}|\txi_c(k+i\tau)| \ge C|\tau|.
\ee
\end{enumerate}
\end{lemma}

\begin{proof} (of Lemma \ref{B inv})
Suppose that $\ep \in (0,\ep_3)$ where $\ep_3 :=\min(1,\ep_1,\ep_2,{\delta})$ and $q \in (0,q_3]:=\min(\tau_1,\tau_2,\tau_3,q_0,q_1,q_2)$.

The map $\B_\ep$ can be viewed as a Fourier multiplier with symbol
$$
\tilde{\B}_\ep(Z)=\txi_{\sqrt{c_w^2 + \ep^2}}(\ep Z).
$$
From Part (vi) of Lemma \ref{lambda lemma} and the estimate \eqref{xi estimate} we know that 
that 
$\tilde{\B}_\ep(Z)$  has exactly two zeros, both real and simple, at $Z = \pm K_\ep= \pm k_{\sqrt{c_w^2 + \ep^2}}/\ep$, in $\overline{\Sigma}_{q} \subset \overline{\Sigma}_{\tau_1/\ep}$. 
Thus we see that if $f \in H^{r+2}_q$, with $0< q\le q_3$, then $\hat{\B}_\ep(\pm K_\ep) \hat{f}(\pm K_\ep)=0$; this is ``only if" of Part (i). (This is also equivalent to the condition that $\iota_\ep[g] =0$ discussed above, if $g$ is odd.)

And so we see that
$
1/\tilde{\B}_\ep(Z)
$
has two simple poles at $P_\ep =\left\{ \pm K_\ep \right\}$ and no other poles in $\overline{\Sigma}_q$ when $q \in (0,q_3]$.
Similarly, Part (ii) of Lemma \ref{xi lemma} indicates that $1/|\tilde{\B}_\ep(Z)| \le C |\Re Z|^{-2}$ for $|Z|$ large enough.
Thus $1/\tilde{\B}_\ep(Z)$
satisfies all the conditions of the multiplier in Theorem \ref{beale1} for any decay rate $q \in (0,q_3]$, $m = 2$ and pole set $\left\{ \pm K_\ep \right\}$.

And so we have a well-defined  map $\B_\ep^{-1}$ 
from
$U^r_{\ep,q}:=\left\{g \in H^r_q : \hat{g}(\pm K_\ep) = 0 \right\}$
into $H^{r+2}_q$ which inverts $\B_\ep$. Specifically $\B_\ep \B_\ep^{-1}$ is the identity on $U^r_{\ep,q}$.
Putting $f = \B_\ep^{-1}$ gives the other implication in Part (i). The uniqueness of Part (ii) follows from the injectivity of $\B_\ep$.

The estimates \eqref{B inv bound} in Part (iii) follow from \eqref{mu inv est}, \eqref{constant} and the estimate in \eqref{xi estimate}.
Specifically, fix $q \in (0,q_3]$. The formula $\eqref{constant}$ tells that
$
\|\B_\ep^{-1} g\|_{r+j,q} \le C_{\ep,j} \|g\|_{r,q} 
$
where
$$
C_{\ep,j}:=\sup_{K \in \R} \left \vert (1+|K|^2)^{j/2} \tB^{-1}_\ep(K+iq)\right \vert 
$$
when $j = 0,1$ or $2$. Thus to get \eqref{B inv bound} we need to show that $C_{\ep,j} \le C_q/\ep^{j+1}$.

Letting $k = \ep K$ we see:
$$
C_{\ep,j}
=\sup_{K \in \R} \left \vert (1+|K|^2)^{j/2} \txi^{-1}_{\sqrt{c_w^2 + \ep^2}} (\ep K+i\ep q)\right \vert 
= \sup_{k \in \R} \left \vert (1+|k/\ep|^2)^{j/2} \txi^{-1}_{\sqrt{c_w^2 + \ep^2}} (k+i\ep q)\right \vert.
$$
Then we multiply by one on the inside and use elementary estimates to get:
\bes\begin{split}
C_{\ep,j}
& \le  \sup_{k \in \R} \left \vert (1+|k/\ep|^2)^{j/2} \over (1+k^2)^{j/2} \right \vert 
\sup_{k \in \R} \left\vert \ (1+k^2)^{j/2} \txi^{-1}_{\sqrt{c_w^2 + \ep^2}} (k+i\ep q)\right \vert\\
 & \le \ep^{-j} \sup_{k \in \R} \left\vert \ (1+k^2)^{j/2} \txi^{-1}_{\sqrt{c_w^2 + \ep^2}} (k+i\ep q)\right \vert
\end{split}\ees
Then we use \eqref{xi estimate} with $\tau = \ep q$ to get
$$
C_{\ep,j} \le C\ep^{-j} |\ep q|^{-1} = C|q|^{-1} \ep^{-j-1}.
$$
Thus we have, using \eqref{mu inv est} and \eqref{constant},
$$
\| \B_\ep^{-1} g \|_{r+j,q} \le {C_q \over \ep^{j+1}} \| g \|_{r,q}
$$
which was our goal. Note that $C_q = C/|q|$ and so we have $C_q \to \infty$ as $q \to 0^+$, as stated in Part (vi).

To prove parts (iv) and (v) we first observe that $\B_\ep$ (and therefore $\B_\ep^{-1}$) maps
odd functions to functions. For odd functions, a short computation shows that $\iota_\ep[g] = 2 \pi i \hat{g}(K_\ep)$.
Thus 
$$
\Fo\left[g - {1 \over \kappa_\ep} \iota_\ep[g] \chi_\ep\right](\pm K_\ep) = 0.
$$
So we can apply parts (i)-(ii) to
get Part (iv). The estimate in Part (v) is shown as follows. 
The Riemann-Lebesgue estimate \eqref{iota bound} implies that $|\iota_\ep [g]| \le C_q \ep^r \| g \|_{r,q}$.
And so if we use this, the estimate in Part (iii), \eqref{chi est} and \eqref{key kappa estimate}, we have:
$$
\| \P_\ep g \|_{r+j,q} \le {C_q \over \ep^{j+1}}\left \| g -{1 \over \kappa_\ep}\iota_\ep[g]\chi_\ep\right\|_{r,q}
\le {C_q \over \ep^{j+1}}\left( \| g\|_{r,q} +|\iota_\ep[g]|\ep^{-r}\right) \le {C_q \over \ep^{j+1}}\| g\|_{r,q}. 
$$
\end{proof}

\subsection{Symbol truncation estimates}
In this subsection we prove a series of results which give estimates
the operator norm of things like $J^\ep_1 -J^0_1$ when $\ep$ is small. 
\begin{lemma} \label{truncate}
Suppose that $\tilde{\mu}$ meets the hypotheses of Theorem \ref{beale1}
with 
$P_\mu = \left\{\cdot \right\}$ and $m = 0$. (That is to say, $\tmu$ is bounded.)
Let $\tilde{\zeta}_n(z):=\tilde{\mu}(z) - \sum_{j=0}^{n} \mu_j z^j$
where the constants $\mu_j\in\C$ are the coefficients in the Maclaurin series of $\tilde{\mu}$.
Then the Fourier multiplier operators $\zeta_n$ with symbols $\tilde{\zeta}_n$ are bounded
from $H^{r+n+1}_q$ to $H^{r}_q$ and satisfy the estimate
$$
\| \zeta_n f \|_{r,q} \le C_n \| f^{(n+1)} \|_{r,q}.
$$
The constant $C_n>0$ is independent of $r$.

\end{lemma}
\begin{proof} 
Within the radius of convergence of the Maclaurin series we have:
$$
\tilde{\zeta}_n(z) = \sum_{j=n+1}^\infty \mu_j z^j = z^{n+1} \sum_{j=0}^\infty \mu_{j+n+1} z^j.
$$
So if we put
$$
\tilde{\upsilon}_n(z) = \tilde{\zeta}_n(z)/z^{n+1}
$$
then clearly the singularity at $z = 0$ is removable. This implies that on any closed disk containing the origin and with radius smaller than the radius of convergence, $\tilde{\upsilon}(z)$ is analytic and bounded. Outside this disk, but within $\overline{\Sigma}_q$, we have $\tilde{\upsilon}(z) = (\tilde{\mu}(z) - \sum_{j=0}^{n} \mu_j z^j)/z^{n+1}$. In this case, all of the functions on the right hand side are analytic.
Moreover they are all bounded since $\tilde{\mu}(z)$ is bounded and $|z|$ is smallest on the boundary of the disk. In short, $\tilde{\upsilon}(z)$ meets the hypothesis of Theorem \ref{beale1} on $\overline{\Sigma}_q$ with an empty pole set, $m = 0$.  Let $\upsilon$ be the operator associated to $\tilde{\upsilon}$.
Observing that $\zeta_n f = (-i)^n \upsilon_n f^{(n+1)}$ and applying the results of Theorem \ref{beale1} finishes the proof.
\end{proof}

This result implies the following:
\begin{lemma}\label{long wave J}
There exists $C>0$
such  that for $q\in [0,\tau_2]$ and $\ep \in (0,1)$ we have the following
for all $r \ge 0$,
 \be\label{lw J}
\| (J_1^\ep - {J}^0_1)\ub\|_{r,q} \le C \ep \| \ub \|_{r+1,q}
\mand
\| (J_2^\ep - {J}_2^0)\ub\|_{r,q} \le C \ep \| \ub \|_{r+1,q}.
\ee
\end{lemma}
\begin{proof}
We prove the estimates for  $r = 0$. That multipliers commute with derivatives will extend this case to the general. So
$$\ds
\|(J_n^\ep - J_n^0) \ub \|_{0,q}^2  = \int_\R \left \vert (J_n^\ep-J_n^0)\ub(X) \right \vert^2 \cosh^2(qX) dX.
$$
If we let $\ub^{\ep}(x) = \ub(\ep x)$ then the discussion in Remark \ref{operator scaling} implies that
$$
\|(J_n^\ep - J_n^0) \ub \|_{0,q}^2  = \int_\R \left \vert [(J_n-J_n^0)\ub^\ep] (X/\ep) \right \vert^2 \cosh^2(qX) dX.
$$
We make the change of variables $X = \ep x$ in the integral to get:
$$
\|(J_n^\ep - J_n^0) \ub \|_{0,q}^2  = \ep \int_\R \left \vert [(J_n-J_n^0)\ub^\ep] (x) \right \vert^2 \cosh^2(q\ep x) dX = \ep \| (J_n - J_n^0) \ub^\ep\|_{0,q\ep}^2.
$$

Since $q \in [0,\tau_2$ and $\ep \in (0,1)$ we have $q\ep \le \tau_2$. We know that $\tilde{J}_2$ is analytic on $\overline{\Sigma}_{\tau_2}$ from Corollary \ref{J cor}.
Thus  we can use Lemma \ref{truncate} to get
$
\|(J_n^\ep - J_n^0) \ub \|_{0,q}^2 \le C \ep \| \partial_x \ub^\ep\|_{0,q\ep}^2.
$
A routine calculation show that $\| \partial_x \ub^\ep\|_{0,q\ep}^2 = \ep \| \partial_X \ub\|_{0,q}^2$. Thus we have the $r=0$ estimates
$
\|(J_n^\ep - J_n^0) \ub \|_{0,q} \le C \ep  \| \ub\|_{1,q}.
$

\end{proof}

\subsection{Estimates for $\varpi^\ep$}
In this subsection we shall prove some useful estimates for $\varpi^\ep$ and, in particular, show that
it converges in the operator norm topology to $\varpi^0$ as $\ep \to 0^+$. This result is similar to one employed in \cite{friesecke-pego1} and stands 
in contrast to the results from the previous subsection where the the approximation of $J^\ep_j$ by $J^0_j$
comes at the cost of a derivative. We have:

\begin{lemma} \label{pi props}
There exists $C>0$, $\ep_4 \in (0,1)$ and $q_4 >0$ such that following hold for all $\ep \in [0,\ep_4]$
and $q\in [0,q_4]$.
\begin{enumerate}[(i)]
\item $\varpi^\ep$ is a bounded map from  $H^r_q$ to $H^{r+2}_q$, for any $r \ge 0$.
\item For all $r \ge 0$ we have
\be\label{varpi smooth} \| \varpi^\ep f\|_{r+2,q} \le C \|f\|_{r,q}.\ee
\item For all $r \ge 0$ we have \be\label{lw varpi} \|\varpi^\ep - \varpi^0\|_{r,q} \le C \ep^2 \| f\|_{r,q}.\ee 
\item For all $\omega>1$ we have
\be\label{AT trick}
\| \varpi^\ep \left(f e^{i \omega \cdot}\right)\|_{1,q} \le C\omega^{-1} \| f\|_{2,q}.
\ee
\end{enumerate}
\end{lemma}
To prove this we need the following result, which is proved in Section \ref{proofs}.
\begin{lemma}\label{pi lemma}
There exists $C>0$, $\ep_4 \in (0,1)$ and $q_4 >0$ such that following hold for all $\ep \in [0,\ep_4]$
and $q\in [0,q_4]$.
\begin{enumerate}[(i)]
\item $\tilde{\varpi}^\ep(Z)$ is analytic and bounded in the strip $\overline{\Sigma}_q$.
\item $|\tilde{\varpi}^\ep(Z)| \le C/(1+|Z|^2)$ for all $Z \in \overline{\Sigma}_q$.
\item $|\tilde{\varpi}^\ep(Z)-\tilde{\varpi}^0(Z)| \le C\ep^2$ for all $Z \in \overline{\Sigma}_q$.
\end{enumerate}

\end{lemma}

\begin{proof} (Of Lemma \ref{pi props})
Parts (i)-(iii) of Lemma \ref{pi props} follow from Parts (i)-(iii) of Lemma \ref{pi lemma}
and an application of Theorem \ref{beale1}. We spare the details.

Part (iv) is proven using an idea employed from \cite{amick-toland}.
Let
$(1-\alpha_w \partial_X^2)^{-1} \left(f e^{i \omega X}\right) = g$. 
Equivalently,
$\ds
\alpha_w g'' -g = - f e^{i \omega X}.
$
Put $g_1:=g\ds +{1 \over \alpha_w\omega^2} f e^{i \omega X} $. Then
\bes\begin{split}
\alpha_w g_1'' -g_1 =& \alpha_w g'' - g -{1 \over \alpha_w \omega^2} f e^{i \omega X}+  f e^{i\omega X} +{2i \over  \omega} f' e^{i\omega X} + {1 \over  \omega^2} f'' e^{i \omega X} \\
=&-{1 \over \alpha_w \omega^2} f e^{i \omega X} +{2i \over  \omega} f' e^{i\omega X} + {1 \over  \omega^2} f'' e^{i \omega X} 
\end{split}\ees
or rather
\bes\begin{split}
g_1 =&  -(1-\alpha_w \partial_X^2)^{-1}\left(-{1 \over \alpha_w \omega^2} f e^{i \omega X} +{2i \over  \omega} f' e^{i\omega X} + {1 \over  \omega^2} f'' e^{i \omega X} \right)\\
        =&  c_w^{-2}\varpi^0\left(-{1 \over \alpha_w \omega^2} f e^{i \omega X} +{2i \over  \omega} f' e^{i\omega X} + {1 \over  \omega^2} f'' e^{i \omega X} \right).
\end{split}\ees
Then we use Part (i) to conclude that
$$
\| g_1\|_{2,q} \le C\omega^{-1} \| f\|_{2,q}.
$$

Next, naive estimates show that $\| f e^{i \omega \cdot} \|_{1,q} \le C \omega \| f\|_{1,q}$. And since 
$(1-\alpha_w \partial_X^2)^{-1}(f e^{i \omega X})=g = g_1- \ds {1 \over \alpha_w \omega^2} f e^{i \omega X}$ we have
$$
\|(1-\alpha_w \partial_X^2)^{-1}(f e^{i \omega \cdot})\|_{1,q} \le \| g_1\|_{1,q} + C \omega^{-2}\| f e^{i \omega \cdot} \|_{1,q}\le C \omega^{-1} \| f\|_{2,q}.
$$
This establishes the estimate for $\varpi^0=-c_w^2(1-\alpha_w \partial_X^2)^{-1}$. To establish it for $\varpi^\ep$ observe that the estimates in Part (ii) give $\|(1-\alpha_w \partial_x^2)\varpi^\ep f\|_{r,q} \le C \|f\|_{r,q}$. Thus we can reduce the $\varpi^\ep$ estimate to the $\varpi^0$ case in a simple way.
\end{proof}

\subsection{Estimates for $B^\ep$}
Finally we have several basic estimates for $B^\ep$. First we have some straightforward upper bounds.
\begin{lemma}\label{B lemma} For all $r \ge 0$ there exists $C_r>0$ such that for all $q,q' \in [0,q_2]$ and $\ep \in [0,\ep_1]$
we have the following estimates.
If $q \ge q'$ then 
\be\label{B bound not borrow}
\|B^\ep(\thetab,\grave{\thetab})\|_{r,q} \le C_r\|\thetab\|_{r,q'} \|\cosh(|q-q'| \cdot) J_2^\ep \grave{\thetab}\|_{W^{r,\infty}} 
\ee
for all $r\ge0$.
If $q \le q'$ then
\be\label{B bound borrow}
\|B^\ep(\thetab,\grave{\thetab})\|_{r,q} \le C_r\|\thetab\|_{r,q'} \|\sech(|q-q'| \cdot) J_2^\ep \grave{\thetab}\|_{W^{r,\infty}} 
\ee
for all $r\ge0$.
Lastly, if $0 \le q' \le q$ and $r\ge1$ then
\be\label{B bound set}
\|B^\ep(\thetab,\grave{\thetab})\|_{r,q} \le C \|\thetab\|_{r,q'}\|\grave{\thetab}\|_{r,q-q'}.
\ee

\end{lemma}
\begin{proof}
First use the bound on $J_1^\ep$  in \eqref{all ops bound} from Corollary \ref{all ops} to get
$
\| B^\ep(\thetab,\grave{\thetab})\|_{r,q} \le \| J_2^\ep \thetab . J_2^\ep \grave{\thetab}\|_{r,q}.
$
Using the various product estimates in Lemma \ref{rq estimates} followed by the bound 
 on $J_2^\ep$  in \eqref{all ops bound} from Corollary \ref{all ops} gives the estimates.

\end{proof}

The next result deals with approximation of $B^\ep$ by $B^0$.
\begin{lemma}\label{long wave B} There exists $\ep_5,q_5>0$ such that 
for all $r\ge0$ there exists $C_r>0$  such  that for $q\in [0,q_5]$ and $\ep \in (0,\ep_5]$ we have the following inequality
\be\label{lw B}
\|B^\ep(\thetab,\grave{\thetab}) - B^0(\thetab,\grave{\thetab})\|_{r,q} \le C_r \ep \| \thetab \|_{r+1,q'} \|\grave{\thetab}\|_{r+1,q''}.
\ee
Here $q' + q'' = q \in [0,q_4]$ and both are positive.
\end{lemma}
\begin{proof}
Let $\ep_5:=\min(1,\ep_1)$ and $q_5:=\min(q_2,\tau_2)$.
The triangle inequality gives
\be\label{ewe}\begin{split}
\| B^\ep(\thetab,\grave{\thetab}) - B^0(\thetab,\gthetab)\|_{r,q} 
=  &\left\| J_1^\ep\left(J_2^\ep \thetab . J_2^\ep \gthetab \right)-J_1^0\left(J_2^0 \thetab . J_2^0 \gthetab \right) \right\|_{r,q}\\
\le&\left\| J_1^\ep\left(J_2^\ep \thetab . J_2^\ep \gthetab \right)-J_1^\ep\left(J_2^\ep\thetab . J_2^0 \gthetab \right) \right\|_{r,q}\\
+&\left\| J_1^\ep\left(J_2^\ep \thetab . J_2^0 \gthetab \right)-J_1^\ep\left(J_2^0 \thetab . J_2^0 \gthetab \right) \right\|_{r,q}\\
+&\left\| J_1^\ep\left(J_2^0 \thetab . J_2^0 \gthetab \right)-J_1^0\left(J_2^0 \thetab . J_2^0 \gthetab \right) \right\|_{r,q}.
\end{split}\ee
For the first term, we use the bound on $J_1^\ep$  in \eqref{all ops bound} from Corollary \ref{all ops} to get$$
\left\| J_1^\ep\left(J_2^\ep \thetab . J_2^\ep \gthetab \right)-J_1^\ep\left(J_2^\ep\thetab . J_2^0 \gthetab \right) \right\|_{r,q}
\le C\left\| J_2^\ep \thetab .\left( J_2^\ep \gthetab-J_2^0 \gthetab\right) \right\|_{r,q}.
$$
Then we use the product inequality \eqref{product inequality not borrow}
$$
\left\| J_1^\ep\left(J_2^\ep \thetab . J_2^\ep \gthetab \right)-J_1^\ep\left(J_2^\ep\thetab . J_2^0 \gthetab \right) \right\|_{r,q}
\le C\left\| \cosh(q' \cdot) J_2^\ep \thetab\right \|_{W^{r,\infty}}\left \|  J_2^\ep \gthetab-J_2^0 \gthetab \right\|_{r,q''}
$$
where we have $q' + q'' = q$ and both are positive.
The Sobolev embedding theorem applied to the first term on the right hand side gives:
$$
\left\| J_1^\ep\left(J_2^\ep \thetab . J_2^\ep \gthetab \right)-J_1^\ep\left(J_2^\ep\thetab . J_2^0 \gthetab \right) \right\|_{r,q}
\le C\left\| J_2^\ep \thetab\right \|_{r+1,q'}\left \|  J_2^\ep \gthetab-J_2^0 \gthetab \right\|_{r,q''}.$$
Using the boundedness of $J_2^\ep$ we get
$$
\left\| J_1^\ep\left(J_2^\ep \thetab . J_2^\ep \gthetab \right)-J_1^\ep\left(J_2^\ep\thetab . J_2^0 \gthetab \right) \right\|_{r,q}
\le  C\left\|  \thetab\right \|_{r+1,q'}\left \|  J_2^\ep \gthetab-J_2^0 \gthetab \right\|_{r,q''}.
$$
On the second term we use the estimate \eqref{lw J} from Lemma \ref{long wave J} to arrive at:
$$
\left\| J_1^\ep\left(J_2^\ep \thetab . J_2^\ep \gthetab \right)-J_1^\ep\left(J_2^\ep\thetab . J_2^0 \gthetab \right) \right\|_{r,q}
\le  C \ep \left\|  \thetab\right \|_{r+1,q'} \|  \gthetab \|_{r+1,q''}. 
$$
So the first term in \eqref{ewe} is handled.

The rest of the terms in \eqref{ewe} are estimated in the same way, and we leave out the details.

\end{proof}

\section{The proof of Proposition \ref{main mover}.}\label{NE}
We are now in position to estimate $N^\ep$ and prove Proposition \ref{main mover}. Put $$q_\star := \min(q_0,q_1,q_2,q_3,q_4,q_5)\mand \ep_\star:=\min(\ep_0,\ep_1,\ep_2,\ep_3,\ep_4,\ep_5).$$
In this section we restrict
$$
\ep \in (0,\ep_\star] \mand q \in \left[{q_\star / 2},q_\star\right].
$$
We have the lower bound on $q$ in place so that the constant $C_q$ in the estimates for $\B^{-1}_\ep$ and $\P_\ep$ in Lemma \ref{B inv} is bounded above.
In this way, any constant $C>0$ which appears below is independent of $\ep$, $q$, $\etab$ (which is  $E^1_q \times O^1_q$) and $a$ (which is in $[-a_0,a_0]$).
Note that it is a consequence of Lemma \ref{EO2} that if $\eta_1$ is even and $\eta_2$ is odd that $N_1^\ep$ and $N_2^\ep$ are even and odd, respectively.
 

\subsection{The mapping estimates}
In this subsection we prove the estimate \eqref{ball1}.
The definitions of $N^\ep_1,N^\ep_2$ and $N^\ep_3$ give us:
$$
\|N^\ep_1(\etab,a)\|_{1,q} \le \| \A^{-1} j_1\|_{1,q} +\| \A^{-1} j_2\|_{1,q}+ \| \A^{-1} j_3\|_{1,q} +\| \A^{-1} j_4\|_{1,q}+ \| \A^{-1} j_5\|_{1,q},
$$
$$
\|N^\ep_2(\etab,a)\|_{1,q} \le \| \ep^2 \P_\ep  l_1\|_{1,q} + \| \ep^2 \P_\ep  l_2\|_{1,q} + \| \ep^2 \P_\ep  l_{31}\|_{1,q} +\| \ep^2 \P_\ep  l_4\|_{1,q} + \| \ep^2 \P_\ep  l_5\|_{1,q}
$$
and
$$
|N^\ep_3(\etab,a)| \le | \iota_\ep l_1|+ | \iota_\ep l_2|+| \iota_\ep l_{31}|+| \iota_\ep l_4|+| \iota_\ep l_5|.
$$

Using the uniform bound \eqref{A inv bound} on $\A^{-1}$ from Lemma \ref{A inv} we have
$$
\|N^\ep_1(\etab,a)\|_{1,q} \le C\left( \|  j_1\|_{1,q} +\|  j_2\|_{1,q}+ \|  j_3\|_{1,q} +\|  j_4\|_{1,q}+ \|  j_5\|_{1,q}\right).
$$
In Lemma \ref{B inv}, the estimate \eqref{P bound} gives $\|\P_\ep f\|_{1,q} \le C \ep^{-1} \| f\|_{1,q}$. Thus
$$
\|N^\ep_2(\etab,a)\|_{1,q} \le C\ep\left( \|  l_1\|_{1,q} + \|   l_2\|_{1,q} + \|  l_{31}\|_{1,q} +\|   l_4\|_{1,q} + \|   l_5\|_{1,q}\right).
$$
And the Riemann-Lebesgue estimate for $\iota_\ep$, \eqref{iota bound}, following Lemma \ref{rl lemma} gives (for $r=1$)
$$
|N^\ep_3(\etab,a)| \le C\ep\left( \|  l_1\|_{1,q} + \|   l_2\|_{1,q} + \|  l_{31}\|_{1,q} +\|   l_4\|_{1,q} + \|   l_5\|_{1,q}\right).
$$
Thus we will have \eqref{ball1} if we can show that each of the ten terms
$$
\|  j_1\|_{1,q},\ \|  j_2\|_{1,q},\ \|  j_3\|_{1,q},\ \|  j_4\|_{1,q},\ \|  j_5\|_{1,q},\ 
\ep\|l_1\|_{1,q},\ \ep\|l_2\|_{1,q},\ \ep\|l_{31}\|_{1,q},\ \ep\|l_4\|_{1,q}\ \text{and}\ \ep\|l_5\|_{1,q}
$$
is bounded by $C |RHS_{map}|$, where
$$
|RHS_{map}|:= \ep + \ep \|\etab\|_{1,q} + \ep |a|+   \|\etab\|^2_{1,q}+a^2.
$$

\subsubsection{Mapping estimates for $j_1$ and $ l_1$}
The choice of $\sigma$ was made so that 
$
\sigma+\varpi^0 b_1^0(\sigmab,\sigmab) = 0.
$
See~\eqref{what sigma does}.
Which means that we have
$$
j_1 = \varpi^0 b_1^0(\sigmab,\sigmab) - \varpi^\ep b_1^\ep(\sigmab,\sigmab)
=(\varpi^0-\varpi^\ep)b_1^0(\sigmab,\sigmab) + \varpi^\ep\left( b_1^0(\sigmab,\sigmab) - b_1^\ep(\sigmab,\sigmab)\right).
$$
Call the terms on the right hand side $I$ and $II$ respectively.

To estimate $I$,  we first use the
estimate \eqref{lw varpi} for  $\varpi^\ep-\varpi^0$  to get
$
\|I\|_{1,q} \le C \ep^2 \|b_1^0(\sigmab,\sigmab)\|_{1,q}. 
$
Then we use the estimate for $B^0$ from \eqref{B bound set} and get 
$
\|I\|_{1,q} \le C \ep^2 \| \sigmab \|^2_{1,q/2}.
$
Then the uniform bounds \eqref{sigma bound} for $\sigmab$ give
$
\|I\|_{1,q} \le C\ep^2.
$

For $II$ we use the smoothing property of $\varpi^\ep$  and the associated estimate
\eqref{varpi smooth} to get
$\|II\|_{1,q} \le  C\|b_1^0(\sigmab,\sigmab) - b_1^\ep(\sigmab,\sigmab)\|_{0,q}$.
Then we use the approximation estimates for $B^\ep$ by $B^0$ in \eqref{lw B}
to get
$\|II\|_{1,q} \le  C\ep\| \sigmab \|^2_{1,q/2}$.
Then the uniform bounds \eqref{sigma bound} for $\sigmab$ give
$
\|II\|_{1,q} \le C\ep.
$
Thus we have
$$
\|j_1\|_{1,q} \le C \ep \le C|RHS_{map}|.
$$

To estimate $l_1=\lambda_+^\ep b_2^\ep(\sigmab,\sigmab)$
is very easy using the  bounds for $\lambda_+^\ep$ in \eqref{all ops bound}, the bounds on $B^\ep$
in \eqref{B bound set} and the bounds on $\sigmab$ in \eqref{sigma bound}. We get
$\| l_1\|_{1,q} \le C$. And so 
$$
\ep \|l_1\|_{1,q} \le C \ep \le C|RHS_{map}|.
$$



\subsubsection{Mapping estimates for $j_2$ and $l_2$}
By adding zero we see
$$
j_2 = -2 (\varpi^\ep-\varpi^0) b_1^\ep(\sigmab,\etab) - 2 \varpi^0 \left(b_1^\ep(\sigmab,\etab)-b_1^0(\sigmab,\etab) \right).
$$
Call these two terms $I$ and $II$ respectively. 

To estimate $I$,  we first use the
estimate \eqref{lw varpi} for  $\varpi^\ep-\varpi^0$  to get
$
\|I\|_{1,q} \le C \ep^2 \|b_1^0(\sigmab,\etab)\|_{1,q}. 
$
Then we use the estimate for $B^0$ from \eqref{B bound set} and get 
$
\|I\|_{1,q} \le C \ep^2 \| \sigmab \|_{1,q}\|\etab\|_{1,0}.
$
Then the uniform bounds \eqref{sigma bound} for $\sigmab$ give
$
\|I\|_{1,q} \le C\ep^2\|\etab\|_{1,0}.
$

For $II$ we use the smoothing property of $\varpi^0$  and the associated estimate
\eqref{varpi smooth} to get
$\|II\|_{1,q} \le  C\|b_1^\ep(\sigmab,\etab) - b_1^0(\sigmab,\etab)\|_{0,q}$.
Then we use the approximation estimates for $B^\ep$ by $B^0$ in \eqref{lw B}
to get
$\|II\|_{1,q} \le  C\ep\| \sigmab \|_{1,q}\|\etab\|_{1,0}$.
Then the uniform bounds \eqref{sigma bound} for $\sigmab$ give
$
\|II\|_{1,q} \le C\ep\|\etab\|_{1,0}.
$
Thus we have
\be\label{j2 map est}
\|j_2\|_{1,q} \le C \ep\|\etab\|_{1,0} \le C|RHS_{map}|.
\ee

To estimate $l_2=\lambda_+^\ep b_2^\ep(\sigmab,\etab)$, as with $l_1$, 
is very straightforward using the  bounds for $\lambda_+^\ep$ in \eqref{all ops bound}, the bounds on $B^\ep$
in \eqref{B bound set} and the bounds on $\sigmab$ in \eqref{sigma bound}. We get
$\| l_2\|_{1,q} \le C\|\etab\|_{1,0}$. And so 
\be\label{l2 map est}
\ep \|l_2\|_{1,q} \le C \ep\|\etab\|_{1,0} \le C|RHS_{map}|.
\ee

\subsubsection{Mapping estimates for $j_3$ and $l_{31}$}
Recalling the definition of $j_3$, we have
\bes\begin{split}
j_3 &= -2a\varpi^\ep J_1^\ep \left( J_2^0\sigmab . J_2^\ep \nub_\ep\right)\cdot \ib -2 a \varpi^\ep J_1^\ep \left((J_2^\ep-J_2^0)\sigmab. J^\ep_2\nub_\ep\right)\cdot\ib -2a \varpi^\ep B^\ep(\sigmab,\varphib_\ep^a-\varphib^0_\ep)\cdot\ib \\&=: j_{30} + j_{31}
\end{split}\ees
with 
$
j_{30}:=-2a\varpi^\ep J_1^\ep \left( J_2^0\sigmab .J_2^\ep \nub_\ep\right)\cdot \ib
$
and $j_{31}$ is the rest. 

First we will estimate
$
B^\ep(\sigmab,\varphib_\ep^a - \varphib_\ep^{0}).
$
Using the  ``decay borrowing estimate" \eqref{B bound borrow} for $B^\ep$ we have
$$
\|B^\ep(\sigmab,\varphib_\ep^a - \varphib_\ep^{0}) \|_{r,q} \le C_r\| \sigmab\|_{q_0}\| \sech(|q_0-q| \cdot) J_2^\ep (\varphib_\ep^a - \varphib_\ep^{0})\|_{W^{r,\infty}}.
$$
%
%
%
Using the definition of the $W^{r,\infty}$ norm  we get:
$$
 \| \sech(|q_0-q| \cdot) J_2^\ep (\varphib_\ep^a - \varphib_\ep^{0})\|_{W^{r,\infty}} \le C_r\sup_{X \in \R} \left \vert
  \sech(|q_0-q| X)  \sum_{n = 0}^r (\partial_X^n J_2^\ep (\varphib_\ep^a - \varphib_\ep^{0}))\right \vert.
$$
We estimated $\partial_X^n J_2^\ep (\varphib_\ep^a - \varphib_\ep^{0})$ above in
\eqref{Lip est phi} in Lemma \ref{phi lemma}. We use that estimate here to get
$$
 \| \sech(|q_0-q| \cdot) J_2^\ep (\varphib_\ep^a - \varphib_\ep^{0})\|_{W^{r,\infty}} \le C_r\ep^{-r}\sup_{X \in \R} \left \vert
  \sech(|q_0-q| X)  (1+|X|)\right \vert|a|.
$$
Our restrictions on $q$ imply $|q_0 - q|$ is strictly bounded away from zero in a way independent of $q$ or $\ep$. 
Thus
$\sup_{X \in \R} \left \vert
  \sech(|q_0-q| X)  (1+|X|)\right \vert \le C.
$
And so we see that
 \bes\label{borrow1}
 \| \sech(|q_0-q| \cdot) J_2^\ep (\varphib_\ep^a - \varphib_\ep^{0})\|_{W^{r,\infty}} \le C_r\ep^{-r}|a|
\ees
which in turn gives
\be\label{B sig phi est 1}
\| B^\ep(\sigmab,\varphib_\ep^a - \varphib_\ep^{0})\|_{r,q} \le C_r\ep^{-r} |a|.
\ee
This estimate is one of the keys for  estimating $j_{31}$ and, as it happens, $l_{31}$.

Estimation of $j_{31}$ goes as follows. First we use the smoothing estimate \eqref{varpi smooth} for $\varpi^\ep$ from 
 Lemma \ref{pi props}:
\begin{align*}
\| j_{31}\|_{1,q} &\le  C|a|\|  \varpi^\ep J_1^\ep \left((J_2^\ep-J_2^0)\sigmab. J^\ep_2\nub_\ep\right)\|_{1,q}
+C |a|\|  \varpi^\ep B^\ep(\sigmab,\varphib_\ep^a-\varphib^0_\ep)\|_{1,q}  \\
&\le C|a| \|  J_1^\ep \left((J_2^\ep-J_2^0)\sigmab. J^\ep_2\nub_\ep\right)\|_{0,q} 
+C |a|\|  B^\ep(\sigmab,\varphib_\ep^a-\varphib^0_\ep)\|_{0,q}.
\end{align*}
Then we use \eqref{B sig phi est 1}, with $r = 0$, on the second term to get
$$
\| j_{31}\|_{1,q} \le  C|a|\|  J_1^\ep \left((J_2^\ep-J_2^0)\sigmab. J^\ep_2\nub_\ep\right)\|_{0,q} +C a^2 . 
$$
As for the first term, first we use the boundedness of $J_1^\ep$ from \eqref{all ops bound} followed by the product estimate \eqref{product inequality not borrow}:
\begin{align*}
\| j_{31}\|_{1,q} \le  C|a|\|(J_2^\ep-J_2^0)\sigmab\|_{0,q} \|J^\ep_2\nub_\ep\|_{W^{0,\infty}} +C a^2 . 
\end{align*}
Using the estimate for $\nub_\ep=\varphib_\ep^0$ in \eqref{phi bound} gives
\begin{align*}
\| j_{31}\|_{1,q} \le  C|a|\|(J_2^\ep-J_2^0)\sigmab\|_{0,q}  +C a^2 . 
\end{align*}
Then we use the estimate \eqref{lw J} for $J^\ep_2-J^0_2$ from Lemma \ref{long wave J} and the bounds on $\sigmab$ in \eqref{sigma bound}:
\begin{align}\label{j31 est map}
\| j_{31}\|_{1,q} \le  C\ep|a| + C a^2 \le C|RHS_{map}|.
\end{align}

Next we estimate $j_{30}$. Estimates like the ones we just used give
\be\label{bad j30}
\| j_{30}\|_{1,q} \le C|a|\|\varpi^\ep J_1^\ep( J_2^0\sigmab . J_2^\ep \nub_\ep)\|_{1,q}
\le C|a|\| J_1^\ep( J_2^0\sigmab . J_2^\ep \nub_\ep)\|_{0,q} \le C |a|.
\ee
This is not less than $C |RHS_{map}|$.
It turns out this estimate is not good enough for our purposes; basically, with this estimate $j_{30}$ looks like an $\O(1)$ linear perturbation in our equation and will ruin our contraction mapping argument.
But we can improve this  using the estimate \eqref{AT trick} from Lemma \ref{pi props}. 

To wit
$
\| j_{30}\|_{1,q} = 2|a|\| \varpi^\ep J_1^\ep \left( J_2^0\sigmab .J_2^\ep \nub_\ep\right)\cdot \ib\|_{1,q}.
$
Using the boundedness \eqref{all ops bound} of $J^1_\ep$ from Corollary \ref{all ops} we have
$
\| j_{30}\|_{1,q} \le C|a|\| \varpi^\ep (  J_2^0\sigmab .J_2^\ep \nub_\ep) \cdot \ib\|_{1,q}.
$
Then we apply $J_2^\ep$ to $\nub_\ep$ (as in \eqref{Jnub}) and $J_2^0$ to $\sigmab$ 
to get
$
\| j_{30}\|_{1,q} \le C|a|\| \varpi^\ep ( \sigma e^{i K_\ep\cdot}) \|_{1,q}.
$
Then we use the estimate \eqref{AT trick} of Lemma \ref{pi props} with $\omega = K_\ep$ to see
$
\| j_{30}\|_{1,q} \le C|a| |K_\ep|^{-1} \|\sigma\|_{2,q}.
$
And since $K_\ep = \O(1/\ep)$ we have
$
\| j_{30}\|_{1,q} \le C\ep |a|. $
This is one whole power of $\ep$ better than the naive estimate \eqref{bad j30}. With \eqref{j31 est map} 
we have
\be\label{j3 est map}
\|  j_{3}\|_{1,q} \le C \ep|a| + C a^2  \le C|RHS_{map}|.
\ee

To estimate $l_{3}$ is much the same, though recall that we do not need to estimate all of $l_3$, rather just the term $l_{31}$. As with $j_3$ above, we have:
$$
l_3 = -2a \lambda_+^\ep J_1^\ep \left(  J_2^0\sigmab .J_2^\ep \nub_\ep\right)\cdot \jb -2 a  \lambda_+^\ep  J_1^\ep \left((J_2^\ep-J_2^0)\sigmab. J^\ep_2\nub_\ep\right)\cdot\jb -2a  \lambda_+^\ep  B^\ep(\sigmab,\varphib_\ep^a-\varphib^0_\ep)\cdot\jb .
$$
Since $\chi_\ep:= -2a \lambda_+^\ep J_1^\ep \left(  J_2^0\sigmab .J_2^\ep \nub_\ep\right)\cdot \jb$ we see that 
\be\label{this is l31}
l_{31} = -2 a   \lambda_+^\ep J_1^\ep \left((J_2^\ep-J_2^0)\sigmab. J_2^\ep\nub_\ep\right)\cdot\jb -2a \lambda_+^\ep   B^\ep(\sigmab,\varphib_\ep^a-\varphib^0_\ep)\cdot\jb .
\ee

Now that we have an explicit formula for $l_{31}$, we use the
boundedness \eqref{all ops bound} of $\lambda_+^\ep$:
\begin{align*}
\| l_{31}\|_{1,q} \le & C|a|\|J_1^\ep \left((J_2^\ep-J_2^0)\sigmab. J^\ep_2\nub_\ep\right)\|_{1,q}
+C |a|\|  B^\ep(\sigmab,\varphib_\ep^a-\varphib^0_\ep)\|_{1,q}. & 
\end{align*}
Then we use \eqref{B sig phi est 1}, with $r = 1$, on the second term:
\begin{align*}
\| j_{l1}\|_{1,q} \le & C|a|\|  J_1^\ep \left((J_2^\ep-J_2^0)\sigmab. J^\ep_2\nub_\ep\right)\|_{1,q} +C \ep^{-1}a^2 . 
\end{align*}
For the first term, first we use the boundedness of $J_1^\ep$ from \eqref{all ops bound} followed by the product estimate \eqref{product inequality not borrow}:
\begin{align*}
\| l_{31}\|_{1,q} \le & C|a|\|(J_2^\ep-J_2^0)\sigmab\|_{1,q} \|J^\ep_2\nub_\ep\|_{W^{1,\infty}} +C \ep^{-1}a^2 . 
\end{align*}
Using the estimate for $\nub_\ep=\varphib_\ep^0$ in \eqref{phi bound}, with $r = 1$, gives
\begin{align*}
\| l_{31}\|_{1,q} \le & C\ep^{-1}|a|\|(J_2^\ep-J_2^0)\sigmab\|_{0,q}  +C\ep^{-1} a^2 . 
\end{align*}
Then we use the estimate \eqref{lw J} for $J^\ep_2-J^0_2$ from Lemma \ref{long wave J} and the bounds on $\sigmab$ in \eqref{sigma bound}
$
\| l_{31}\|_{1,q} \le  C|a| + C \ep^{-1} a^2.
$
Thus
\be\label{l31 est}
\ep \| l_{31}\|_{1,q} \le  C\ep|a| + C a^2 \le C|RHS_{map}|.
\ee
and we can move on.

%


\subsubsection{Mapping estimates for $j_4$ and $l_4$}
Applying the estimate  \eqref{B bound not borrow} for $B^\ep$ gives
$\|B^\ep(\etab,\varphib_\ep^a)\|_{r,q} \le C_r\| \etab\|_{r,q}\|J_2^\ep \varphib_\ep^a\|_{W^{r,\infty}}$. Then \eqref{phi bound} gives
\be\label{j4l4 bound}\begin{split}
\|B^\ep(\etab,\varphib_\ep^a)\|_{r,q}  \le C_r\ep^{-r}\|\etab\|_{r,q}.
\end{split}\ee
And so, using the same steps as in the previous subsubsection, we get:
\bes\label{j4 est map}
\| j_4\|_{1,q} \le  C|a|\| \varpi^\ep b_1^\ep(\etab,\varphib_\ep^a)\|_{1,q}  \le C|a|\|  b_1^\ep(\etab,\varphib_\ep^a)\|_{0,q} \le
C |a|\|\etab\|_{0,q}\le  C|RHS_{map}|
\ees
and
\bes\label{l4 est map 2}
\ep \| l_4\|_{1,q} \le  C\ep |a|\| b_2^\ep(\etab,\varphib_\ep^a)\|_{1,q}   \le C |a|\|\etab\|_{1,q}\le C|RHS_{map}|.
\ees

\subsubsection{Mapping estimates for $j_5$ and $l_5$}
Using \eqref{B bound set} we have
$
\| B^\ep(\etab,{\etab})\|_{r,q} \le C_r \|\etab\|_{r,q/2}^2.
$
Thus
\bes\label{j5 est map}
\| j_5\|_{1,q} \le C \|\etab\|_{1,q/2}^2\le  C|RHS_{map}|
\mand 
\ep \|  l_5\|_{1,q} \le C \ep \|\etab\|_{1,q/2}^2\le  C|RHS_{map}|.
\ees
With these estimates, the validation of the mapping estimate \eqref{ball1} is complete.
We move on to the Lipschitz estimates.

\subsection{The Lipschitz estimates}
Now we prove the estimate \eqref{lip1}. In this subsection  $\grave{j}_n$  is the same as ${j}_n$ but evaluated at $\grave{\etab}$ and $\ga$ instead of at $\etab$ and $a$.
Likewise  $\grave{l}_n$  is the same as ${l}_n$ but evaluated at $\grave{\etab}$ and $\ga$ instead of at $\etab$ and $a$. Also we have $q_\star/2 \le q < q' \le q_\star$.

We have by definition and the triangle inequality:
\begin{multline*}
\|N_1^\ep(\etab,a)-N_1^\ep(\grave{\etab},\ga)\|_{1,q}\\ \le \| \A^{-1} (j_2 -\grave{j}_2)\|_{1,q} +  \| \A^{-1} (j_3 -\grave{j}_3)\|_{1,q} +  \| \A^{-1} (j_4 -\grave{j}_4)\|_{1,q} +  \| \A^{-1} (j_5 -\grave{j}_5)\|_{1,q} 
\end{multline*}
\begin{multline*}
\|N_2^\ep(\etab,a)-N_2^\ep(\grave{\etab},\ga)\|_{1,q}\\ \le \| \ep^2\P_\ep (l_2 -\grave{l}_2)\|_{1,q} +  \| \ep^2\P_\ep (l_{31} -\grave{l}_{31})\|_{1,q} +  \| \ep^2\P_\ep (l_4 -\grave{l}_4)\|_{1,q} +  \| \ep^2\P_\ep (l_5 -\grave{l}_5)\|_{1,q} 
\end{multline*}
and
$$
|N_3^\ep(\etab,a)-N_2^\ep(\grave{\etab},\ga)| \le  \left\vert \iota_\ep [l_2 -\grave{l}_2]\right\vert +  \left\vert \iota_\ep [l_{31} -\grave{l}_{31}]\right\vert +  \left\vert \iota_\ep [l_4 -\grave{l}_4]\right\vert +  \left\vert \iota_\ep [l_5 -\grave{l}_5]\right\vert.
$$
Using the uniform bound \eqref{A inv bound} on $\A^{-1}$ from Lemma \ref{A inv} we have
\bes
\|N_1^\ep(\etab,a)-N_1^\ep(\grave{\etab},\ga)\|_{1,q}\\ \le C\left( \| j_2 -\grave{j}_2\|_{1,q} +  \| j_3 -\grave{j}_3\|_{1,q} +  \| j_4 -\grave{j}_4\|_{1,q} +  \| j_5 -\grave{j}_5\|_{1,q}\right).
\ees
In Lemma \ref{B inv}, the estimate \eqref{P bound} gives $\|\P_\ep f\|_{1,q} \le C \ep^{-1} \| f\|_{1,q}$. Thus
\bes
\|N_2^\ep(\etab,a)-N_2^\ep(\grave{\etab},\ga)\|_{1,q}\\ \le C \ep \left(\|  l_2 -\grave{l}_2\|_{1,q} +  \| l_{31} -\grave{l}_{31}\|_{1,q} +  \| l_4 -\grave{l}_4\|_{1,q} +  \| l_5 -\grave{l}_5\|_{1,q} \right).
\ees
And the Riemann-Lebesgue estimate \eqref{iota bound} in Lemma \ref{rl lemma} gives (with $r=1$)
\bes
|N_3^\ep(\etab,a)-N_2^\ep(\grave{\etab},\ga)|\\ \le C \ep \left(\|  l_2 -\grave{l}_2\|_{1,q} +  \| l_{31} -\grave{l}_{31}\|_{1,q} +  \| l_4 -\grave{l}_4\|_{1,q} +  \| l_5 -\grave{l}_5\|_{1,q} \right). 
\ees

Thus we will have \eqref{lip1} if we can show that each of the eight terms
\begin{multline*}
\| j_2 -\grave{j}_2\|_{1,q},\
\| j_3 -\grave{j}_3\|_{1,q},\
\| j_4 -\grave{j}_4\|_{1,q},\
\| j_5 -\grave{j}_5\|_{1,q},\\
\ep\| l_2 -\grave{l}_2\|_{1,q},\
\ep\| l_3 -\grave{l}_3\|_{1,q},\
\ep\|l _4 -\grave{l}_4\|_{1,q},\ \text{and}\
\ep\| l_5 -\grave{l}_5\|_{1,q}
\end{multline*}
is bounded by $C |RHS_{lip}|$, where
$$
|RHS_{lip}|:= {1 \over |q-q'|} 
\left(\ep + \|{\etab}\|_{1,q'}+\|\grave{\etab}\|_{1,q'}
+|a|+|\ga|\right)(|a-\ga|+\|\etab-\grave{\etab}\|_{1,q}). 
$$

\subsubsection{Lipschitz estimates for $j_2$ and $l_2$}
Note that $j_2$ and $l_2$ are linear in $\etab$ and do not depend at all on $a$. Thus we can 
use the estimates \eqref{j2 map est} and \eqref{l2 map est} to get:
\bes
\|j_2 -\grave{j}_2\|_{1,q} \le C\ep \| \etab - \grave{\etab}\|_{1,0} \le C|RHS_{lip}|
\mand
\ep \|l_2 -\grave{l}_2\|_{1,q} \le C\ep \| \etab - \grave{\etab}\|_{1,0} \le C|RHS_{lip}|.
\ees

\subsubsection{Lipschitz estimates for $j_3$ and $l_3$}
Explicit computations give
\begin{multline*}
j_3 -\grave{j}_3= -2(a-\ga)\varpi^\ep J_1^\ep \left( J_2^0\sigmab. J_2^\ep \nub_\ep\right)\cdot \ib -2 (a-\ga) \varpi^\ep J_1^\ep \left((J_2^\ep-J_2^0)\sigmab. J^\ep_2\nub_\ep\right)\cdot\ib\\ 
-2(a-\ga) \varpi^\ep B^\ep(\sigmab,\varphib_\ep^a - \varphib_\ep^0)\cdot \ib + 2\ga \varpi^\ep B^\ep(\sigmab,\varphib_\ep^\ga - \varphib_\ep^a)\cdot \ib
\end{multline*}
and
\begin{multline*}
l_{31} -\grave{l}_{31}=  -2 (a-\ga) \lambda_+^\ep J_1^\ep \left((J_2^\ep-J_2^0)\sigmab. J^\ep_2\nub_\ep\right)\cdot\jb\\ 
-2(a-\ga) \lambda_+^\ep B^\ep(\sigmab,\varphib_\ep^a - \varphib_\ep^0)\cdot \jb + 2\ga \lambda_+^\ep B^\ep(\sigmab,\varphib_\ep^\ga - \varphib_\ep^a)\cdot \jb.
\end{multline*}

We begin with an estimate of
$
B^\ep(\sigmab,\varphib_\ep^a - \varphib_\ep^{\ga}).
$
This estimate parallels the one for $
B^\ep(\sigmab,\varphib_\ep^a - \varphib_\ep^{0})
$
in \eqref{B sig phi est 1}.
Using the  ``decay borrowing estimate" \eqref{B bound borrow} for $B^\ep$ we have
$$
\|B^\ep(\sigmab,\varphib_\ep^a - \varphib_\ep^\ga) \|_{r,q} \le C_r\| \sigmab\|_{q_0}\| \sech(|q_0-q| \cdot) J_2^\ep (\varphib_\ep^a - \varphib_\ep^\ga)\|_{W^{r,\infty}}.
$$
%
%
%
Using the definition of the $W^{r,\infty}$ norm  we get:
$$
 \| \sech(|q_0-q| \cdot) J_2^\ep (\varphib_\ep^a - \varphib_\ep^\ga)\|_{W^{r,\infty}} \le C_r\sup_{X \in \R} \left \vert
  \sech(|q_0-q| X)  \sum_{n = 0}^r (\partial_X^n J_2^\ep (\varphib_\ep^a - \varphib_\ep^\ga))\right \vert.
$$
We estimated $\partial_X^n J_2^\ep (\varphib_\ep^a - \varphib_\ep^\ga)$ above in
\eqref{Lip est phi} in Lemma \ref{phi lemma}. We use that estimate here to get
\be\label{this guy}
 \| \sech(|q_0-q| \cdot) J_2^\ep (\varphib_\ep^a - \varphib_\ep^\ga)\|_{W^{r,\infty}} \le C_r\ep^{-r}\sup_{X \in \R} \left \vert
  \sech(|q_0-q| X)  (1+|X|)\right \vert|a-\ga|.
\ee
We know that
$\sup_{X \in \R} \left \vert
  \sech(|q_0-q| X)  (1+|X|)\right \vert \le C
$
and so
 \bes
 \| \sech(|q_0-q| \cdot) J_2^\ep (\varphib_\ep^a - \varphib_\ep^\ga)\|_{W^{r,\infty}} \le C_r\ep^{-r}|a-\ga|
\ees
which in turn gives
\be\label{B sig phi est 2}
\| B^\ep(\sigmab,\varphib_\ep^a - \varphib_\ep^\ga)\|_{r,q} \le C_r\ep^{-r} |a-\ga|.
\ee

The formulas for $j_3-\grave{j}_3$ and $l_{31}-\grave{l}_{31}$ differ only slightly from those 
for $j_3$ and $l_{31}$. Any differences there are can be handled with \eqref{B sig phi est 2}; we leave out the steps.
The results are
\bes\label{j3 est lip}
\|  j_{3}-\grave{j}_3\|_{1,q} \le C \ep|a-\ga| + C\left(|a|+|\ga|\right)|a-\ga| \le C|RHS_{lip}|
\ees
and
\bes\label{le1 est 3}
\ep \| l_{31}-\grave{l}_{31}\|_{1,q} \le C \ep|a-\ga| + C\left(|a|+|\ga|\right)|a-\ga|\le C|RHS_{lip}|.
\ees

\subsubsection{Lipschitz estimates for $j_4$ and $l_4$}
Much of this parallels the earlier treatment of 
 $j_3$ and $l_3$, just swapping out $\sigmab$ for $\etab$.
 There is one major wrinkle however: varying $a$ results in a loss of decay rate $q$ for $\etab$. 
 This is not terribly suprising given the estimate \eqref{Lip est phi}.
Thus we must keep careful track of the decay rates. So fix $q \in [q_*/2,q_*)$.

We look at
$\ds
B^\ep(\etab,\varphib_\ep^a - \varphib_\ep^{\grave{a}}). 
$
We use the decay borrowing estimate \eqref{B bound borrow} to get, if $q'> q$,
$$
\| B^\ep(\etab,\varphib^a_\ep-\varphib^\ga_\ep)\|_{r,q} \le C_r\|\etab\|_{r,q'} \|\sech(|q-q'|\cdot) J_2^\ep(\varphib^a_\ep-\varphib^\ga_\ep)\|_{W^{r,\infty}}.
$$
As with the estimates that led to \eqref{this guy}, we can use \eqref{Lip est phi} to get
\bes
 \| \sech(|q-q'| \cdot) J_2^\ep (\varphib_\ep^a - \varphib_\ep^\ga)\|_{W^{r,\infty}} \le C\ep^{-r}\sup_{X \in \R} \left \vert
  \sech(|q-q'| X)  (1+|X|)\right \vert|a-\ga|.
\ees
Now, however, we do not know that $|q-q'|$ is bounded strictly from zero.

Elementary calculus can be used to show that
$$
\sup_{X\in\R} \left \vert (1+ |X|) \sech(|q-q'|X) \right \vert \le C|q-q'|^{-1}.
$$
Thus we have
\be\label{j4l4 est lip 2}
\| B^\ep(\etab,\varphib^a_\ep-\varphib^\ga_\ep)\|_{r,q} \le C\ep^{-r}|q-q'|^{-1}\|\etab\|_{r,q'} |a-\ga| \quad \text{when $q'>q$}.
\ee

Now the triangle inequallity and binliearity of $B^\ep$ give:
\begin{multline*}
     \| B^\ep(\etab,a\varphib^a_\ep) - B^\ep(\grave{\etab},\ga\varphib^\ga_\ep)\|_{r,q}  
\le |a| \| B^\ep(\etab,\varphib^a_\ep-\varphib^\ga_\ep)\|_{r,q} \\
+    |a-\ga|\| B^\ep(\etab,\varphib^\ga_\ep) \|_{r,q} 
+    \| B^\ep(\etab-\grave{\etab},\ga\varphib^\ga_\ep)\|_{r,q} 
\end{multline*}
So if we use \eqref{j4l4 bound} and \eqref{j4l4 est lip 2} we get:
\begin{multline*}
  \| B^\ep(\etab,a\varphib^a_\ep) - B^\ep(\grave{\etab},\ga\varphib^\ga_\ep)\|_{r,q} 
\le  C\ep^{-r}|q-q'|^{-1} \|\etab\|_{r,q'}|a||a-\ga| \\+ C\ep^{-r} \left(\|\etab\|_{r,q}|a-\ga|+|\ga|\|\etab - \grave{\etab}\|_{r,q}
\right) \quad \text{when $q'>q$}.
\end{multline*}
This estimate, together with the sorts of steps we have used above, lead to:
\bes
\| {j}_4 - \grave{j}_4\|_{1,q} \\\le C|q-q'|^{-1}\left( \|\etab\|_{1,q'}|a||a-\ga|+\|\etab\|_{1,q}|a-\ga|+|\ga|\|\etab - \grave{\etab}\|_{1,q}\right) \le C|RHS_{lip}|
\ees
and
\bes
\ep \|{l}_4 - \grave{l}_4\|_{1,q_*/2}\\ \le C|q-q'|^{-1}\left( \|\etab\|_{1,q'}|a||a-\ga|+\|\etab\|_{1,q}|a-\ga|+|\ga|\|\etab - \grave{\etab}\|_{1,q}\right) \le C|RHS_{lip}|
\ees
so long as $q_*/2 \le q < q' \le q_*$.

\subsubsection{Lipschitz estimates for $j_5$ and $l_5$}
Using \eqref{B bound set} from  Lemma \ref{B lemma} gives
$$
\| B^\ep(\etab,{\etab})-B^\ep(\grave{\etab},\grave{\etab})\|_{r,q} \le C 
(\|\etab\|_{r,q/2}+
\|\grave{\etab}\|_{r,q/2})\|\etab-\grave{\etab}\|_{r,q/2}.
$$
Thus
\bes\label{j5 est lip}
\| j_5-\grave{j_5} \|_{1,q} \le C 
(\|\etab\|_{1,q/2}+
\|\grave{\etab}\|_{1,q/2})\|\etab-\grave{\etab}\|_{1,q/2}\le C|RHS_{lip}|
\ees
and
\bes
\ep \| l_5-\grave{l}_5 \|_{1,q} \le C \ep
(\|\etab\|_{1,q/2}+
\|\grave{\etab}\|_{1,q/2})\|\etab-\grave{\etab}\|_{1,q/2}\le C|RHS_{lip}|.
\ees

This completes the estimate that give rise to \eqref{lip1} and we move on to the bootstrap estimates.

\subsection{The bootstrap estimates}
In this section we prove the estimates \eqref{boot1} and \eqref{boot2}. 
The triangle inequality gives:
$$
\|N_1(\etab,a)\|_{r+1,q} \le \| \A^{-1} j_1\|_{r+1,q} +\| \A^{-1} j_2\|_{r+1,q}+ \| \A^{-1} j_3\|_{r+1,q} +\| \A^{-1} j_4\|_{r+1,q}+ \| \A^{-1} j_5\|_{r+1,q}
$$
$$
\|N_2(\etab,a)\|_{r+1,q} \le \| \ep^2 \P_\ep  l_1\|_{r+1,q} + \| \ep^2 \P_\ep  l_2\|_{r+1,q} + \| \ep^2 \P_\ep  l_{31}\|_{r+1,q} +\| \ep^2 \P_\ep  l_4\|_{r+1,q} + \| \ep^2 \P_\ep  l_5\|_{r+1,q}
$$
and
$$
|N_3(\etab,a)| \le | \iota_\ep l_1|+ | \iota_\ep l_2|+| \iota_\ep l_{31}|+| \iota_\ep l_4|+| \iota_\ep l_5|
$$

Using the bound \eqref{A inv bound} for $\A^{-1}$ we have
$$
\|N_1(\etab,a)\|_{r+1,q} \le C_r \left (\| j_1\|_{r+1,q} +\| j_2\|_{r+1,q}+ \|  j_3\|_{r+1,q} +\| j_4\|_{r+1,q}+ \| j_5\|_{r+1,q}\right).
$$
As seen in Lemma \ref{B inv}, the operator $\P_\ep$ is smoothing by up to two derivatives.
However each derivative of smoothing comes at a cost of an additional negative power of $\ep$. Choosing to smooth by just one derivative 
we have:
$$
\|N_2(\etab,a)\|_{r+1,q} \le C (\| l_1\|_{r,q}+\| l_2\|_{r,q} + \|  l_{31}\|_{r,q} +\|  l_4\|_{r,q} + \|   l_5\|_{r,q}).
$$
And the Riemann-Lebesgue estimate \eqref{iota bound} in Lemma \ref{rl lemma} 
\bes
|N_3^\ep(\etab,a)|\\ \le C \ep^r(\| l_1\|_{r,q}+\| l_2\|_{r,q} + \|  l_{31}\|_{r,q} +\|  l_4\|_{r,q} + \|   l_5\|_{r,q}).
\ees

Thus we will have \eqref{boot1} and \eqref{boot2} if we can show that each of the ten terms
\begin{multline*}
\| j_1\|_{r+1,q},\ \| j_2\|_{r+1,q},\ \|  j_3\|_{r+1,q},\ \| j_4\|_{r+1,q},\ \| j_5\|_{r+1,q},\\
\| l_1\|_{r,q},\  \| l_2\|_{r,q},\  \|  l_{31}\|_{r,q},\ \|  l_4\|_{r,q},\ \text{and}\ \|   l_5\|_{r,q}
\end{multline*}
is bounded by $C_r |RHS_{boot}|$, where
$$
|RHS_{boot}|:= \ep + \|\etab\|_{r,q} + \ep^{1-r} |a|   +\ep^{-r}a^2+\ep^{-r}|a|\|\etab\|_{r,q}+\|\etab\|_{r,q}^2.
$$
%
%

\subsubsection{Bootstrap estimates $j_1$ and $l_1$}
Since $\sigmab$ is a smooth function, the estimates on $j_1$ and $l_1$ can be improved from above more or less for free. Specifically we have, for any $r \ge 0$,
we have
$$
\| j_1 \|_{r+1,q_0} + \ \|l_1\|_{r,q_0} \le C_r\ep\le C_r |RHS_{boot}|.
$$

\subsubsection{Bootstrap estimates for $j_2$ and $l_2$}
Recall that $j_2 = -2 \varpi^\ep b_1^\ep(\sigmab,\etab) + 2 \varpi^0 b_1^0(\sigmab,\etab)$. From the estimate \eqref{varpi smooth} in
Lemma \ref{pi props} we see that the operators $\varpi^\ep$ and $\varpi^0$ smooth by up to two derivatives at no cost in $\ep$. Thus we conclude, with the help of \eqref{B bound set},
$$
\| j_1 \|_{r+1,q} \le C_r \|b_1^\ep(\sigmab,\etab)\|_{r,q} + C_r \|b_1^0(\sigmab,\etab)\|_{r,q}\le
C_r \| \etab \|_{r,q}\le C_r |RHS_{boot}|.
$$
Since $l_2 = 2b_2^\ep(\sigmab,\etab)$, we use \eqref{B bound set} and see that
$$
\|l_2\|_{r,q} \le C_r\| \etab \|_{r,q}\le C_r |RHS_{boot}|.
$$

\subsubsection{Bootstrap estimates for $j_3$ and $l_{31}$}
Since $\varpi^\ep$ smooths by up to two derivatives we have, using \eqref{varpi smooth},
$
\| j_3 \|_{r+1,q} \le C|a|\| b_1^\ep(\sigmab,\varphib^a_\ep)\|_{r-1,q}.
$
Then using the product inequality for $B^\ep$ in  \eqref{B bound not borrow} followed by the estimate for $J_2^\ep \varphib_\ep^a$ in \eqref{phi bound} gives
$$
\| j_3 \|_{r+1,q} \le C|a|\|\sigmab\|_{r-1,q} \|J_2^\ep \varphib^a_\ep\|_{W^{r-1,\infty}} \le C_r \ep^{1-r}|a|.
$$

We saw above that
$
l_{31} =-2 a  J_1^\ep \left((J_2^\ep-J_2^0)\sigmab. J_2^\ep\nub_\ep\right)\cdot\jb -2a  B^\ep(\sigmab,\varphib_\ep^a-\varphib^0_\ep)\cdot\jb .
$
The estimate of this in $H^r_{q}$ is not much different than our earlier estimate in $H^1_q$.  Specifically, using \eqref{all ops bound} and \eqref{product inequality not borrow}:
$$
\|J_1^\ep \left((J_2^\ep-J_2^0)\sigmab. J_2^\ep\nub_\ep\right)\cdot\jb\|_{r,q} \le C \| (J_2^\ep - J_2^0)\sigmab\|_{r}\|J_2 \nub_\ep\|_{W^{r,\infty}}.
$$
Using the approximation inequality \eqref{lw J} on the first term and then the bounds on $\varphib_\ep^a$ in \eqref{phi bound} for the second gives
$$
\|J_1^\ep \left((J_2^\ep-J_2^0)\sigmab. J_2^\ep\nub_\ep\right)\cdot\jb\|_{r,q} \le C\ep^{1-r}.
$$
Then we recall \eqref{B sig phi est 1} gives
$
\|B^\ep(\sigmab,\varphib_\ep^a-\varphib^0_\ep)\|_{r,b} \le C_r \ep^{-r}|a|
$
and so all together
we find that
$$
\| l_{31} \|_{r,q} \le C_r \ep^{1-r} |a| + C_r\ep^{-r} a^2 \le C|RHS_{boot}|.
$$

\subsubsection{Bootstrap estimates  $j_4$ and $l_{4}$}
Since $\varpi^\ep$ smooths by two derivatives we have, after using \eqref{varpi smooth},
$
\| j_4 \|_{r+1,q} \le C|a|\| b_1^\ep(\etab,\varphib^a_\ep)\|_{r,q}.
$
Then using the product inequality  \eqref{B bound not borrow} and the $\varphib$ bound \eqref{phi bound} gives.
$$
\| j_4 \|_{r+1,q} \le C|a|\|\etab\|_{r,q} \|J_2^\ep\varphib^a_\ep\|_{W^{r,\infty}} \le C_r \ep^{-r}|a|\|\etab\|_{r,q} \le C|RHS_{boot}|.
$$
The term $l_4$ has no smoothing operator attached, but is otherwise estimated in the same way. We have
$$
\| l_4 \|_{r,q} \le C|a|\|\etab\|_{r,q} \|\varphib^a_\ep\|_{W^{r,\infty}} \le C_r \ep^{-r}|a|\|\etab\|_{r,q} \le C|RHS_{boot,1}|.
$$

\subsubsection{Bootstrap estimates for $j_5$ and $l_{5}$}
Since $\varpi^\ep$ smooths by up to two derivatives we have, after using \eqref{varpi smooth} and \eqref{B bound set},
$$
\| j_5 \|_{r+1,q} \le C\|\etab\|_{r,q/2}^2 \le C|RHS_{boot}|.
$$
The term $l_4$ has no smoothing operator attached, but is otherwise estimated in the same way. We have
$$
\| l_5 \|_{r,q} \le C\|\etab\|_{r,q/2}^2 \le C|RHS_{boot}|.
$$
That completes our proof of \eqref{boot1}, \eqref{boot2} and Proposition \ref{main mover}.

\section{Function analysis}\label{proofs}
In this section we prove  Lemmas \ref{lambda lemma}, \ref{varpi symbol props}, \ref{xi symbol props}, \ref{xi decomp}, \ref{xi lemma} and \ref{pi lemma}. Each of these lemmas gives quantitative estimates for some specific meromorophic function.
We use little more than foundational methods from real and complex analysis here though their implementation is sometimes complicated.

\subsection{Multiplier properties in $\C$}
This subsection contains the proofs of Lemmas \ref{lambda lemma}, \ref{xi lemma} and \ref{pi lemma}.
\begin{proof} (of Lemma \ref{lambda lemma})
Parts (i),  (iii)  and (iv) are easily inferred from properties of cosine. 

For Part (ii) note that so long as $\Re((1-w)^2 + 4w \cos^2(z))>0$ we can use
the prinicipal square root to extend $\tilde{\varrho}(k)$ analytically into the complex plane.  Complex trigonometry identities show
that
\begin{multline*}
\Re((1-w)^2 + 4w \cos^2(k+i \tau)) = 1+w^2 + 2 w \cos(2k) \cosh(2 \tau) \\
\ge  1+w^2 - 2 w  \cosh(2 \tau) =: f(\tau).
\end{multline*}
Note that $f(0) = (1-w)^2$. Since $w>1$, this is strictly positive. Thus we can find $\tau_0 > 0$ such that $f(\tau) > (1-w)^2/2$ when  $|\tau| \le \tau_0$.
In turn this implies that $\tilde{\varrho}(z)$ (and thus $\tlambda_\pm(z)$) is analytic when $|\Im(z)| \le \tau_0.$  Since the functions are periodic in the real
direction and the strip $\overline{\Sigma}_{\tau_0}$ is bounded in the imaginary direction, the extreme value theorem implies that the functions (and all their derivatives) are uniformly bounded
on it. Thus we have (ii).

For Part (v), we compute
$$\left \vert \tlambda_\pm'(k)\right \vert = \left \vert {4w \sin(k)\cos(k) \over \sqrt{ (1+w)^2 - 4w \sin^2(k)}}\right \vert =2c_w^2 |\sin(k)|\left \vert  {\cos(k) \over \sqrt{1-{4w \over (1+w)^2}\sin^2(k) }}\right \vert.$$
It is clear that
$$
\sup_{s \in [0,1]} \left \vert {{1-s}\over {1-{r}s }} \right \vert \le 1
$$
when $0 \le r \le 1$. The fact that $w>1$ implies that $ 0 < {4w / (1+w)^2}< 1$. Thus 
$$
\sup_{k \in \R} \left \vert  {\cos(k) \over \sqrt{1-{4w \over (1+w)^2}\sin^2(k) }} \right \vert=
\sup_{k \in \R} \left \vert  {\sqrt{1 - \sin^2(k) \over 1-{4w \over (1+w)^2}\sin^2(k) }} \right \vert \le 1.
$$
Since $|\sin(k)| \le |k|$ for all $k$, this gives
$\left \vert \tlambda_\pm'(k)\right \vert \le 2c_w^2 |k|$, the second inequality in Part (v). The first inequality is simpler and omitted.
Since $c_w^2 = 2w/(1+w)$ and $w>1$ we have $c_w^2 > 1$.

For Part (vi), by (iv) we have 
$\tilde{\lambda}_+(k) \in [2w,2+2w]$ for all $k$.
We have $c^2 k^2 < 2w$ when $|k|<\sqrt{2w}/c=:k_1$ and $c^2 k^2 > 2+2w$
when $|k|> \sqrt{2+2w}/c=:k_2$. Since our functions are continuous, the intermediate value theorem
implies that there is at least one value of $k$ such that $c^2 k^2 = \tilde{\lambda}_+(k)$ in $[k_1,k_2]$. Likewise, there can be no solutions of 
outside $[k_1,k_2]$.

Now put $c_-= 5/4\sqrt{2}\in(0,1)$ and assume $c > c_-$.
If $k\ge k_1$, and 
because $w > 1$, we have
$$
2c^2 k \ge 2 c^2 k_1 = 2 c \sqrt{2 w} \ge2 \sqrt{2}c_-  = {5 \over 2}.
$$
This implies that 
\be\label{xi deriv estimate}
\left \vert \ds {d \over dk}\left(c^2 k^2 - \tilde{\lambda}_+(k) \right) \right \vert  = \left \vert 2c^2 k - \tilde{\lambda}_+'(k)\right \vert\ge {5 \over 2}- 2=l_0>0
\ee
when $k\ge k_1$. This implies that there can be at most one solution
of $c^2 k^2 - \tilde{\lambda}(k) = 0$ for $k\ge k_1$ and also gives the estimate for $|2c^2 k_c - \tilde{\lambda}_+'(k_c)|$.  
The smoothness of the map $c \mapsto k_c$
follows in a routine way from this derivative estimate and the implicit function theorem. Thus we have all of Part (vi).

\end{proof}

\begin{proof}  (of Lemma \ref{xi lemma})
Take $c_-$ as in the proof of Lemma \ref{lambda lemma} and let $c_+=\sqrt{c_w^2 + 1}$ when $c > c_-$.
Note that Part (vi) of that lemma tells us that $\txi_c(k_c)=0$ and $|\txi'(k_c)| \ge l_0$.
Henceforth assume $c \in (c_-,c_+)$. Clearly $c_w \in (c_-,c_+)$.
Parts (i) and (ii) follow immediately from Lemma \ref{lambda lemma}.
And so all that remains is to prove the estimate \eqref{xi estimate} in Part (iii).

{\bf Estimates when $|z|$ is large:}
First, note that since $\tlambda_+(z)$ is bounded in $\overline{\Sigma}_{\tau_0}$ 
there exists $k_{big}>0$ such that $|\Re z|\ge k_{big}$ and
$\Im(z) \le \tau_0$ implies
$$
|\txi_c(z)|\ge {1 \over 2} c_w^2 |z|^2
$$
for any $c\in(c_-,c_+)$.
So clearly, for $j =0,1,2$, we have
\be\label{biggish}
\inf_{k\ge k_{big}} (1+k^2)^{-j/2} |\txi_c(k+i\tau)| \ge C_1>0
\ee
where $C_1$ does not depend on $c$.
Thus we only need to concern ourselves with $|\Re(z)| \le k_{big}$. Our next stop is near $k_c$.

{\bf Estimates when $z\sim k_c$:}
Since $\txi_c''(z) = -2c^2 + \tlambda_+''(z)$ there exists $C_2<0$ such that
$
|\txi_c''(z)| \le C_2
$
for all $z \in \overline{\Sigma}_{\tau_0}$ and $c \in (c_-,c_+)$. This implies
$
| \txi_c'(z) - \txi'_c(z')| \le C_2|z-z'|
$
for all $z,z' \in \overline{\Sigma}_{\tau_0}$.
Thus if $|z-k_c| \le \delta_1:=l_0/2C_2$ this implies $\left \vert\txi_c'(z)-\txi'_c(k_c) \right \vert \le l_0/2$.
Note that $\delta_1$ does not depend on $c$.
The reverse triangle inequality gives
$$
|\txi_c'(z)| = \left\vert |\txi'(k_c)|- \left \vert\txi_c'(z)-\txi'_c(k_c) \right \vert\right \vert \ge {l_0 \over 2}
$$
provided $|z-k_c| \le \delta_1$. The FTOC then implies
$$
|\txi_c(z)| \ge {l_0 \over 2}|z-k_c|
$$
for all $|z-k_c| \le \delta_1$ and $c \in (c_-,c_+)$.

Thus if $z = k + i \tau$ and  $|z-k_c| \le \delta_1$ then
\be\label{middle}
|\txi_c(z)| \ge {l_0 \over 2}|z-k_c|\ge {l_0 \over 2} |\tau|
\ee
for any $c \in (c_-,c_+)$. 

{\bf Estimates for $z$ everywhere else:}
Now we know that $k_c$ depends smoothly on $c$. So select $\delta>0$ such that $|c-c_w| \le \delta$ implies $|k_c- k_{c_w}| \le {1 \over 100} \delta_1$.
Let $K:=[0,k_{c_w}-\delta_1/2] \cup [k_{c_w}+\delta_1/2,k_{big}]$. Let
$$
m_*:=\inf_{|k| \in K,k\in\R} \inf_{|c-c_w| \le \delta} |\txi_c(k)|
$$
This number is strictly positive, since the only real zeros for $\txi_c(k)$, by the definition of $\delta$, lie outside of $K$, which is compact.

In the set $\left\{ |\Re(z)| \le k_{big}, |\Im(z)|\le \tau_0\right\}$, which is compact, $\xi_c(z)$ is Lipschitz with a constant (say $C_3>0$) that, so long as $c$ lies in a compact set,
is bounded independent of $c$. Thus, if $|k| \in K$ and $|\tau|\le m_*/2C_3$ we have
$$
|\txi_c(k+i\tau) - \txi_c(k)| \le C_3|\tau| \le m_*/2.
$$
The reverse triangle inequality then gives
\be\label{smallish}
|\txi_c(k+i \tau)| \ge |\txi_c(k)| - C_3|\tau| \ge m_*/2.
\ee

{\bf Overall estimates:}
So put $\tau_3:=\min(\delta_1/2,\tau_0,m_*/2C_3)$. If $z = k + i \tau$ with $|\tau| \le \tau_3$ then notice that we have either
(a) $|k| \in K$, (b) $|k-k_c|\le \delta_1$ or (c) $|k|\ge k_{big}$. Thus we can use either \eqref{biggish}, \eqref{middle} or  \eqref{smallish} to see that
$$
(1+k^2)^{-j/2}|\txi_c(k+i \tau)| \ge C|\tau|
$$
where $C>0$. This completes the proof.

\end{proof}

\begin{proof} (of Lemma \ref{pi lemma})
We have the series expansion
$
\tlambda_-(z) = c_w^2z^2 - \alpha_w z^4 + \cdots.
$
So if we put $$\tzeta(z):={\tilde{\lambda}_-(z) \over z^2}$$ 
we see the singularity at $z = 0$ is removable. 
Since $\tlambda_-(z)$ is analytic and uniformly bounded in the strip 
$\overline{\Sigma}_{\tau_0}$, we have the same for $\tzeta(z)$. With this, we rewrite $\tvarpi^\ep$ as:
$$
\tvarpi^\ep(Z) = -{\ep^2 \tlambda_-(\ep Z) \over (c_w^2+\ep^2) \ep^2 Z^2 - \tlambda(\ep Z)}= -{\ep^2 \tzeta(\ep Z) \over c_w^2 + \ep^2  - \tzeta(\ep Z)}.
$$
Obviously
$
\tzeta(z) = c_w^2 - \alpha_w z^2 + \cdots.
$
And so Taylor's theorem (with the uniform bound) then implies there exists $C_1>0$ such that  for all $z$ in the strip $\overline{\Sigma}_{\tau_0}$
we have:
\be\label{taylor zeta jones}
|\tzeta(z)| \le C_1,\quad |\tzeta(z) - c_w^2| \le C_1|z|^2, |\tzeta(z) - c_w^2 + \alpha_w z^2| \le C_1 |z|^4.
\ee
Likewise, the uniform bound on $\tzeta'(z)$ implies
\be\label{zeta lip}
|\tzeta(z) - \tzeta(z')| \le C_1|z-z'|
\ee
for all $z,z'$ in the strip $\overline{\Sigma}_{\tau_0}$.


{\bf Estimates near $Z=0$:}
The reverse triangle inequality gives:
\begin{multline}\label{reverse}
|c_w^2 + \ep^2 - \tzeta(\ep Z)| = |\ep^2 + \ep^2 \alpha_w Z^2 + (c_w^2 - \ep^2 \alpha_w Z^2 - \tzeta(\ep Z)|\\
\ge \left \vert \ep^2|1+\alpha_w Z^2| - | c_w^2 - \ep^2 \alpha_w Z^2 - \tzeta(\ep Z)|\right \vert
\end{multline}
If $Z= K + i q$, then we have
$$ 
|1 + \alpha_w Z^2| \ge |1 + \alpha_w \Re Z^2| = |1 + \alpha_w (K^2-q^2)|.
$$
If we restrict $|q|$ so that
$1-\alpha_w q^2 >1/2$ 
then we see that 
\be\label{jet 1}
 |1+\alpha_w Z^2| \ge {1 \over 2}  (1 + 2\alpha_w K^2) \ge {1 \over 2}
\ee
for all $Z$.

From \eqref{taylor zeta jones} we have
$$
| c_w^2 - \ep^2 \alpha_w Z^2 - \tzeta(\ep Z)| \le C_1 \ep^4 |Z|^4.
$$
So let us demand that $$
|Z| \le {\delta_1 \over \ep}.
$$
We will specify $\delta_1>0$ in a moment.
Then
$$
C_1\ep^4|Z|^4 \le C_1 \delta_1^2 \ep^2|Z|^2 = C_1 \delta_1^2 \ep^2(K^2 + q^2). 
$$
Of course there exists $C_2$ such that
$$
K^2 + q^2 \le C_2(1+2 \alpha_w K^2)
$$
Thus 
$$
C_1\ep^4|Z|^4 \le  C_2C_1 \delta^2_1 \ep^2(1+2\alpha_w K^2 ). 
$$
Then take $\delta_1 = \ds{1 \over 2 \sqrt{C_1 C_2}}$ so that
\be\label{jet 2}
C_1 \ep^4|Z|^4 \le {1 \over 4} \ep^2(1+2\alpha_w K^2 ). 
\ee

Putting \eqref{reverse}, \eqref{jet 1} and \eqref{jet 2} together gives, for all $ \ep \in (0,1)$,
\be\label{den est in}
|c_w^2 + \ep^2 - \tzeta(\ep Z)| \ge  {1 \over 4} \ep^2(1+2\alpha_w K^2 ) \ge {1 \over 4} \ep^2
\ee
so long as \be\label{inside}
|\Im (Z) |\le q_{51}:= \min\left({1 \over \sqrt{2 \alpha_w}},\tau_0\right) \mand |Z|\le {\delta_1 \over \ep}.
\ee

And so, if $Z$ meets \eqref{inside} then \eqref{den est in} and the first estimate in \eqref{taylor zeta jones} give
\be\label{smoothing bound part 1}
|\tvarpi^\ep(Z)| =  {\left \vert \ep^2 \tzeta(\ep Z) \right \vert\over \left |c_w^2 + \ep^2 - \tzeta(\ep Z)\right| } \le {4C_1 \over 1+2 \alpha_w K^2} \le {C \over 1+K^2} .
\ee

Next look at
$$
\tilde{\rho}^\ep(Z):= \tvarpi^\ep(Z) - \tvarpi^0(Z) =  -{\ep^2\tzeta(\ep Z)  \over c_w^2 + \ep^2 - \tzeta(\ep Z)}+ {c_w^2 \over 1+\alpha_w Z^2}
$$
Adding zero and the triangle inequality gives:
$$
\left \vert\tilde{\rho}^\ep(Z) \right \vert \le \left \vert {\tzeta(\ep Z) - c_w^2 \over 1+ \alpha_w^2Z^2} \right\vert+ \left \vert \tzeta(\ep Z) \right\vert \left \vert {\ep^2 \over c_w^2 + \ep^2 - \tzeta(\ep Z)} - {1 \over 1+\alpha_w Z^2} \right \vert=:I+II.
$$

The second estimate in \eqref{taylor zeta jones} gives
$$
I \le {C_1 \ep^2 |Z|^2 \over |1+ \alpha_w^2 Z^2|}.
$$
Then if we assume \eqref{inside} and apply \eqref{jet 1} we have:
$$
I \le C\ep^2 {K^2+q^2 \over 1+ 2 \alpha_w K^2} \le C \ep^2.
$$

Next, combining fractions gives:
$$
II=  { \left \vert \tzeta(\ep Z) \right\vert \left \vert \tzeta(\ep Z) - c_w^2 + \ep^2 \alpha_w Z^2 \right \vert\over \left \vert c_w^2 + \ep^2 - \tzeta(\ep Z) \right \vert \left \vert 1+\alpha_w Z^2 \right \vert} 
$$
Using the first and third inequalities in \eqref{taylor zeta jones} gives:
$$
II \le  {  C\ep^4 |Z|^4\over \left \vert c_w^2 + \ep^2 - \tzeta(\ep Z) \right \vert \left \vert 1+\alpha_w Z^2 \right \vert} 
$$
Then if we assume \eqref{inside} and apply \eqref{jet 1} and \eqref{den est in} we have:
$$
II \le  {  C\ep^2 |Z|^4\over \left( 1+2 \alpha_w K^2 \right )^2 } 
$$
Then we see
$$
II \le  {  C\ep^2 (K^2+q^2)^2\over \left( 1+2 \alpha_w K^2 \right )^2 } \le C \ep^2.
$$
Therefore, if \eqref{inside} is met, we have
\be\label{error bound part 1}
\left \vert  \tvarpi^\ep(Z) - \tvarpi^0(Z)\right \vert \le C\ep^2.
\ee


{\bf Estimates far from $Z=0$:} We saw in \eqref{super sonic} that $c_w^2 k^2 - \tlambda_-(k) \ge 0$ with equality only at $k =0$.
This implies that
 \be\label{zeta on R}\tzeta(k) < c_w^2\ee for all $k \ne 0$ and $k \in \R$.
%
%
Take $\delta_1$ as above and put $\tzeta_1:=\sup_{|k|\ge \delta_1/2} \tzeta(k)$. Because we have \eqref{zeta on R}, we know that
$$
c_w^2 - \tzeta_1 =: \delta_3 >0.
$$
It should be obvious that $\delta_3$ does not depend at all on $\ep$.

Suppose that $Z = K+iq$ with $|K| \ge \delta_1/2\ep$. Then the reverse triangle inequality gives:
\be\label{rev2}
\left \vert c_w^2 + \ep^2 - \tzeta(\ep Z)\right \vert \ge 
\left \vert   
\left \vert c_w^2 + \ep^2 - \tzeta(\ep K)\right \vert - \left \vert \tzeta(\ep K) - \tzeta(\ep K + i \ep q)\right \vert 
\right\vert 
\ee
Since $\ep|K|\ge\delta_1/2$ we have
\be\label{jet 3}
\left \vert c_w^2 + \ep^2 - \tzeta(\ep K)\right \vert  \ge \delta_3.
\ee
Then we use \eqref{zeta lip} to see that
$$
\left \vert \tzeta(\ep K) - \tzeta(\ep K + i \ep q)\right \vert \le C_1 \ep |q|.
$$
Thus if we restrict $|q| \le q_{52}:=\min(\delta_3/2C_1,\tau_0)$ we have
\be\label{jet 4}
\left \vert \tzeta(\ep K) - \tzeta(\ep K + i \ep q)\right \vert \le\delta_3/2
\ee
for all $\ep \in (0,1)$.

Thus if we have
\be\label{outside}
\Im(Z) \le q_{52} \mand \Re(Z) \ge \delta_1/2\ep
\ee
then \eqref{rev2}, \eqref{jet 3} and \eqref{jet 4} give:
\be\label{den est 2}
\left \vert c_w^2 + \ep^2 - \tzeta(\ep Z)\right \vert \ge {1 \over 2} \delta_3.
\ee

This gives 
$$
\left \vert \tvarpi^\ep(Z) \right \vert \le C\ep^2 \left \vert \tzeta(\ep Z) \right \vert  \le C{\left\vert \tlambda_-(\ep Z)\right \vert\over |Z|^2}.
$$
when we have \eqref{outside}. The uniform bound on $\tlambda_-$ converts this to
$$
\left \vert \tvarpi^\ep(Z) \right \vert \le C\ep^2 \left \vert \tzeta(\ep Z) \right \vert  \le {C\over |Z|^2}.
$$
But since we have \eqref{outside}, clearly
$$
|Z|^2 \ge K^2 + q^2 \ge {1 \over 2} K^2 + {\delta_1^2 \over 4\ep^2}
$$
This implies
\be\label{ep out}
\left \vert \tvarpi^\ep(Z) \right \vert \le {C \ep^2  \over 1+ \ep^2 K^2} \le {C \over 1+ K^2}
\ee
Along the same lines we can prove
\be\label{zero out}
\left \vert \tvarpi^0(Z) \right \vert = {c_w^2 \over |1 + \alpha_w Z^2|} \le {C \ep^2  \over 1+ \ep^2 K^2} 
\ee
provided we have \eqref{outside}.

{\bf Overall estimates:}
Let $q_{5}:=\min(q_{51},q_{52}).$ If $\Im(Z) \le q_5$ the observe that $Z$ satisfies either \eqref{inside} or \eqref{outside}.
Thus putting \eqref{smoothing bound part 1} together with \eqref{ep out} yields
$$
\left \vert \tvarpi^\ep(Z) \right \vert \le {C \over 1+ |Z|^2}. 
$$
This estimate holds for all $\ep \in (0,1)$.
This is Part (ii) of the lemma and  it implies Part (i).
Putting \eqref{error bound part 1} with \eqref{ep out}, \eqref{zero out} and the triangle inequality gives
$$
\left \vert \tvarpi^\ep(Z)-\tvarpi^0(Z) \right \vert  \le {C \ep^2  \over 1+ \ep^2 |Z|^2}.
$$
This estimate holds for all $\ep \in (0,1)$.
This implies Part (iii) and we are done.

\end{proof}

\subsection{Multiplier properties in $\R$}\label{mult R props}
This subsection contains the proofs of Lemmas \ref{varpi symbol props}, \ref{xi symbol props} and \ref{xi decomp}.  We begin with two more lemmas to prove Lemma \ref{varpi symbol props}.  

\begin{lemma}\label{lambda- triple prime}
$\lambda_-'''(k) < 0$ for $k \in (0,\pi/2)$.
\end{lemma}

\begin{proof}
We find
$$
\tlambda_-'''(k) = -\frac{16w\sin(k)\cos(k)}{\tilde{\varrho}(k)^5}q_w(\cos^2(k)),
$$
where $q_w$ is the quadratic 
$$
q_w(X) = 4w^2X^2 + (2w^3-4w^2+2w)X + (w^4-w^3-w+1).
$$
The discriminant of $q_w$ is
$$
\Delta(w) := (2w^3-4w^2+2w)-4(4w^2)(w^4-w^3-w+1) = -12w^6+24w^4-12w^2 = -12w^2(w^4-w^2+1),
$$
and when $w > 1$,
$$
w^4-w^2+1 > w^4-2w^2+1 = (w^2-1)^2 > 0,
$$
hence $\Delta(w) < 0$, and so $q_w$ is either strictly positive or strictly negative.  Since $q_w$ has positive leading coefficient $4w^2$, $q_w$ is strictly positive, and because $\sin(k)\cos(k) > 0$ on $(0,\pi/2)$, we conclude $\tlambda_-'''(k) < 0$ on $(0,\pi/2)$.
\end{proof}

\begin{lemma}\label{quadratic bound below}
For all $\delta > 0$ there exists $C_{\text{quad},\delta} > 0$ such that $C_{\text{quad},\delta}k^2 \le c_w^2k^2-\tlambda_-(k)$ for all $k \in [\delta,\infty)$.
\end{lemma}

\begin{proof}
Without loss of generality, suppose $0 < \delta < \pi/2$.  From the proof of Lemma \ref{pi lemma} the function
$$
\tzeta(k) := \begin{cases}
\dfrac{\tlambda_-(k)}{k^2}, &k\ne 0 \\
\\
c_w^2, &k= 0 \\
\end{cases}
$$
is bounded, analytic, and nonnegative on $\R$, and it is an easy computation to see that $\tzeta$ is Lipschitz as well.  Next,
$$
\tzeta'(k) = \frac{k^2\tlambda_-'(k)-2k\tlambda_-(k)}{k^2}, k \ne 0.
$$
We will show $k\tlambda_-'(k)-2\tlambda_-(k) < 0$ for all $k \in (0,\pi/2)$, which implies $\tzeta'(k) < 0$ on $(0,\pi/2)$ and therefore that $\tzeta$ is decreasing there.  First, Taylor's theorem gives
$$
\tlambda_-(k) = \tlambda_-(0) + \tlambda_-'(0)k + k^2\int_0^1 (1-t)\tlambda_-''(sk) \ ds = k^2\int_0^1(1-t)\tlambda_-''(sk) \ ds.
$$
Then differentiating under the integral, we find
$$
\tlambda_-'(k) = 2k\int_0^1(1-s)\tlambda_-''(sk) \ ds + k^2\int_0^1s(1-s)\tlambda_-'''(sk) \ ds,
$$
hence
$$
k\tlambda_-'(k)-2\tlambda_-(k) = k^3\int_0^1s(1-s)\tlambda_-'''(sk) \ ds < 0
$$
by Lemma \ref{lambda- triple prime}.

So, $\tzeta$ is decreasing on $(0,\pi/2)$, and therefore
$$
c_w^2 = \tzeta(0) > \tzeta(\delta) > \tzeta(k)
$$
for $k \in (\delta,\pi/2)$, from which
\be\label{tzj 1}
0 < c_w^2-\left(\frac{c_w^2-\tzeta(\delta)}{10}\right)-\tzeta(k), k \in (\delta,\pi/2).
\ee
The inequality \eqref{tzj 1} remains true at $k=\pi/2$ since $\tzeta(\pi/2) = 0$.  When $k > \pi/2$, observe that
$$
\tzeta(k) \le \frac{\tlambda_-(k)}{\left(\dfrac{\pi}{2}\right)^2} = \frac{4\tlambda_-(k)}{\pi^2} \le \frac{8}{\pi^2} = \tzeta\left(\frac{\pi}{2}\right)
$$
by \eqref{lambda bounds}, thus
\begin{multline*}
c_w^2-\left(\frac{c_w^2-\tzeta(\delta)}{10}\right) - \tzeta(k) \ge c_w^2-\left(\frac{c_w^2-\tzeta(\delta)}{10}\right)-\tzeta\left(\frac{\pi}{2}\right) > \frac{9c_w^2}{10}-\frac{\tzeta(\delta)}{10}-\tzeta\left(\frac{\pi}{2}\right) \\
\\
> \frac{9c_w^2}{10}- \left(1+\frac{1}{10}\right)\frac{8}{\pi^2} = \frac{9}{10}\left(\frac{2w}{w+1}\right) - \frac{11}{10}\left(\frac{8}{\pi^2}\right) > 0
\end{multline*}
since $w>1$. So, we take $C_{\text{quad},\delta} = (c_w^2-\tzeta(\delta))/10$.
\end{proof}

\begin{proof} (of Lemma \ref{varpi symbol props})
We begin with some comments on our choice of $\ep_{12}$.  From \eqref{size of kc} in Part (vi) of Lemma \ref{lambda lemma}, we have
\be\label{K-ep bounds}
m_*(w) := \sqrt{\frac{2w}{c_w^2+1}} \le \ep K_{\ep} \le \frac{\sqrt{2+2w}}{c_w} =: m^*(w), \ 0 < \ep < 1. 
\ee
By taking $\ep_* = \ep_*(w)$ close to 0, we will have
\be\label{ep min}
K_{\ep} \ge 2 \mand m_*(w) - \ep_* > 0. 
\ee
Consequently,
\be\label{ep aux}
K_{\ep} + t \ge 1 \mand \ep K_{\ep} + \ep{t} \ge m_*(w) - \ep_* \text{ for } |t| \le 1, 0 < \ep < \ep_*. 
\ee
With $\ep_4$ as in Lemma \ref{pi lemma}, set $\ep_{11} = \min\{\ep_*,\ep_4\}$.  Then Lemma \ref{pi lemma} gives $C > 0$ such that 
$$
\sup_{0 < \ep < \ep_{11}} |\tvarpi^{\ep}(K)| \le \frac{C}{1+K^2}, \ K \in \R.
$$
Then for each $k \in \Z$, 
$$
\sup_{0 < \ep < \ep_{11}} |\tvarpi^{\ep,K_{\ep}+t}(k)| = \sup_{0 < \ep < \ep_0} |\tvarpi^{\ep}((K_{\ep}+t)k)| \le \frac{C}{1+((K_{\ep}+t)k)^2}.
$$
Using \eqref{ep aux}, we have
$$
\frac{C}{1+((K_{\ep}+t)k)^2} \le \frac{C}{1+k^2}
$$
for all $k \in \Z$. This proves the first estimate \eqref{varpi bounded} for the multiplier $\tvarpi^{\ep,K_{\ep}+t}$.

We prove the Lipschitz estimate \eqref{varpi lip} only when $k \ge 1$ as when $k=0$ the left side of this inequality is zero, and evenness takes care of $k\le -1$.  Fix $0 < \ep < \ep_{12}, k \ge 1$, and $|t| \le 1$ and abbreviate $K := \ep({K_{\ep}}+t)k$ and $\grave{K} := \ep({K_{\ep}}+\grave{t})k$ to find
\begin{align*}
\varpi^{\ep,{K_{\ep}}+t}(k) - \varpi^{\ep,{K_{\ep}}+\grave{t}}(k) &= \frac{\ep^2\tlambda_-(K)}{(c_w^2+\ep^2)K^2-\tlambda_-(K)} - \frac{\ep^2\tlambda_-(\grave{K})}{(c_w^2+\ep^2)K^2-\tlambda_-(K)} \\
\\
&+ \frac{\ep^2\tlambda_-(\grave{K})}{(c_w^2+\ep^2)K^2-\tlambda_-(K)} - \frac{\ep^2\tlambda_-(\grave{K})}{(c_w^2+\ep^2)\grave{K}^2-\tlambda_-(\grave{K})} \\
\\
&= \frac{\ep^2(\tlambda_-(K)-\tlambda_-(\grave{K}))}{(c_w^2+\ep^2)K^2-\tlambda_-(K)} \\
\\
&+ \frac{\ep^2\tlambda_-(\grave{K})\big(c_w^2+\ep^2)\grave{K}^2-\tlambda_-(\grave{K})\big)-\ep^2\tlambda_-(\grave{K})\big((c_w^2+\ep^2)K^2-\tlambda_-(K)\big)}{\big((c_w^2+\ep^2)K^2-\tlambda_-(K)\big)\big((c_w^2+\ep^2)\grave{K}^2-\tlambda_-(\grave{K})\big)}
\end{align*}

Call the last two terms above $I$ and $II$.  Set $\delta = m_*(w) - \ep_*$ and invoke Lemma \ref{quadratic bound below} to find $C_{\text{quad},\delta} > 0$ such that 
\be\label{use of quad bound}
C_{\text{quad},\delta}K^2 \le c_w^2K^2-\tlambda_-(K) \le (c_w^2+\ep^2)K^2-\tlambda_-(K)
\ee
for all $K \in [\delta,\infty)$. By \eqref{ep aux} we have $\ep(K_{\ep}+t)k \ge \delta$ for all $0 < \ep < 1, |t| \le 1, k \ge 1$.  Then using the additional estimates $\Lip(\tlambda_-) \le 2$ from \eqref{lambda bounds} and ${K_{\ep}}+t \ge 1$, which is \eqref{ep aux}, we estimate I by
$$
|I| \le \left|\frac{\ep^2(\tlambda_-(K)-\tlambda_-(\grave{K}))}{(c_w^2+\ep^2)K^2-\tlambda_-(K)}\right| \le \frac{2\ep^2K-\grave{K}|}{C_{\text{quad},\delta}|K|^2} =\frac{2\ep^3|k||t-\grave{t}|}{C_{\text{quad},\delta}\ep^2({K_{\ep}}+t)^2k^2} \le 2\ep_{11}|t-\grave{t}|.
$$

Next, we rewrite $II$ as
$$
II = \frac{\ep^2(c_w^2+\ep^2)\tlambda_-(\grave{K})(\grave{K}^2-K^2) + \ep^2\tlambda_-(\grave{K})(\tlambda_-(K)-\tlambda_-(\grave{K}))}{\big((c_w^2+\ep^2)K^2-\tlambda_-(K)\big)\big((c_w^2+\ep^2)\grave{K}^2-\tlambda_-(\grave{K})\big)},
$$
hence
$$
|II| \le \frac{\ep^2(c_w^2+\ep^2)|\tlambda_-(\grave{K})(\grave{K}^2-K^2)|}{C_{\text{quad},\delta}^2|K|^2|\grave{K}|^2} + \frac{\ep^2|\tlambda_-(\grave{K})(\tlambda_-(K)-\tlambda_-(\grave{K}))|}{C_{\text{quad},\delta}^2|K|^2|\grave{K}|^2}.
$$
Labeling these two terms as $III$ and $IV$, we find
\begin{align*}
|III| &\le \frac{2\ep^4(c_w^2+\ep^2)k^2|({K_{\ep}}+t)+({K_{\ep}}+\grave{t})||t-\grave{t}|}{C_{\text{quad},\delta}^2\ep^4k^4({K_{\ep}}+t)^2({K_{\ep}}+\grave{t})^2} \\
\\
&\le \frac{2(c_w^2+\ep^2)}{C_{\text{quad},\delta}^2}\left(\frac{|{K_{\ep}}+t|}{({K_{\ep}}+t)^2({K_{\ep}}+\grave{t})^2} + \frac{|{K_{\ep}}+\grave{t}|}{({K_{\ep}}+t)^2({K_{\ep}}+t)^2}\right)|t-\grave{t}| \\
\\
&\le \frac{4(c_w^2+\ep_{11}^2)}{C_{\text{quad},\delta}^2}|t-\grave{t}|
\end{align*}
and, since $\grave{K}^2\tzeta(\grave{K}) = \tlambda_-(\grave{K})$,
$$
|IV| \le \frac{2\ep^2|\grave{K}^2||\tzeta(\grave{K})||K-\grave{K}|}{C_{\text{quad},\delta}^2|K|^2|\grave{K}|^2} \le \frac{2\ep^5c_w^2({K_{\ep}}+\grave{t})^2k^3|t-\grave{t}|}{C_{\text{quad},\delta}^2\ep^4k^4({K_{\ep}}+t)^2({K_{\ep}}+\grave{t})^2} \le \frac{2c_w^2\ep_{11}}{C_{\text{quad},\delta}^2}|t-\grave{t}|.
$$
Together, the estimates on $I$, $III$, and $IV$ give \eqref{varpi lip} for $k \ge 1$, which satisfies our purposes.
\end{proof}

\begin{proof} (of Lemma \ref{xi symbol props}) 
\begin{enumerate}[(i)]
\item Since $\tilde{\Pi}_2(\pm1) = 0$ and $\txi_{\cep}(0) = 2+2w$, we show
\be\label{xi bdd below}
0 < \inf_{\substack{0 < \ep < \ep_{12} \\ |k| \ge 2 \\ |t| \le 1}} |\tilde{\xi}_{\cep}(\ep({K_{\ep}}+t)k)| = :m_{\txi}
\ee
for an appropriate $\ep_{12} >0$ and set
\be\label{C xi min}
C_{\txi \min} = \min\{m_{\txi},2+2w\} \mand C_{\txi \max} = C_{\txi \min}^{-1}.
\ee

We begin with some seemingly unrelated calculations which result in our choice of $\ep_{12}$.  Set
$$
f(\gamma,w) := -(1+\gamma)^2\left(\frac{4w^2}{3w+1}\right)+2+2w \mand g(w) := \lim_{\gamma\to1^-} f(\gamma,w) = -\frac{16w^2}{3w+1}+2+2w.
$$
Elementary algebra shows that for $w > 1$, $g(w) < 0$ if and only if $5w^2-4w-1 > 0$, and it is the case that this quadratic is positive on $(1,\infty)$. So, we may find some $\gamma_* = \gamma_*(w) \in (0,1)$ such that $f(\gamma,w) < 0$.  

Let $\ep_{12} = \ep_{12}(\gamma,w) > 0$ be so small that the inequalities \eqref{ep min} from the proof of Part (i) of Lemma \ref{varpi symbol props} above hold with $\ep_* = \ep_{12}$.  Furthermore, require $\ep_{12}$ to satisfy
\be\label{ep gamma}
2m_*(w) - 2\ep_* \ge (1+\gamma_*)m_*(w). 
\ee
and suppose that $\ep_{12}$ is small enough that $|c_{\ep}-c_w| < \delta$ for $0 < \ep < \ep_{12}$, where $\delta$ is from Lemma \ref{xi lemma}.

Then for $k \ge 2, |t| \le 1, 0 < \ep < \ep_*$, we have
\begin{align*}
\ep(K_{\ep}+t)k &= \ep K_{\ep}(k-1) + \ep{t}k + \ep K_{\ep} \\
&\ge m_*(w)(k-1)-\ep_*k + m_*(w) \\
&= (m_*(w)-\ep_*)k \\
&\ge 2(m_*(w)-\ep_*)k \\
&\ge (1+\gamma_*)m_*(w) \text{ by \eqref{ep gamma}}.
\end{align*}

Now, observe that $\txi_c$ is increasing on $(0,\infty)$ whenever $c \ge c_w$ as from \eqref{lambda derivative bounds}, when $k > 0$ we have 
\be\label{xi aux 1}
\txi_c'(k) = 2c^2k- \tlambda_+'(k) \ge 2c_w^2k- \tlambda_+'(k) \ge 2c_w^2|k|-|\tlambda_+'(k)| \ge 0.
\ee
Then the work above shows
\begin{align*}
\txi_{\cep}(\ep({K_{\ep}}+t)k) &\le \txi_{c_w}(\ep({K_{\ep}}+t)k) \\
&\le \txi_{c_w}((1+\gamma_*)m_*(w)) \\
&= -c_w^2[(1+\gamma_*)m_*(w)]^2 + \tilde{\lambda}_+((1+\gamma_*)m_*(w)) \\
&\le  -c_w^2[(1+\gamma_*)m_*(w)]^2 +2+2w \\
&\le f(\gamma_*,w) < 0.
\end{align*}
Thus \eqref{xi bdd below} follows with $0 < |f(\gamma_*,w)| \le m_{\txi}$.

\item To prove the Lipschitz estimate \eqref{xi lip}, note that it already holds when $k=0$, so for $k \ge 2$ set $K = \ep({K_{\ep}}+t)k, \grave{K} = \ep({K_{\ep}}+\grave{t})k$ and compute
\begin{align*}
\left|\frac{1}{\txi_{\cep}(K)} - \frac{1}{\txi_{\cep}(\grave{K})}\right| &= \left|\frac{(c_w^2+\ep^2)\grave{K}^2-\tlambda_+(\grave{K})-\big((c_w^2+\ep^2)K^2-\tlambda_+(K)\big)}{\txi_{\cep}(K)\txi_{\cep}(\grave{K})}\right| \\
\\
&\le \frac{(c_w^2+\ep^2)|\grave{K}^2-K^2|}{\left|\txi_{\cep}(K)\txi_{\cep}(\grave{K})\right|} + \frac{|\tlambda_+(K)-\tlambda_+(\grave{K})|}{\left|\txi_{\cep}(K)\txi_{\cep}(\grave{K})\right|}
\end{align*}
Call the two terms above $I$ and $II$.  By choice of $\ep_{12}$ above, Lemma \ref{xi lemma} furnishes $C,R > 0$ such that when $k \ge R$ and $0 < \ep < 2$, then
\be\label{xi intermediate bound}
\frac{1}{|\txi_{\cep}(k)|} \le \frac{1}{Ck^2} 
\ee
Since $m_*(w) - \ep_{12} > 0$ by \eqref{ep min}, we may set  $R_* = R/(m_*(w)-\ep_{12})$ to see that when $k > R_*$, then 
$$
R < (m_*(w)-\ep_{12})k \le \ep({K_{\ep}}+t)k.
$$

{\bf{Estimates when $2 \le k \le R_*$.}} With $C_{\txi \min}$ as in \eqref{C xi min}, we bound
\begin{align*}
|I| \le \frac{\ep_{12}^2(c_w^2+\ep^2)|\grave{K}+K|\ep|k||t-\grave{t}|}{C_{\txi \min}^2} &= \frac{\ep_{12}^2(c_w^2+\ep^2)\ep{k}^2|2\ep{K_{\ep}}+\ep{t}+\ep\grave{t}||t-\grave{t}|}{C_{\txi \min}^2} \\
\\
&\le \left(\frac{2\ep_{12}^3R_*^2(c_w^2+\ep_{12}^2)(\beta(w)+\ep_{12})}{C_{\txi \min}^2}\right)|t-\grave{t}|
\end{align*}
and using $\Lip(\tlambda_+) \le 2$,
$$
|II| \le \frac{2\ep_{12}^2|K-\grave{K}|}{C_{\txi \min}^2} = \frac{2\ep|k|||t-\grave{t}|}{C_{\txi \min}^2} \le \left(\frac{2\ep_{12}^3R_*}{C_{\txi \min}^2}\right)|t-\grave{t}|.
$$

{\bf{Estimates when $k > R_*$.}} Using \eqref{xi intermediate bound} we have
$$
|I| \le \frac{\ep^2(c_w^2+\ep^2)|K+\grave{K}||K-\grave{K}|}{C^2|K|^2|\grave{K}|^2} = \frac{\ep^4(c_w^2+\ep^2)k^2|({K_{\ep}}+t) + ({K_{\ep}}+\grave{t})||t-\grave{t}|}{C^2\ep^4k^4({K_{\ep}}+t)^2(K_{\ep}+\grave{t})^2} \le \left(\frac{2(c_w^2+\ep_{12}^2)}{C^2R_*^2}\right)|t-\grave{t}|,
$$
since by \eqref{ep min}
$$
\frac{|({K_{\ep}}+t)+({K_{\ep}}+\grave{t})|}{({K_{\ep}}+t)^2(K_{\ep}+\grave{t})^2} \le 2.
$$

Next, adding and subtracting $K\tzeta(\grave{K})$ and using the triangle inequality in the numerator of {II} and \eqref{xi intermediate bound} in the resulting denominators gives
\begin{align*}
|II| &\le \frac{\ep^2|K^2\tzeta(K)-\grave{K}^2\tzeta(\grave{K})|}{C^2|K|^2|\grave{K}|^2} \\
\\
&\le  \frac{\ep^2K^2|\tzeta(K)-\tzeta(\grave{K})|}{C^2|{K}|^2|\grave{K}|^2} + \frac{\ep^2|\tzeta(\grave{K})||K^2-\grave{K}^2|}{C^2|K|^2|\grave{K}|^2} \\
\\
&\le \frac{\ep^2\Lip(\tzeta)|K-\grave{K}|}{C^2|\grave{K}|^2} + \frac{\ep^2c_w^2|K+\grave{K}||K-\grave{K}|}{C^2|K|^2|\grave{K}|^2} \\
\\
&\le \frac{\ep^3|k|\Lip(\tzeta)|t-\grave{t}|}{C^2\ep^2|{K_{\ep}}+t|^2|k|^2} + \frac{c_w^2\ep^4k^2|({K_{\ep}}+t)+({K_{\ep}}+\grave{t})||t-\grave{t}|}{C^2\ep^4k^4({K_{\ep}}+t)^2(K_{\ep}+\grave{t})^2} \\
\\
&\le \left(\frac{\Lip(\tzeta)\ep_{12}+2c_w^2}{C^2}\right)|t-\grave{t}|
\end{align*}
by reasoning similar to that above. 
\end{enumerate}
\end{proof} 

\begin{proof} (of Lemma \ref{xi decomp}) 
Taylor's theorem and some straightforward algebra imply
$$
\txi_{\cep}(\ep{K_{\ep}}+\tau)-\txi_{\cep}'(\ep{K_{\ep}})\tau = \tau^2\left(\int_0^1(1-s)\tlambda_+''(\ep{K_{\ep}}+s\tau) \ ds - (c_w^2+\ep^2)\right)
$$
Set
$$
R_{\ep}(\tau) := \int_0^1(1-s)\tlambda_+''(\ep{K_{\ep}}+s\tau) \ ds - (c_w^2+\ep^2).
$$
The estimates \eqref{R bounds} for $R_{\ep}$ follow directly from properties of $\tlambda_-'''$ so long as $\ep$ is bounded.
\end{proof}

\section{Conclusions/Questions/Future Directions}\label{CR}

We have considered a diatomic lattice where the only spatially variable material property was the particles' masses. We also 
took a very basic form for the spring force $F_s(r) = -k_s r - b_s r^2$. 
We are confident that the results here can be extended to much more general situations, be it including more complicated and spatially heterogeneous springs or assuming that the number of ``species" of masses and springs is greater than two, {\it i.e.} a polymer lattice.

%

As mentioned above, we do not yet have lower bounds for the size of the periodic part's amplitude, though the methods used in \cite{sun} provide a roadmap 
for establishing them.
Again, we expect that those methods will show that that the periodic part is genuinely non-zero, at least for almost all $\ep$.
But if the periodic part is non-zero then necessarily the total mechanical energy of the nanopteron solution will be {\it infinite}. 
This stands in stark contrast to the solitary waves for monatomic FPUT, which are not only finite energy but also constrained minimizers of an appropriate related energy functional \cite{friesecke-wattis}.

Nonetheless we contend that the nanopteron solutions we construct are essential for understanding the long time behavior of small amplitude long wave solutions for diatomic FPUT. 
It is known that the monatomic solitary waves are asymptotically stable \cite{friesecke-pego2}-\cite{friesecke-pego4} and, moreover, there are stable multisoliton-like solutions \cite{hoffman-wayne1} \cite{mizumachi}. It is widely held that all small, long wave initial conditions for monatomic FPUT will satisfy that the soliton resolution conjecture, which is to say that they  will converge to a linear superposition of well-separated solitary waves plus a dispersive tail, also called ``radiation."

But if, as we expect, there are no localized traveling long wave solutions in diatomic FPUT then clearly some other asymptotic behavior takes place.
We know, by virtue of the approximation results in \cite{gaison-etal} and \cite{chirilus-bruckner-etal}, that long wave initial data for \eqref{r eqn} remains close to suitably scaled solutions of the KdV equation
for very long times. Specifically for times $t$ up to  $\O(1/\ep^3)$ where $\ep$ is consistent with its meaning here.
And since the soliton resolution conjecture is known to be true ({\it via} integrability) for KdV that means we can expect the solution of diatomic FPUT to resolve, at least temporarily, into a sum of $\sech^2$ like solitary waves. But on time scales beyond this those approximation theorems tell us {\it nothing}. 

There are any number of possibilities for what happens afterwards. One possibility, which we favor, is that there is a very slow ``leak" of energy from the acoustic branch into the optical branch that will eventually erode the solution into nothing but radiation. That is to say, we conjecture the existence of metastable solutions which look for very long times like localized solitary waves but eventually converge to zero.
Another possibility is that there is a heretofore unknown finite energy coherent structure  with a more complicated temporal behavior to which the solution converges---for instance something akin to the traveling breathers that exist in  modified KdV 
or a localized quasi-periodic solution.
Yet another scenario is that there are a discrete set of choices for $\ep$ where the ripple vanishes.
These waves would then, in a rough sense,  quantize the possible behavior as $t$ goes to infinity.

These questions are likely very difficult to settle. Note that the time scales are so long that only very careful numerics performed on very large domains will shed any light. 
And so, as we stated in the introduction, we feel our work here raises many interesting questions.

\bibliographystyle{alpha}
\bibliography{massdimer-longwavelimit}{}

\end{document}